\documentclass[pdflatex,sn-mathphys-num]{sn-jnl}
\usepackage{}

\RequirePackage{amsthm,amsmath,amsfonts,amssymb}
\RequirePackage[numbers,sort&compress]{natbib}
\usepackage{amsfonts}
\usepackage{mathrsfs}
\usepackage[greek,english]{babel}
\usepackage{amsmath,amsfonts,amsthm,amssymb,amscd,color,xcolor,mathrsfs,verbatim,microtype}
\usepackage{graphicx,eurosym}
\usepackage{amssymb,amsmath,comment}
\usepackage{amsfonts,mathrsfs}
\usepackage{color}
\usepackage{bbm}
\usepackage{bm}
\usepackage{hyperref}
\usepackage{stmaryrd} 

\usepackage[applemac]{inputenc}

\usepackage[cyr]{aeguill}
\usepackage{graphicx}%
\usepackage{multirow}%
\usepackage{amsmath,amssymb,amsfonts}%
\usepackage{amsthm}%
\usepackage{mathrsfs}%
\usepackage[title]{appendix}%
\usepackage{xcolor}%
\usepackage{textcomp}%
\usepackage{manyfoot}%
\usepackage{booktabs}%
\usepackage{algorithm}%
\usepackage{algorithmicx}%
\usepackage{algpseudocode}%
\usepackage{listings}%
\usepackage{graphicx,amscd,mathrsfs,wrapfig,mathrsfs,lipsum}
\usepackage{eufrak}
\usepackage{float}
\usepackage{tikz}
\usepackage{multicol}
\usepackage{caption}
\usetikzlibrary{arrows}
\usepackage{capt-of}

\colorlet{darkblue}{blue!50!black}

\newcommand{\vertiii}[1]{{\left\vert\kern-0.25ex\left\vert\kern-0.25ex\left\vert #1
		\right\vert\kern-0.25ex\right\vert\kern-0.25ex\right\vert}}

\usepackage{marginnote}

\newcommand{\e}{\varepsilon}

\def\cB{{\mathcal B}}

\newcommand{\R}{{\mathbb R}}

\newcommand{\pP}{{\mathbb P}}
\newcommand{\I}{{\mathbb I}}

\def\cN{{\mathcal N}}

\newcommand{\ty}{\infty}

\newcommand{\nn}{{\mathfrak n}}

\newcommand{\mm}{{\mathfrak m}}

\newcommand{\mmm}{{\mathfrak{m}}}

\newcommand{\mT}{{\mathbb T}}

\newcommand{\bae}{\begin{equation}\begin{aligned}}
		\newcommand{\eae}{\end{aligned}\end{equation}}
\newcommand{\baee}{\begin{equation*}\begin{aligned}}
		\newcommand{\eaee}{\end{aligned}\end{equation*}}

\newcommand{\aA}{{\cal A}}

\newcommand{\FF}{{\cal F}}

\def\mE{{\mathbb E}}
\def\dif{{\mathord{{\rm d}}}}
\def\cF{{\mathcal F}}
\def\cZ{{\mathcal Z}}
\def\mP{{\mathbb P}}
\def\mN{{\mathbb N}}
\def\mR{{\mathbb R}}
\def\mZ{{\mathbb Z}}
\def\cD{{\mathcal D}}
\def\cM{{\mathcal M}}
\def\cA{{\mathcal A}}
\def\cS{{\mathcal S}}

\def\cK{{\mathcal K}}
\def\cR{{\mathcal R}}

\def\eps{\varepsilon}
\def\cL{{\mathcal L}}
\def\mI{{\mathbb I}}
\def\cC{{\mathcal C}}

\newcommand{\lag}{\langle}
\newcommand{\rag}{\rangle}

\newcommand{\dd}{{\textup d}}

\theoremstyle{plain}

\newtheorem*{lemma*}{Lemma}
\newtheorem{theorem}{Theorem}[section]
\newtheorem{lemma}[theorem]{Lemma}
\newtheorem{proposition}[theorem]{Proposition}
\newtheorem{corollary}[theorem]{Corollary}
\theoremstyle{definition}
\newtheorem{definition}[theorem]{Definition}
\newtheorem{condition}[theorem]{Condition}

\newtheorem{remark}[theorem]{Remark}
\newtheorem{example}[theorem]{Example}

\numberwithin{equation}{section}

\theoremstyle{thmstyleone}%

\theoremstyle{thmstyletwo}%

\theoremstyle{thmstylethree}%

\begin{document}

\title[Article Title]{Ergodicity of the viscous scalar  conservation laws with  a degenerate noise}


\author[1]{\fnm{Xuhui} \sur{Peng}}\email{xhpeng@hunnu.edu.cn}

\author[2,3]{\fnm{Houqi} \sur{Su}}\email{marksu@amss.ac.cn}

\affil[1]{\orgdiv{MOE-LCSM, School of Mathematics and Statistics}, \orgname{Hunan Normal University}, \orgaddress{\street{Yuelushan road NO. 36}, \city{Changsha}, \postcode{410081}, \state{Hunan}, \country{P.R. China}}}

\affil[2]{\orgdiv{School of Mathematical Sciences}, \orgname{Capital Normal University}, \orgaddress{\street{Baiduizijia NO. 23}, \city{Beijing}, \postcode{100048}, \state{Beijing}, \country{P.R. China}}}

\affil[3]{\orgdiv{Academy of Mathematics and Systems Science}, \orgname{Chinese Academy of Sciences}, \orgaddress{\street{Zhongguancun road NO. 55}, \city{Beijing}, \postcode{100190}, \state{Beijing}, \country{P.R. China}}}


\abstract{	This paper establishes the ergodicity in  $H^\nn,\nn=\lfloor\frac{d}{2}+1\rfloor$ of the viscous scalar conservation laws on torus $\mT^d$ with  general polynomial flux  and   a degenerate noise.  The
	noise could appear in as few as several   directions. We introduce a localized framework that restricts attention to trajectories with controlled energy growth, circumventing the limitations of traditional contraction-based approaches. This  localized   method  allows for a  demonstration of  e-property and consequently proves the uniqueness of   invariant measure  under a H{\"o}rmander-type condition. Furthermore,  we characterize   the absolute continuity of the invariant measure's projections onto  any finite-dimensional subspaces under requirement on a new algebraically non-degenerate condition for the flux.}

\keywords{Stochastic conservation laws, Invariant measure, H{\"o}rmander condition}



\maketitle

	\section{Introduction and Main results}
\subsection{Introduction}
In this paper, we investigate the long  time behaviour of stochastic viscous scalar conservation laws (SVSCL) with a degenerate noise:
%
\begin{eqnarray}\label{1-1 pre}
	\left\{
	\begin{split}
		&\dif u_t+\operatorname{div} A(u_t)\dif t=\nu\Delta u_t\dif t+\dif \eta_t,\quad x\in \mathbb T^d,\,u_t(x)\in \R,
		\\ & u_t\big|_{t=0}=u_0,
	\end{split}
	\right.
\end{eqnarray}
where the viscosity coefficient $\nu>0$, and $\mT^d=[-\pi,\pi]^d$  denotes the $d$-dimensional torus. The noise $\eta_t$ is a highly degenerate $Q$-Wiener process, affecting only  on several    number of Fourier modes.
The flux  $A=(A_1,\cdots,A_d): \mathbb R\rightarrow\mathbb R^d$ is  defined such that each component $A_i(u),i=1,\cdots,d$ is   a polynomial in $u.$

Viscous conservation laws are ubiquitous in science and engineering, arising in models of fluid dynamics, traffic flow, chemical reactions, and numerous other phenomena. Understanding their long-time behavior is crucial for determining the stability and equilibrium states of the associated systems. In the context of stochastic partial differential equations (SPDEs), a fundamental aspect of long-time behavior is ergodicity. Ergodicity plays a vital role in modeling physical systems subject to noise, as it enables the prediction of long-term behavior and the characterization of macroscopic properties from microscopic dynamics.

The transition semigroup $P_t$ of \eqref{1-1 pre} is defined by $P_t \phi(u_0) = \mathbb E\phi(u_t), \forall \phi \in C_b(H^\nn)$, where  $\nn=\lfloor\frac{d}{2}+1\rfloor$.An invariant measure  of  $P_t$  is a probability measure $\mu$ on $H^\nn$  such that $P_t^*\mu = \mu$, where $P_t^*$ denotes the  dual of $P_t$.
{    The primary focus of this paper is the ergodicity of equation \eqref{1-1 pre}, specifically the existence and uniqueness of an invariant measure. While the existence of an invariant measure is well-established in our setting (see, e.g., \cite{DPZ96}), the main challenge lies in proving its uniqueness.

}

{
	The  equation  (\ref{1-1 pre}) is closely  related with stochastic Burgers equation.
	When  $d=1$ and  $A(u)=\frac{u^2}{2} $,  the  equation  (\ref{1-1 pre})
	is commonly referred to as the one-dimensional stochastic Burgers equation.
	For the inviscid($\nu=0$)  one dimensional stochastic  Burgers equation on torus $\mT$,
	E et al. \cite{EMKY-00} establish a seminal result, proving the existence, and  uniqueness of   invariant measure.
	Since the flux $A$  in \cite{EMKY-00}  is quadratic, the solution can be expressed via the Lax-Oleinik formula, which plays a key role in the establishment of ergodicity.
	Bakhtin,  Cator and  Khanin \cite{BCK-14}  generalize
	the results in  \cite{EMKY-00} to  the real line.
	For the  viscous
	($\nu>0$)  one dimensional stochastic  Burgers equation,
	Bakhtin and Li \cite{BL-19} establish  an  ergodic result on the real line,   relying on the Feynman-Kac formula and  the assumption  $A(u)=\frac{u^2}{2}$.
	Regarding the ergodicity of the general
	$d$-dimensional stochastic Burgers equation, we refer readers to \cite{IK-03}, \cite{B16}, and related works.
}

%
%
%

When concerning the general flux in scalar conservation laws (\ref{1-1 pre}), the  $L^1$ contraction is widely utilized in the existing  study of the ergodicity of stochastic scalar conservation laws.
For deterministic scalar conservation laws, this powerful $L^1$ contraction of the deterministic nonlinear semigroup is derived from the celebrated Kru\v{z}kov estimate.
$L^1$  contraction  property   ensures all trajectories don't leave each other too far    for  almost every noise realization, which is very useful in  in proving the uniqueness of the invariant measure.
%
%
%

Since the primary focus of this article is on stochastic conservation laws, we begin with a brief overview of recent developments in this field. In the inviscid case ($\nu=0$), under the assumption of a sub-quadratic growth condition on the flux $A$, Debussche and Vovelle \cite{DV15} establish pathwise convergence using a small noise argument to demonstrate ergodicity.
Moreover, they also demonstrate the existence of an invariant measure for sub-cubic fluxes and the uniqueness of the invariant measure for sub-quadratic fluxes. Due to the challenges associated with establishing tightness within the $L^1$ framework, Debussche and Vovelle do not consider sup-cubic fluxes when addressing the existence of an invariant measure. By constructing a weighted $L^1$ contraction, Dong, Zhang (Rangrang), and Zhang (Tusheng) \cite{DZZ23} establish ergodicity as well as polynomial mixing.
{
The flux  in their paper should   be strictly increasing odd functions and also  satisfy a  non-degenerate condition
\begin{align}
	\label{pp23-1}
	\sum_{j=1}^d|A_j(u)-A_j(v)|\geq C|u-v|^{1+q_0}, \forall u,v\in \mR,
\end{align} where $C>0,q_0>1$ are constants.
Therefore,
a stronger  $L^1$ contraction   property   is obtained,   please see  \cite[Theorem 4.2]{DZZ23} for details.}
When viscosity is present ($\nu>0$), Boritchev \cite{B13} establishes polynomial ergodicity for strongly convex fluxes with polynomial growth and spatially smooth noise. Martel and Reygner \cite{MR20} relax the convexity condition and the sub-quadratic growth condition on the flux $A(u)$ to allow for any polynomial growth, while also proving the uniqueness of the invariant measure for $P_t$. In both works \cite{B13, MR20}, the analysis of ergodicity is restricted to a one-dimensional spatial domain.
{  We would also like to mention the work of Dirr and Souganidis \cite{DS-05}, in which they thoroughly investigate the invariant measures of stochastic Hamilton-Jacobi equations with additive noise. }
%
%
{In all the aforementioned works, for a general flux $A(u)$ and spatial dimension $d\in \mN$, establishing ergodicity typically requires either a convexity condition on $A(u)$ or the assumption that $A(u)$ is a strictly increasing odd function satisfying (\ref{pp23-1}).
These restrictions on the flux are imposed to ensure that the contraction between solutions originating from different initial data is sufficiently strong to establish ergodicity. As a result, the existing ergodicity results via contraction--based route are {\textit{independent of the number of noise terms.}}
{  In this paper,
we focus on  the viscous case ($\nu>0$).
For a general flux and dimension, we investigate the ergodicity of stochastic scalar conservation laws (\ref{1-1 pre}) under a H\"ormander-type condition. }
In other words, our ergodicity results {\textit{depend on the number of noise terms and the interactions between the flux and the noise.}} Our findings hold for general flux $A(u)$ when $d=1$. {For $d\geq 2$,
We only require a Hormander-type condition, without needing $A(u)$  to be quadratic, convex or a strictly increasing odd function satisfying  (\ref{pp23-1}), as required in \cite{DZZ23}.
}

}




In our problem \eqref{1-1 pre}, the flux is \textit{general polynomial type}, the noise is \textit{highly degenerate} and the variable is in  \textit{high-dimensional} spatial domain.  When the flux and spatial dimension are general, the solutions cannot be expressed using the Lax-Oleinik formula or the Feynman-Kac formula, as in \cite{EMKY-00,BCK-14,BL-19}, due to the absence of a quadratic flux. Furthermore, the contractivity approaches employed in \cite{DS-05,DV15,MR20,DZZ23} encounter fundamental challenges in this setting. This is a significant reason for introducing viscosity, as the $L^2$ space becomes the most natural framework for studying ergodicity within a Markovian setting. However, pathwise contraction does not hold in the $L^2$  topology, or at the very least, such contraction in $L^2$ depends on the viscosity, as demonstrated in \cite{Mat99}. On the other hand, the independence of ergodicity from viscosity is crucial. It characterizes the external forces that ensure ergodicity across a wide range of physical scenarios, including turbulent regimes where the interplay between viscosity, energy injection, and dissipative scales plays a pivotal role. For further discussions on this topic, we refer to \cite{HM-2006,BZ-17,BZPW-19,BP22}.
Therefore,  we consider the ergodicity
of \eqref{1-1 pre}  with only several noises
in the viscous situation. For the equation   \eqref{1-1 pre} with polynomial growth  flux,  to the best of our knowledge, there is no well-posedness result in  $L^2$ space. Therefore,  we consider the equation  \eqref{1-1 pre} in space $H^\nn, \nn=\lfloor\frac{d}{2}+1\rfloor.$

In the general ergodic theory of Markov processes, the analysis often relies on the  \textit{smoothing effect} of the Markov semigroup rather than \textit{contractivity}. For instance, the celebrated Doob-Khasminskii theorem states that the strong Feller property and irreducibility of a Markov semigroup imply the uniqueness of the invariant measure. Intuitively, greater randomness leads to a smoother semigroup, which facilitates the establishment of uniqueness.
However, for the semigroup generated by solutions to SPDEs, it is often impossible to prove the strong Feller property when the noise is degenerate. Hairer and Mattingly \cite{HM-2006} introduce the concept of the asymptotic strong Feller property and utilize it to establish exponential mixing for the 2D Navier-Stokes equations on the torus, provided that the random perturbation is an additive Gaussian noise involving only a finite number of Fourier modes.

Unfortunately, the classical approach outlined in the preceding paragraph also faces significant challenges when applied to the viscous scalar conservation laws (\ref{1-1 pre}) with a general flux $A(u)$. Below, we provide a detailed explanation of these difficulties.
In the papers~\cite{HM-2006, HM-2011},
their  ideas of proof of  the  asymptotic strong Feller property is~to approximate the perturbation
$J_{0,t}\xi$ caused by the variation of the initial condition with a variation, $\cA_{0,t}v=\cD^v u_t $, of the noise by an appropriate   process $v$.
For rigorous  definitions of  $J_{0,t}\xi$ and  $\cD^v u_t$,
please see (\ref{10-1}) and  (\ref{p17-1}) below, respectively.
Denote by $\rho_t$ the  residual error between~$J_{0,t}\xi$    and~$\cA_{0,t}v$:
\begin{align*}
\rho_t=J_{0,t}\xi-\cA_{0,t}v.
\end{align*}
%
%
{ By the formula of integration by parts
in Malliavin calculus,} it holds that
\bae\label{IBP}
D_\xi  P_t \varphi(u_0) & =\mathbf{E}_{u_0}\big((D  \varphi)\left(u_t\right)J_{0,t}\xi  \big)  =\mathbf{E}_{u_0}\big((D  \varphi)\left(u_t\right) (\cD^v u_t+ \rho_t)\big) \\
& =\mathbf{E}_{u_0}\Big(\varphi\left(u_t\right) \int_0^t v(s) d W(s)\Big)+\mathbf{E}_{u_0}\big((D \varphi)\left(u_t\right) \rho_t\big) \\
& \leq\|\varphi\|_{\infty} \Big(\mathbf{E}_{u_0}\big|\int_0^t v(s) d W(s)\big|^2\Big)^{1/2}+\|D  \varphi\|_{\infty} \mathbf{E}_{u_0}\left\|\rho_t\right\|.
\eae
In the above,   the integral $\int_0^t v(s) d W(s)$ is interpreted as
the Skrohod integral in Malliavin calculus,  $\|\cdot \|$ denotes the $L^2$ norm,
$D_\xi f$ is the    Fr\'echet derivative of $f$  in the direction of $\xi$,
and $\xi\in H:=\{ u: \int u(x)  \dif x = 0\}.$
Hairer and   Mattingly \cite{HM-2006, HM-2011}  choose suitable direction $v$ and prove that
\begin{eqnarray}
\label{p17-2}
\Big(\mathbf{E}_{u_0}\big|\int_0^t v(s) d W(s)\big|^2\Big)^{1/2}\leq C(\|u_0\|),~  \mE \|\rho_t\|\leq C(\| u_0\| )e^{-\gamma t},\quad  \forall t\geq 0,
\end{eqnarray}
where $\gamma$ is a positive constant and $C$ is a local bounded function on $[0,\infty).$
(\ref{IBP}) and  (\ref{p17-2}) imply
the  following gradient estimate
\begin{eqnarray}
\label{p17-3}
\left\|D  P_{t} \varphi(u_0)\right\| \leq C(\|u_0 \|)\left(\|\varphi\|_{\infty}+e^{-\gamma t} \|D \varphi\|_{\infty}\right)
\end{eqnarray}
which is called   asymptotic strong Feller property in \cite{HM-2006, HM-2011}.
In the  above  processes, it naturally requires some integrable property of random variables $J_{s,t}\xi,J_{s,t}^{(2)}(\phi,\psi),\xi,\phi,\psi\in H, $
where  $J^{(2)}_{s,t}(\phi,\psi)$ is  the second derivative of $u_t$ with respect to initial value  $u_0$ in the directions of $\phi$ and $\psi$(see (\ref{0927-1}) for more details).
However, since we consider a general flux $A(u)$, establishing the integrability of the random variables $J_{s,t}\xi,J_{s,t}^{(2)}(\phi,\psi)$ poses a significant challenge.
In this scenario, \cite[Assumption B.3]{HM-2011} may not hold, potentially resulting in the non-integrability of $\rho_t$ rather than the desired vanishing moments as that in (\ref{p17-2}).
In particular, achieving a result analogous to (\ref{p17-3}) is currently unattainable in our setting.
In general, unbounded variations ( i.e. the non-integrability of $J_{s,t}\xi$ and $J_{s,t}^{(2)}(\phi,\psi)$ is closely related to the instability properties of the linearized problem, which can often provide crucial insights into the long-term behavior of the full nonlinear system. Specifically, the presence of instability in the linearized dynamics can lead to a bifurcation, leading to the emergence of distinct invariant measures for the nonlinear system, see \cite{HZ-2021} for an interesting example.


In this paper, we develop a localized method to address the challenges outlined in the preceding paragraph. Since the choice of the direction $v$  in (\ref{IBP}) will depend on
the Malliavin covariance matrix  $\cM_{0,t}$ of $u_t$,
it requires  some  invertibility on  $\cM_{0,t}$.
This is where the H{\"o}rmander's condition becomes essential: the Lie algebra generated by nonlinearity term  and the noises must be dense. Analysing the invertiblity of the Malliavin matrix $\cM_{0,t}$, we also characterize the absolute continuity of the unique invariant measure under a new algebraically non-degenerate condition \eqref{p0209-5} for the flux $A$.
{ The usual non-degenerate condition for the flux $A$ is:
\bae\label{nond}
\sup _{\alpha \in \mathbb{R}, \beta \in \mathbb{S}^{d-1}}\operatorname{measure}\{\xi \in \mathbb{R} ;|\alpha+\langle\beta, A^\prime(\xi)\rangle|<\varepsilon\}\leq C \epsilon^b,
\eae
where $\mathbb{S}^{d-1}$ is the unit sphere in $\mathbb R^d$, $C>0$ and $b\in(0,1].$ 	
The above  condition   is firstly addressed in \cite{DLM91} to study the regularity of the kinectic solution. It is vital for the long-time behavior for both demerministic and stochastic settings since it excludes trivial stationary solution such as $u(\alpha_1 x_1+\cdots+\alpha_d x_d)$, one can refer to \cite{DV09,CP09,DV15,GS17} for more details.
Observe that our conditions depend on the iterations between the noises and   the flux $A$ while the  usual non-degenerate condition only poses restrictions on 	flux $A$. Therefore,
in many  cases,  our condition is strictly weaker than usual non-degenerate condition.
For this, please to  see Remark \ref{rm nond}.
}

Now, we provide a brief overview of the localized method employed in this paper. Our approach is inspired by \cite{PZZ24}, in which Peng, Zhai, and Zhang establish the ergodicity for stochastic 2D Navier-Stokes equations driven by a highly degenerate pure jump L\'{e}vy noise $W_{S_t}$ where
$(W_t)_{t\geq 0}$ is a standard Brownin motion and $S_t$ is a pure jump process.
In that paper,  Peng, Zhai and Zhang   identify the  ``bad part" of the sample space $\Omega$, denoted by $\{\omega\in\Omega:\Theta>M\}$, where $\Theta$ is a random variable depending solely on $(S_t)_{t\geq 0}$.
On the ``good part" $\{\omega\in\Omega:\Theta\leq  M\}$  of the sample space, they derive a property analogous to the asymptotic strong Feller property (\ref{p17-3}).
With regard to    the ``bad part", it   has the property  $\lim_{M\rightarrow\infty}\mathbb{P}(\omega\in\Omega:\Theta>M )=0$.
Using this technique,  they prove that the semigroup generated by the solutions possesses the e-property. The e-property, combined with a form of irreducibility, implies the uniqueness of the invariant probability measure. For further details, see, for example, \cite[Theorem 1]{KSS12} and \cite[Proposition 1.10]{GL15}.
Compared to that in \cite{PZZ24},  there are  many new difficulties appeared in this paper, we only mention some of them here.
\begin{itemize}
\item[(1)] Clearly, their localized method relies heavily on the pure jump process $S_t$ which is absent in our setting.
Roughly speaking,  we also need to define a random variable
$\Theta$    to partition the sample space the sample space $\Omega$ into ``bad" and  ``good" parts.
The key challenge in our localized method lies in defining the random variable $\Theta$.
In this paper, $\Theta$ depends on the energy growth of the solution $(u_t)_{t\geq 0}$ along the Lyapunov type structures and, consequently, on the noise $\eta_t$ and initial value $u_0$.
Also,  our proof is much more complicated  since the derivatives of $\Theta$ with respect to initial value and noise are {\bf{not zero}}.
\item[(2)]
The solutions presented in this paper exhibit weaker integrability compared to those in  \cite{PZZ24}   and \cite{HM-2006}. Specifically, while the solutions in \cite{PZZ24}  possess exponential integrability at certain stopping times and those in \cite{HM-2006}  are exponentially integrable for all
$t\geq 0$, the solutions here lack exponential integrability entirely.
Crucially, the estimates of
$J_{s,t}\xi,J_{s,t}^{(2)}(\phi,\psi),\xi,\phi,\psi\in H $ etc
are ultimately governed by the behavior of the solutions. Thus, a more careful estimation of
$J_{s,t}\xi,J_{s,t}^{(2)}(\phi,\psi)$ etc
and a more careful choice of $\Theta$
are  required  here so that the bad part of  them can be controlled by  $\Theta$ and the localized method works.

\item[(3)]{
The solutions in  \cite{PZZ24}   and \cite{HM-2006} have more regularity than that in this paper.
If the initial value  $u_0\in L^2$, then the solutions in  \cite{PZZ24}   and \cite{HM-2006} will be smooth after any fixed time $t>0.$
This property plays an important role in the estimate
of the minimum eigenvalue of the Malliavin matrix.
However, with regard to stochastic  conservation
laws,  if the initial value  $u_0\in H^\nn := \{ u \in W^{\nn,2}(\mathbb T^d,\mathbb R): \int_{\mathbb{T}^{d}}u(x)  \dif x = 0\}$ where $W^{\nn,2}$ is the standard Sobolev space,
we can only prove that
the solution still stay in $H^\nn$ for all the time $t\geq 0$.
Thus, compared to the classical literature
that use Mallavin calculus to prove  a property analogous to
(\ref{p17-3}),
we  have to assume that  $u_0$ is in a more regular space $H^{\nn+5}$,	
also we  have to adjust the details of techniques  since we only have
estimates  of  the smallest eigenvalue of the Malliavin matrix
when the  initial value  $u_0$ is in $H^{\nn+5}.$}


\end{itemize}
\textbf{ In summary, compared to existing literature, the method employed in this study can handle stochastic partial differential equations with fewer integrability and  less  regularity properties.  }

Finally, we note that there is a substantial body of work on the ergodicity of SPDEs. Below, we provide a selection of references for readers interested in further exploration. For the case where the driving noise is a L\'{e}vy process, we refer to \cite{BHR16SIAM,DXZ-2014,DWX20,FHR16CMP,MR10,PZ2011,PSXZ2012,WXX17,WYZZ 2209}. For the case of multiplicative noise, i.e., noise that depends on the solution, we refer to \cite{BFMZ24,DGT20,DP24,FZ24,GS17,Od07,Od08,RZZ15}, among others. For the case of localized noise, i.e., noise that acts only on a subset of the domain, we refer to \cite{Ner24,shirikyan-asens2015,Shi2021}, and related works.

\subsection{Main results}
Recall that $\nn= \lfloor  d/2+1\rfloor $.
$u_t=u_t(x)\in \R$ satisfies the equation on the $d$-dimensional tours $\mathbb T^d$
\begin{eqnarray}\label{1-1}
\left\{
\begin{split}
&\dif u_t+\operatorname{div} A(u_t)\dif t=\nu\Delta u_t\dif t+\dif \eta_t,
\\  &u_0\in H^\nn.
\end{split}
\right.
\end{eqnarray}
In the above,  $\nu>0$ is the  viscosity,  $H^\nn := \{ u \in W^{\nn,2}(\mathbb T^d,\mathbb R): \int_{\mathbb{T}^{d}}u(x)  \dif x = 0\}$ where $W^{\nn,2}$ is the standard Sobolev space.  We assume that the noise $\eta_t$ has the following form
\begin{eqnarray}
\label{p1215-1}
\eta_t=\sum_{i\in  \cZ_0}b_ie_i(x)W_i(t),
\end{eqnarray}
where  $ \cZ_0 \subseteq  \mZ_*^d:=\mZ^d\backslash \{ \mathbf{{0}}\}$
and  $W=(W_i)_{i\in \cZ_0}   $
is a
$|\cZ_0|$-dimensional standard   Brownian motion
on a filtered probability space $(\Omega,\FF,  \{\FF_t\}_{t\geq 0}, \pP)$    satisfying
the usual conditions (see   Definition 2.25 in~\cite{KS91}).
Before introducing $e_k(x)$, we define
\begin{equation*}
\mZ_{+}^d=\{(k_1,\cdots,k_d)\in \mZ_*^d: k_{d_{min}}>0 \}, \quad d_{min}:=\min_ {0\leq j\leq d}\{j,k_j\neq 0\}
\end{equation*}
and
$\mZ_{-}^d=\mZ_*^d\backslash\mZ_{+}^d.$
For any $k=(k_1,\cdots,k_d)\in \mZ_*^d$, we set
\begin{equation*}
e_k(x)=
\begin{cases}
\sin \lag k,x\rag& \text{if } (k_1,\cdots,k_d)\in \mZ_{+}^d, \\
{  -\cos\lag k,x\rag}   & \text{if } (k_1,\cdots,k_d)\in \mZ_{-}^d.
\end{cases}
\end{equation*}
The flux $A=(A_1,\cdots,A_d): \mathbb R^d\rightarrow\mathbb R$. Set $ \Bbbk=\max_{1\leq i\leq d}\big\{  \text{deg}(A_i(u)) \big\}$, where for any polynomial $P(u)=\sum_{j=0}^\infty a_ju^j$ of $u$, $\text{deg}(P(u)):=\sup\{j\geq 0:a_j\neq 0\}.
$
We always assume $\Bbbk\geq 1$ and
\bae\label{flux A}
A_i(u)=&\sum_{j=0}^{\Bbbk}c_{i,j}u^j,\quad i=1,\cdots,d.
\eae
Define
\baee
{c}_j:=& (c_{1,j},\cdots,c_{d,j})\in \mathbb R^d, j=1,\cdots,\Bbbk,\\
A^\perp:=&\{{k}\in \mathbb Z^d,\langle {c}_j,{k} \rangle_{\mathbb R^d}=0, \forall j=1,\cdots,\Bbbk\}.
\eaee
For $\Bbbk\geq 2, $ let
\baee
\mathbb{L}:=&\{\ell\in \mZ^d: \ell=\sum_{i=1}^{\Bbbk-1 }\ell^{(i)}, \ell^{(i)}\in \cZ_0,i=1,\cdots, \Bbbk-1 \}
\eaee
and set $\mathbb{L}=\{\mathbf{{0}}\}$ for $\Bbbk=1.$
For $n\geq 1$, define the sequence of sets recursively
\begin{eqnarray*}
\cZ_{n}= \big\{\kappa+\ell\in \mZ^d:\kappa\in \cZ_{n-1},\ell\in   \mathbb{L} ,
\langle {c}_\Bbbk,\kappa+\ell \rangle_{\mathbb R^d}\neq 0 \big\}.
\end{eqnarray*}
Denote
\begin{eqnarray*}
\cZ_\infty=\cup_{n=0}^\infty \cZ_n.
\end{eqnarray*}
Obviously,  $	\cZ_\infty=\cZ_0$ for $\Bbbk=1.$
Throughout this paper, we always assume that the set
$\cZ_0$  appeared in \eqref{p1215-1} is  finite and  symmetric (i.e., $-\cZ_0=\cZ_0$).
Obviously, $\cZ_n,0\leq n\leq \infty$ is also   symmetry,  i.e., $-k\in \cZ_n$ provided that $k\in \cZ_n.$
\begin{condition}[H{\"o}rmander type condition]
\label{16-5}
$\Bbbk\geq 2$ and  the following inclusion  holds:
\begin{eqnarray*}
\cZ_\infty^c\subseteq A^\perp.
\end{eqnarray*}
\end{condition}

In this paper, we adopt the following notation.
$H^k := \{ u \in W^{k,2}(\mathbb T^d,\mathbb R): \int_{\mathbb{T}^{d}}u(x)  \dif x = 0\}$ where $W^{k,2}$ is the standard Sobolev space.
For the case $k=0$, we simply write $H^0$ as $H$ and   $\langle, \rangle$ denotes the inner product in $H.$
Let  $\tilde H $ be a subspace of $H$ generated by the basis $\{e_n: n=(n_1,\cdots,n_d)\in \cZ_\infty \}$ and
denote  by  $\tilde H^\perp $ a subspace of $H$ generated by the basis $\{e_n: n=(n_1,\cdots,n_d)\in  \mZ_*^d \setminus  \cZ_\infty \}.$
Let    $\tilde  H^k= H^k \cap \tilde  H$. We endow the space $H^k$ and $\tilde  H^k$  with the usual  homogeneous Sobolev norm~$\|\cdot \|_k$.
Actually,  for any $k\in \mN$, one has
\begin{eqnarray*}
&& H^k= \big\{x=\sum_{j\in \mZ_*^d }x_je_j\in H,~ \|x\|_k^2:=\sum_{j\in \mZ_*^d } |j|^{2k}x_j^2<\infty \big\},
\\ && \tilde  H^k= \big\{x=\sum_{j\in \cZ_\infty }x_je_j\in H,~ \|x\|_k^2:=\sum_{j\in \cZ_\infty} |j|^{2k}x_j^2<\infty \big\}.
\end{eqnarray*}
We also define
\begin{eqnarray*}
&&
(\tilde H^k)^\perp= \big\{x=\sum_{j\in \cZ_\infty^c  }x_je_j\in H,~ \|x\|_k^2:=\sum_{j\in \cZ_\infty^c } |j|^{2k}x_j^2<\infty \big\},~\forall k\in \mN.
\end{eqnarray*}
\subsubsection{The first main result}

\

{
Before stating the first main result, we will demonstrate the well-posedness of equation \eqref{1-1}. To this end, we first provide the definition of a solution to the equation
\eqref{1-1}.

\begin{definition}
For any $T,n>0$, an $H^n$-valued  $\cF$-adapted process
$\{u_t\}_{t\in [0,T]}$
is called a solution of  \eqref{1-1} on the interval $[0,T]$ if
the following conditions are satisfied,
\begin{itemize}
	\item[(a)] $u \in C([0, T], H^n), \mP$-a.s.;
	\item[(b)]
	the following equality holds for every $t \in  [0, T]$
	and  smooth function  $\phi$  on $\mT^d$ with $\int_{\mT^d} \phi(x)\dif x=0$,
	\begin{align*}
		\langle u_t, \phi\rangle = & \langle u_0,\phi\rangle  +
		\int_{0}^t \nu   \langle u_s,\Delta \phi \rangle\dif s  +\sum_{i=1}^d\int_0^t
		\langle  A_i(u_s), \partial_{x_i}\phi \rangle   \dif s
		\\ & \quad + \langle \eta(t),\phi\rangle,  \quad \mP\text{-a.s.},
	\end{align*}
	where $\langle \cdot ,\cdot \rangle$ denotes the inner product in $H.$
\end{itemize}
%
\end{definition}

\begin{remark}
In the case of $n\geq 2$,
if an  $H^n$-valued  $\cF$-adapted process
$\{u_t\}_{t\in [0,T]}$
is   a solution of  \eqref{1-1}, then it  is also  a strong solution of
(\ref{1-1}) in the sense of  partial differential equations,
i.e.,
for any $0\leq t\leq T$,
\begin{eqnarray*}
	u_t=  u_0  +\int_{0}^t (\nu \Delta u_s -\operatorname{div} A(u_s))  \dif s+\eta(t),  \quad \mP\text{-a.s.}
\end{eqnarray*}

\end{remark}

}

{
Throughout this paper, we set
\begin{equation*}
\mmm=40 \Bbbk d(d+14\Bbbk)^2, \quad \nn=\lfloor  d/2+1\rfloor.
\end{equation*}
Before we state the first main result, we give a proposition.
}

\begin{proposition}\label{wp}
{
For any { $T\geq 1$}  and $u_0 \in H^\nn$,  
there exists a unique solution $u_t\in C([0,T],H^\nn)$   of  equation \eqref{1-1}, s.t.,
\baee
& \mathbb{E}\Big[\sup _{0 \leq t \leq  T}\|u_t\|_{{ \nn}}^2\Big]\leq \|u_0\|_{{ \nn}}^2+ { C\big(T+\left\|u_0\right\|_{L^\mm}^\mm\big) }.
\eaee
}
\end{proposition}
The proof of this proposition is given in Section \ref{2-2}.
Let $P_t(u_0,\cdot)$ be  the transition probabilities of  equation (\ref{1-1}), i.e,
\begin{eqnarray*}
P_t(u_0,A)=\mP(u_t \in A\big|u_0)
\end{eqnarray*}
for every $A\subseteq \mathscr B( H^\nn)$, where $\mathscr B(X)$ consists of all Borel sets of  metric space $X$.
For every $f:H^\nn  \rightarrow \mR$ and probability measure $\mu$ on $H^\nn$, define
\begin{equation*}
P_tf (u_0)=\int_{H^\nn} f (u)P_t(u_0,\dif u), \quad  P_t^*\mu(A)=\int_{H^\nn}P_t(u_0,A)\mu(\dif u_0).
\end{equation*}

The  first  main  theorem  of    this paper is as follows.
\begin{theorem}
\label{16-8}
Consider the  stochastic viscous conservation laws \textup{(\ref{1-1})} with  noise $\eta_t$ given by \textup{(\ref{p1215-1})}.
Assume that  the set
$\cZ_0$  appeared in \eqref{p1215-1} is  finite and  symmetric \textup{(i.e., $-\cZ_0=\cZ_0$)}.
If   $\Bbbk=1$  or  the Condition \textup{\ref{16-5}}  is satisfied,  then   the followings hold.
\begin{itemize}
\item[(\romannumeral1)]
There exists a unique invariant    measure $\tilde \mu$ on $\tilde H^\nn $ i.e., $\tilde \mu$ is a unique  measure on $\tilde H^\nn$ such that $P_t^* \tilde \mu=\tilde \mu$ for every $t\geq 0.$
\item[(\romannumeral2)]
There exists a unique invariant  measure $\mu$  measure on $H^\nn $ such that $P_t^*\mu=\mu$ for every $t\geq 0.$
Moreover, the  unique invariant  measure  $\mu$ on $H^\nn$ has the form:  $\tilde \mu \otimes \delta_0$, where
$\tilde \mu$ the  unique  invariant measure   on $\tilde H^\nn $ and
$ \delta_0$ is the dirac measure  concentrated on $0\in   (\tilde H^\nn)^\perp.$
\end{itemize}
\end{theorem}
%
%
First,  we demonstrate three  propositions which play key roles in the proof of Theorem \ref{16-8}, then we give a proof of this theorem.

\begin{proposition}
\label{p1218-4}
Under  the   Condition \ref{16-5},
$\tilde H^\nn $ is a  stable space for   the equation \eqref{1-1}, \textup{i.e.},  if the
initial value   $u_0\in \tilde H^\nn $, then $u_t\in \tilde H^\nn $ for all $t\geq 0.$
\end{proposition}
\begin{proof}
Obviously, we have
\begin{equation*}
u_t=P_{\tilde H^\nn }u_t+P_{(\tilde H^\nn)^\perp }u_t,
\end{equation*}
where for any $\phi\in H^\nn$, $P_{\tilde H^\nn}\phi=\sum_{j\in \cZ_\infty}e_j\langle \phi,e_j\rangle$
and $P_{(\tilde H^\nn)^\perp}\phi=\phi-P_{\tilde H^\nn }\phi.$
By (\ref{1-1}),  we have
\begin{eqnarray}\label{sub heat}
\left\{
\begin{split}
	&\frac{\partial}{\partial t} P_{( \tilde H^\nn)^\perp}  u_t+P_{( \tilde H^\nn)^\perp} \operatorname{div} A(u_t)  =\nu P_{( \tilde H^\nn)^\perp}  \Delta u_t,
	\\  & P_{( \tilde H^\nn)^\perp}  u_0\in ( \tilde H^\nn)^\perp.
\end{split}
\right.
\end{eqnarray}
For any  $k\in \cZ_\infty^c \subseteq A^\perp $ and $f\in H^\nn$,
it holds that
\baee
&\Big| \langle \operatorname{div} A(f),e_k\rangle\Big| =
\Big| \sum_{i=1}^d\langle A_i(f), \partial_{x_i} e_k \rangle \Big|
\\
=&\Big| \sum_{i=1}^d\big\langle \sum_{j=1}^{\Bbbk}c_{i,j}f^j,  \partial_{x_i}  e_k \big\rangle\Big|=\Big|\sum_{j=1}^{\Bbbk}\big\langle f^j, \sum_{i=1}^d  c_{i,j}  \partial_{x_i}  e_k \big\rangle\Big|
\\
=&\Big| \sum_{j=1}^{\Bbbk} \langle c_j, k \rangle_{\mathbb R^d}\langle f^j,   e_{-k} \rangle\Big|=0.
\eaee
Thus,  we conclude that
\begin{eqnarray}
\label{p1216-2}
P_{( \tilde H^\nn)^\perp }\operatorname{div} A(f)=0 \text{ for all }f\in H^\nn.
\end{eqnarray}
Furthermore, we also have   $P_{( \tilde H^\nn)^\perp } u_t=0,\forall t\geq 0$ provided that $P_{( \tilde H^\nn)^\perp } u_0=0$.
This implies  that  $\tilde H^\nn$ is a stable space for the dynamic \eqref{1-1}.
The proof is complete.
\end{proof}
\begin{proposition}
{
\label{3-11}
Under  the   Condition \textup{\ref{16-5}}, the Markov semigroup $\{P_t\}_{t\geq 0}$ has the e-property on $\tilde H^{\nn+5}  $,  \textup{i.e.},
for any  $\mathfrak{R}>0,$  $u_0\in B_{\tilde  H^{\nn+5} }(\mathfrak{R}):=\{u\in \tilde  H^{\nn+5}, \|u\|_{\nn+5}\leq \mathfrak{R} \}$,  bounded and Lipschitz continuous function $f$ on $\tilde H$ and $\eps>0$, there exists a $\delta>0$ such that
\begin{eqnarray*}
	|P_tf(u_0')-P_tf(u_0)|<\eps, ~\forall t\geq 0, \forall  u_0'\in \tilde H^\nn  \text{ with }
	\|u_0'-u_0\|<\delta \text{ and } \|u_0'\|_{\nn+5}\leq \mathfrak{R},
\end{eqnarray*}
where $\lVert\cdot\rVert$ denotes the $L^{2}$-norm.
}
\end{proposition}

{
The proof the Proposition \ref{3-11} will be present in Section \ref{proof of 1.6}.
If the Markov semigroup $\{P_t\}_{t\geq 0}$ has the {\em e-property} on $\tilde H^{\nn+5}  $,
then  for any   $\textcolor[rgb]{0.00,0.00,1.00}{u_0\in \tilde H^{\nn}}$
$,N\in \mN$ and
any bounded and Lipschitz continuous function $f$ on $\tilde H$,
one has
\begin{eqnarray}
\label{59-1}
\inf_{\delta>0}\sup_{u_0':\|u_0'-u_0\|\leq \delta}
\limsup_{t\rightarrow \infty}
|P_tf(P_N u_0')-P_tf(P_N u_0)|=0,
\end{eqnarray}
where  ${P}_N$  denotes the orthogonal projections from $ H$ onto $ {H}_N:=\text{span}\{e_j: j\in \mZ_*^d  \text{ and } |j|\leq N\}.$
(Remark: by the definition of $\tilde H$ and $H$, ${P}_N$  also is  the orthogonal projections from $\tilde   H$ onto $\tilde  {H}_N:=\text{span}\{e_j: j\in \cZ_\infty   \text{ and } |j|\leq N\}.$)
}

\begin{proposition}(Irreducibility)
\label{16-6}
Recall  $\nn=\lfloor  d/2+1\rfloor.$
For any $\mathcal{C},\gamma>0$, there exist positive  constants  $T=T(\mathcal{C},\gamma)>0,\tilde p_0=\tilde p_0( \mathcal{C},\gamma)$ such that
\begin{eqnarray*}
P_{T}(u_0, \cB_\gamma)\geq \tilde p_0,\quad\forall u_0\in H^{\nn} \text{ with }\|u_0\|_\nn\leq \mathcal{C},
\end{eqnarray*}
where $\cB_\gamma=\{u\in H^\nn,\|u\|_\nn \leq \gamma\}.$
In the above,  $T(\mathcal{C},\gamma),\tilde p_0( \mathcal{C},\gamma)$ denote  two positive constants depending on  $\mathcal{C},\gamma$ and the data of system \eqref{1-1},  \textup{i.e.},  $\nu,d,\Bbbk,(b_j)_{j\in \cZ_0},\mathbb{U}$,
$(c_{\mathbbm{i},\mathbbm{j}})_{1\leq \mathbbm{i}\leq d,0\leq \mathbbm{j} \leq \Bbbk}.$
\end{proposition}
The proof of  this  Proposition is demonstrated in Section \ref{sec irr}.

\textbf{Now we are in a position to prove Theorem \ref{16-8}.}
\begin{proof}
The proof of the existence of an invariant measure follows a similar approach to \cite[Lemma 8]{MR20}, and we provide the details in Appendix \ref{appen B}. Therefore, we focus solely on proving uniqueness here.

Before presenting the proof of uniqueness, we establish two prior moment bounds for any invariant probability measure $\tilde \mu$ on $\tilde H^\nn$:
\begin{eqnarray}
&& \label{p24-1}   \int_{\tilde H^\nn }\|u_0\|_{L^p}^p \tilde \mu(\dif u_0)\leq C_p<\infty,
~\forall p\in 2\mN,
\\ && \label{p24-7} \int_{\tilde H^\nn }\|u_0\|_{\nn}^2 \tilde \mu(\dif u_0)\leq C<\infty.
\end{eqnarray}
Using It\^o's formula, for $p \in  2\mN$, we get
\bae\label{p24-3}
&\|u_{t}\|_{L^{p}}^{p}= \|u_{0}\|_{L^{p}}^{p} + p\nu\int_{0}^{t}
\int_{\mT^d }u(x)^{p-1}\Delta u(x)\dif x\dif r\\
&+ p\sum_{k\in\cZ_{0}}b_{k}\int_{0}^{t}\langle u_{r}^{p -1}, e_{k}\rangle \dif W_{k}(r)
+ p(p - 1/2)\sum_{k\in\cZ_{0}}b_{k}^{2}\int_{0}^{t} \langle u_{r}^{p -2}, e_{k}^{2} \rangle \dif r.
\eae
By \cite[Proposition A.1]{CGV14},  for $p \in  2\mN$, it holds that
\begin{eqnarray}
\label{p24-4}
\int_{\mT^d }u(x)^{p-1}\big(-\Delta u(x)\big)\dif x\geq
C_{p}^{-1}\|u\|_{L^p}^p+\frac{1}{p}\|(-\Delta)^{1/2}u^{p/2}\|^2_{L^2}.
\end{eqnarray}
In the above,    $C_{p}\in (1,\infty)$ denotes    a positive constant that may   depend  on $p$ and $d.$
Combining the  above with (\ref{p24-3}), we arrive at that
\begin{eqnarray}
\label{p24-6}
\begin{split}
	\|u_{t}\|_{L^{p}}^{p}& \leq   \|u_{0}\|_{L^{p}}^{p}-C_p^{-1}
	\int_0^t \|u_s\|_{L^p}^p\dif s
	\\ & \quad\quad   +C_p t+p\sum_{k\in\cZ_{0}}b_{k}\int_{0}^{t}\langle u_{r}^{p -1}, e_{k}\rangle \dif W_{k}(r),
\end{split}
\end{eqnarray}
where $C_p\in (1,\infty)$ is  a positive constant depending   on $p$ and
$\nu,d,\Bbbk , \{b_k\}_{k\in \cZ_0}$ and $\mathbb{U}=|\cZ_0|.$
For any $\eps>0$,  let $B_\eps=\{u\in H^\nn: \|u\|_\nn \leq b_\eps  \}$,
where $b_\eps$ is a constant such that $\tilde \mu (B_\eps)>1-\eps.$
Then, for any $\cN,\eps>0$, (\ref{p24-6}) implies that
\begin{eqnarray*}
&& \int_{\tilde H^\nn }\big(\|u_0\|_{L^p}^p  \wedge \cN\big)  \tilde \mu(\dif u_0)=\int_{\tilde H^\nn }\mE_{u_0} \big(\|u_t\|_{L^p}^p  \wedge \cN\big)  \tilde \mu(\dif u_0)
\\ &&\leq \cN \eps+\int_{B_\eps }\mE_{u_0} \big(\|u_t\|_{L^p}^p  \wedge \cN\big)  \tilde \mu(\dif u_0)
\leq   \cN \eps+\int_{B_\eps }\mE_{u_0} \|u_t\|_{L^p}^p   \tilde \mu(\dif u_0)
\\ &&\leq   \cN \eps+\int_{B_\eps }\Big(e^{-C_p^{-1}t }\|u_{0}\|_{L^{p}}^{p} +C_p  \Big)  \tilde \mu(\dif u_0)
\\ && \leq   \cN \eps+C_p \Big(e^{-C_p^{-1}t }b_\eps^{p} +1  \Big).
\end{eqnarray*}
In the above, first letting $t\rightarrow \infty$, then letting $\eps\rightarrow 0$ and in the end letting $\cN \rightarrow \infty$,
we obtain  the desired result (\ref{p24-1}).

Now,  we give a proof of (\ref{p24-7}).
Using It\^o's formula for $ \langle u_t, (-\Delta )^{\nn-1 }u_t\rangle$,
we get
\begin{eqnarray*}
&& \|u_t\|_{\nn-1}^2 +2\nu \int_0^t\|u_s\|_{\nn}^2\dif s+2\int_0^t \langle\operatorname{div} A(u_s),  (-\Delta )^{\nn-1 }u_s  \rangle\dif s
\\&&  =2 \sum_{i\in \cZ_0}b_i \int_0^t \langle e_i, (-\Delta )^{\nn-1 }u_s\rangle\dif W_i(s) + \sum_{i\in \cZ_0} |b_i|^2  \langle
e_i, (-\Delta )^{\nn-1 }e_i  \rangle t.
\end{eqnarray*}
Thus,  by the Lemma \ref{div reg} below, for some $m=m(\nn,d,\Bbbk)\in 2\mN$, we conclude that
\begin{eqnarray}
\nonumber  && \|u_t\|_{\nn-1}^2 +2\nu \int_0^t\|u_s\|_{\nn}^2\dif s
\leq
2\int_0^t \| \operatorname{div} A(u_s)\|_{\nn-2}\| u_s\|_{\nn}\dif s
\\ \nonumber &&\quad  +2 \sum_{i\in \cZ_0}b_i \int_0^t \langle e_i, (-\Delta )^{\nn-1 }u_s\rangle\dif W_i(s) + C t
\\ \nonumber &&\leq  \nu \int_0^t\|u_s\|_{\nn}^2\dif s+C\int_0^t \|u_s\|_{L^m}^m\dif s
\\ \label{p24-8}  && \quad +2 \sum_{i\in \cZ_0}b_i \int_0^t \langle e_i, (-\Delta )^{\nn-1 }u_s\rangle\dif W_i(s) + C t.
\end{eqnarray}

For any  invariant probability measure $\tilde  \mu $
on $\tilde H^\nn$, by (\ref{p24-8}) and (\ref{p24-1}), it holds that
\begin{eqnarray*}
&& \nu t  \int_{\tilde H^\nn}  \|u_0\|_{\nn}^2 \tilde
\mu(\dif u_0) =  \nu \int_{\tilde H^\nn}  \int_0^t  \mE_{u_0}\|u_s\|_{\nn}^2~\dif s \tilde
\mu(\dif u_0)
\\ &&
\leq C \int_{\tilde H^\nn}  \int_0^t \mE_{u_0} \|u_s\|_{L^m}^m \dif s\tilde
\mu(\dif u_0)+Ct
\\ &&= Ct  \int_{\tilde H^\nn}   \|u_0\|_{L^m}^m  \tilde
\mu(\dif u_0)+Ct\leq Ct,\quad\forall t>0.
\end{eqnarray*}
The above inequality implies the desired result  (\ref{p24-7}).

Now we give a proof of   $(\romannumeral1)$ and  $(\romannumeral2)$.
First,  we give a proof of $(\romannumeral1)$.
In   the  case  $\Bbbk=1$,  one has  $\cZ_\infty=\cZ_0$. In this case,   we can assume that
$A_i(u)=a_i u,i=1,\cdots,d$ and at least one of $a_i,i=1,\cdots,d$ is not zero.
For any $k\in \cZ_0$, taking inner product with $e_k$ in (\ref{1-1}),
we get
\begin{eqnarray*}
\dif \langle u_t,e_k\rangle=-\nu |k|^2 \langle u_t,e_k\rangle\dif t
-\sum_{i=1}^d a_ik_i \langle u_t,e_{-k}\rangle\dif t +b_k\dif W_k(t).
\end{eqnarray*}
Observing  that $b_k\neq 0,\forall k\in \cZ_0$,
the process $U_t=\sum_{k\in \cZ_0}\langle u_t,e_k\rangle e_k$
has many good properties, such as  smooth density,
$U_t\in \tilde H^\nn  $ provided that $U_0\in \tilde H^\nn $
and  $\mP(U_t\in \mathcal O)>0$  for any  open set   $\mathcal O \subseteq \tilde  H^\nn.$
One can refer to   \cite[Proposition 6.5]{KS91}.
Then, the uniqueness of invariant measure on the space $\tilde H^\nn$ follows.
Now we consider the case $\Bbbk\geq 2$ and assume that Condition \ref{16-5} holds.  Assume that there are two distinct invariant probability measures ${ \tilde\mu_1}$ and ${\tilde  \mu_2}$  on $\tilde H.$
Notice that
the following set of  functions on $\tilde H^\nn$ separate all the points in $\tilde H^\nn:$
\begin{eqnarray*}
&& \Big\{ f(P_Nx): x\in \tilde H^\nn \rightarrow f(P_Nx)  \in \mR ~\big|~
N\in\mN,
f \text{ is a  bounded and
}
\\ && \quad\quad\quad\quad \text{ Lipschitz continuous function  on $\tilde H$}
\Big\}.
\end{eqnarray*}
Thus, by Proposition \ref{3-11}  and (\ref{59-1}),
also  with the help of \cite[Proposition 1.10]{GL15}, one has
\begin{eqnarray}
\label{16-10}
\text{Supp } { \tilde \mu_1} \cap  \text{Supp } { \tilde \mu_2}=\emptyset.
\end{eqnarray}
On the other hand, by (\ref{p24-7}), for every invariant measure $\tilde\mu$ on $\tilde H^\nn$, the following priori bound
\begin{eqnarray*}
\int_{\tilde H}\|u\|_n^2\tilde\mu(\dif u)\leq C,
\end{eqnarray*}
holds. Following the arguments in the proof of \cite[Corollary 4.2]{HM-2006}, and with the help of Proposition \ref{16-6}, for every invariant measure $\tilde\mu$, we have $0\in \text{Supp } \tilde\mu$.This conflicts with (\ref{16-10}).
We complete the proof of $(\romannumeral1)$.

Now, we give a proof of  $(\romannumeral2)$.
For the case  $\Bbbk=1$ and $k\in \cZ_0^c$, by direct calculations, one arrives at
\begin{eqnarray}
\label{p15-2}
\frac{1}{2}\frac{\dif \big(\langle u_t,e_{k}\rangle^2+ \langle u_t,e_{-k}\rangle^2\big)  }{\dif t}=-\nu |k|^2 \big(\langle u_t,e_{k}\rangle^2+ \langle u_t,e_{-k}\rangle^2\big).
\end{eqnarray}
For the case      Condition \ref{16-5}  holds, observe  the facts   (\ref{sub heat})--(\ref{p1216-2}).
Therefore, under our conditions,
for any initial value $u_0\in H^\nn$,
$ P_{( \tilde H^\nn)^\perp}   u_t$  will   decays   to $0$ as $t$ tends to infinity
and
$ P_{( \tilde H^\nn)^\perp}   u_t=0,\forall t\geq 0$ provided  that
$P_{( \tilde H^\nn)^\perp}  u_0=0.$
Therefore,  any   solution of (\ref{1-1})  that  starts  from $u_0\in H^\nn $ will eventually  stay in  $\tilde H^\nn$ as $t\rightarrow \infty$.
Furthemore, any    unique invariant  measure  $\mu$  of $P_t$ on $H^\nn$ must  has the form:  $\tilde \mu \otimes \delta_0$, where
$\tilde \mu$ is  the  unique invariant   measure   on $\tilde H^\nn$ and
$ \delta_0$ is the dirac measure  concentrated on $0\in  ( \tilde H^\nn)^\perp.$
Observing that  the invariant   measure   on $\tilde H^\nn$ is unique, we complete the proof of  $(\romannumeral2)$.
\end{proof}

Our Theorem \ref{16-8} covers many interesting cases. For details, please to see the following corollary and example.
For $i=1,\cdots,d$, set  $
\varsigma_i=(\varsigma_{i,1},\cdots,\varsigma_{i,d})\in \mZ_*^d$ with $\varsigma_{i,i}=1$ and $\varsigma_{i,j}=0$ for  $j\neq i$.

\begin{corollary}
\label{1217-1}
Let  $\cZ_0 = \big\{\varsigma_i, -\varsigma_i, 2\varsigma_i,-2\varsigma_i,
i=1,\cdots,d\big\}$.If one of the   following three  conditions holds, then
the  results of Theorem \textup{\ref{16-8}}    hold.
\begin{itemize}
\item[(\romannumeral1)] Assume  that  $d\geq 1, \Bbbk\geq 2$
and
\begin{eqnarray*}
	A_i(u)=c_{i,\Bbbk}  u^\Bbbk,
\end{eqnarray*}
where
$c_{i,\Bbbk},i=1,\cdots,d$ are  real numbers.

\item[(\romannumeral2)]
Assume that $d\geq 1,\Bbbk\geq 2$ and
\begin{eqnarray*}
	&&  A_i(u)=c_{i,\Bbbk}  u^{\Bbbk}+\sum_{j=0}^{\Bbbk-1}c_{i,j}u^j,\quad i=1,\cdots,d.
\end{eqnarray*}
{Furthermore, assume that $c_\Bbbk$ is rationally independent, \textup{i.e.},}
\begin{eqnarray}
	\label{pp10-4}
	\big\{k\in\mathbb Z^d:\langle c_\Bbbk,k\rangle=\sum_{i=1}^d c_{i,\Bbbk} k_i= 0 \big\}=\{ \mathbf{0}\}.
\end{eqnarray}
\item[(\romannumeral3)] Assume that  $d=1, \Bbbk\geq 1.$
	\end{itemize}
	
\end{corollary}

The proof of this corollary is put in Appendix \ref{B}.

\begin{example}
\label{31-1}
We give an interesting example here. Set \begin{equation*}
	S=\{ \sqrt{n}; n\in \mathbb N\text{ is divisible by no square number other than 1}\}\footnote{
		We say   $n$ is divisible by no square number other than $1$,
		if $\{m\in \mN: \frac{n}{m^2} \text{is an integer}\}=\{1\}.$
	}
	.
\end{equation*}
Then, any finite subset of $S$ is  rationally independent,
i.e., for any  distinct  elements   $s_1,\cdots,s_d$ in  $  S$,
we have  $\big\{k=(k_1,\cdots,k_d)\in\mathbb Z^d: \sum_{i=1}^d s_i k_i= 0 \big\}=\{ \mathbf{0}\}.$
Therefore,  for  the case of (ii) in Corollary \ref{1217-1},
if  $c_{i,\Bbbk},i=1,\cdots, d$ are distinct numbers in $S$, then   the Condition \ref{16-5} holds.

However, for many  examples,  we can't verify the  (ii) in Corollary \ref{1217-1}.
The following is an interesting  example.
Set $d=2$ and
\baee
A_1(u)&= c_{1,1}u+e u^2,\\
A_2(u)&=c_{2,1}u+\pi u^2, \\
\eaee
where   $c_{1,1}$ and $c_{2,1}$ are not all $0$.
In this case,
Wether  $
\{(k_1,k_2)\in \mZ^2: ek_1+\pi k_2=0\}=\{ \mathbf{0} \}
$
or not is unknown.
Actually,  for mathematicians,  it's still  unknown whether  $e/\pi$ is rational or not. One can refer to \cite{Lang66} for  related problems.
\end{example}

\subsubsection{The second main result}

\

Now, we state  the second main theorem of this paper.
\begin{theorem}
\label{p1218-1}
Assume that  the Condition \textup{\ref{16-5}} holds
and let $\tilde \mu$  be the   unique invariant    measure of $P_t$ on $\tilde H^\nn$.
Then the law of orthogonal projection of $\tilde  \mu$ onto any finite subspace is absolutely continuous with respect to the associated Lebesgue measure.
\end{theorem}

Theorem \ref{p1218-1}  implies the following corollary.
\begin{corollary}
\label{pp29-1}
Assume that  the Condition \textup{\ref{16-5}} holds
and let $\mu$  be the   unique invariant    measure of $P_t$ on $H^\nn$.
Then,
the law of orthogonal projection of $ \mu$ onto any finite subspace is absolutely continuous with respect to the associated Lebesgue measure
if and only if the  following {\bf{algebraically non-degenerate condition}}  holds for the flux $A$:
\begin{eqnarray}
\label{p0209-5}
{ A^\perp\subseteq \cZ_0\cup \{0\}.}
\end{eqnarray}
\end{corollary}
Now we give a proof of Corollary \ref{pp29-1} based on Theorem \ref{p1218-1}.
\begin{proof}
First, we assume that  { $A^\perp\subseteq \cZ_0\cup \{0\}.$}
Obviously, it implies that
\begin{eqnarray}
\label{p0209-1}
(A^\perp)^c  \cup \cZ_0=\mZ_*^d.
\end{eqnarray}
By  our Condition \ref{16-5} and (\ref{p0209-1}), we get $\cZ_\infty
=
\cZ_\infty \cup \cZ_0
\supseteq   (A^\perp)^c \cup \cZ_0 =\mZ_*^d.$
Furthermore,  we also  have $\tilde H^\nn=H^\nn$.
Thus, with the help of   Theorem \ref{p1218-1}, the law of orthogonal projection of $ \mu$ onto any finite subspace is absolutely continuous with respect to the associated Lebesgue measure.

Second, we assume that  the law of orthogonal projection of $ \mu$ onto any finite subspace is absolutely continuous with respect to the associated Lebesgue measure. Furthermore, we assume that (\ref{p0209-5}) does not hold.
Then,   for some $k\in \mZ_*^d$, one has
\begin{eqnarray}
\label{p0209-3}
k \in A^\perp\setminus  \cZ_0.
\end{eqnarray}
{ By    the  $(ii)$   in Theorem \ref{16-8},} we get $e_k\notin   (\tilde H^\nn)^\perp$ and eventually, one has  $k\in \cZ_\infty. $
In view of  our definition of $\cZ_n$, since $k \in A^\perp$,  we conclude that $k\notin \cZ_n$ for any $n\geq 1.$
Thus, we must have $k\in \cZ_0.$
This conflicts with  (\ref{p0209-3}).
We complete the proof of (\ref{p0209-5}).

\end{proof}

\begin{remark}\label{rm nond}
For polynomial flux $A$ of the form  \eqref{flux A}, the non-degenerate condition \eqref{nond} is equivalent to
\begin{equation*}
\sup _{\alpha \in \mathbb{R}, \beta \in \mathbb{S}^{d-1}}\operatorname{measure}\Big\{\xi \in \mathbb{R} : \big|\alpha+\sum_{j=1}^{\Bbbk}\xi^{j-1}j\langle c_j,\beta \rangle\big|<\varepsilon\Big\}\leq C \epsilon^b, ~ \text{where } C>0,  b\in (0,1].
\end{equation*}
Obviously,   the above inequality   implies
\bae\label{polynond}
A^\perp_\ast:=\{\beta\in\mathbb R^d:  \langle c_j,\beta \rangle=0,\forall j=2,\cdots, \Bbbk\}=\{\textbf{0}\}.
\eae
Recall that
\baee
A^\perp=&\{{k}\in \mathbb Z^d:\langle {c}_j,{k} \rangle_{\mathbb R^d}=0, \forall j=1,\cdots,\Bbbk\}.
\eaee
In general, $A^\perp\subsetneqq A^\perp_\ast$. Thus,  we conclude that the algebraically non-degenerate condition \eqref{p0209-5} is strictly weaker than the usual  non-degenerate condition \eqref{nond}.

{

Additionally, in many cases, the H\"ormander-type condition \ref{16-5}  and the algebraically non-degenerate condition (\ref{p0209-5})  hold, even though the standard non-degenerate condition \eqref{nond}  may fail. For instance, consider the scenario where
$\Bbbk-1<d,$ in such case, the equality (\ref{polynond})   does not hold. Nevertheless, as demonstrated in Example \ref{31-1}, even when
$\Bbbk-1<d,$   the H\"ormander-type condition \ref{16-5} and the algebraically non-degenerate condition (\ref{p0209-5}) are   still satisfied
provided that $\cZ_0 = \big\{\varsigma_i, -\varsigma_i, 2\varsigma_i,-2\varsigma_i,
i=1,\cdots,d\big\}$ and   $c_{i,\Bbbk},i=1,\cdots, d$ are distinct numbers in $S$.

%


}
\end{remark}

Before stating the proof of Theorem \ref{p1218-1}, we give a lemma first.

Considering  the equation (\ref{1-1}),
if the initial value  $u_0\in \tilde H^\nn$,
then $u_t\in \tilde H^\nn$ for any $t\geq 0.$
Assume that  $N\geq 1$ and  $k_1,k_2,\cdots,k_N$ are elements in $\cZ_\infty$, let     $\tilde P_N: \phi \in \tilde H^\nn   \mapsto \sum_{\ell =1}^N \langle
\phi,e_{k_\ell}
\rangle e_{k_\ell} $ be an
orthogonal projection.
Then, the following lemma holds.
\begin{lemma}
\label{p1218-3}
Consider the equation \eqref{1-1} with    initial value  $u_0\in \tilde H^\nn$,
then
the  law of  $\tilde P_N  u_t$   is absolutely continuous with respect to the associated Lebesgue measure.
\end{lemma}
\begin{proof}[Proof of Theorem \textup{\ref{p1218-1}}]
Since
$
\tilde P_N  u_t=\sum_{\ell =1}^N \langle
u_t,e_{k_\ell }
\rangle e_{k_\ell },
$
by Lemma \ref{17-1}, we conclude that
\begin{eqnarray*}
\cD_r^i\big(\tilde P_N  u_t\big)=\sum_{\ell =1}^N \langle
\cD_r^i u_t,e_{k_\ell}
\rangle e_{k_\ell}=\sum_{\ell =1}^N \langle
J_{r,t}Q\theta_i,e_{k_\ell}
\rangle e_{k_\ell},
\end{eqnarray*}
where   $\cD_r^i$ denotes the Malliavin derivative with respect to the $i$th component of the noise at time $r$
and $\{\theta_i\}_{i =1}^{\mathbb{U} }$ is the standard basis of $\mR^{\mathbb{U}}.$
Hence, by  Lemma \ref{L^1}, one arrives at
\begin{eqnarray*}
&&  \big\| \cD_r^i\big(\tilde P_N  u_t\big)\big\|\leq
C \sum_{\ell =1}^N \big| \langle
J_{r,t}Q\theta_i,e_{k_\ell}
\rangle \big|
\\ &&  \leq C \sum_{\ell =1}^N
\| J_{r,t}Q\theta_i\|_{L^1}\| e_{k_\ell}\|_{L^\infty}\leq C<\infty,
\end{eqnarray*}
where $C$ is a constant depending on $N,\mathbb{U}=|\cZ_0|, (k_\ell)_{\ell=1}^N$, and
$(b_j)_{j\in \cZ_0}$.
The above inequality implies that
\begin{eqnarray}
\nonumber &&  \tilde P_N  u_t \in \mathbb{H}^1(\Omega, \tilde P_N H)
:=\Big\{
X:\Omega\rightarrow  \tilde P_N H:
\mE \|X\|^2,\mE \int_0^t \|\cD_r^i X\|^2 \dif r<\infty,
\\  \label{p1218-2} && \quad\quad \quad\quad  \quad\quad  \quad\quad   \quad\quad   \quad \text{for all }  i=1,\cdots,\mathbb{U}.
\Big\}.
\end{eqnarray}
On the other hand,   by Proposition \ref{1-66}, for any $\phi \in \tilde H$ with $\phi \neq 0$, it holds that
\begin{eqnarray*}
\langle \cM_{0,t}\phi,\phi\rangle>0 \quad  a.s.
\end{eqnarray*}
Combining the above with (\ref{p1218-2}), also with the help of   \cite[Theorem 2.1.2]{nualart2006}, we obtain the desired result and complete the proof.
\end{proof}

\textbf{Now we are in a position to continue the   proof of  Theorem  \ref{p1218-1}. }
Assume that  $N\geq 1$ and  $k_1,k_2,\cdots,k_N$ are some elements in $\cZ_\infty$, let     $\tilde P_N: \phi \in \tilde H^\nn  \mapsto \sum_{\ell =1}^N \langle
\phi,e_{k_\ell}
\rangle e_{k_\ell} $ be an
orthogonal projection.
If  the law   of initial value $u_0$ is the invariant measure $\tilde \mu$, then the  law   of $u_t$ is also  $\tilde \mu.$
Thus, for any $A\subseteq P_N \tilde H^\nn$ and $t>0$,
\begin{eqnarray*}
\tilde P_N \tilde \mu(A)=\int_{\tilde H^\nn} \mP(\tilde  P_N  u_t\in A ) \dif \tilde \mu.
\end{eqnarray*}
The conclusion of Theorem  \ref{p1218-1} follows directly by combining the preceding discussion with Lemma \ref{p1218-3}.
The proof is complete.

%
%

\subsection{Organizations of this paper}

The remainder of this paper is organized as follows: Section \ref{pre} establishes the necessary mathematical preliminaries, including notation, definitions, the well-posedness of the stochastic conservation laws, and a priori estimates of the solutions. Section \ref{S:3} is devoted to proving the invertibility of the Malliavin matrix $\cM_{0,t}$. Based on a localized technique, we provide a proof of Proposition \ref{3-11} in Section \ref{proof of 1.6}, which establishes the e-property. Finally, Section \ref{sec irr} presents a proof of irreducibility. The proof of the existence of an invariant measure is included in Appendix \ref{appen B}, as it follows a similar approach to \cite[Lemma 8]{MR20}. Since the proof of Corollary \ref{1217-1} is largely independent of the other content in this paper, it is placed in Appendix \ref{B}.

\section{Preliminaries}\label{pre}
\label{S:2}

\subsection{Notation}

In this paper, we use the following notation.

\smallskip
\noindent
$\mN$ denotes the set of  positive integers. $\mZ_*^d:=\mZ^d\backslash \{ \mathbf{{0}}\}.$

\smallskip
\noindent
{  For any function $f$ on $\mT^d$ and $p>0$,
$\|f\|_{L^p}:=\big(\int_{\mT^d}|f|^p \dif x\big)^{1/p}.$
$L^\ty(H)$  is the space of bounded Borel-measurable functions $f:H\to\R$   with the norm $\|f\|_{L^\infty}=\sup_{w\in H}|f(w)|$.
}

\smallskip
\noindent
For any $N\in\mN$, let ${H}_N=\text{span}\{e_j: j\in \mZ_*^d  \text{ and } |j|\leq N\}.$
${P}_N$  denotes the orthogonal projections from $H$ onto ${H}_N$.
Define   ${Q}_Nu:=u-{P}_Nu, \forall u\in H.$

\smallskip
\noindent
For $\alpha\in \mR$ and a smooth function $u\in H$, we define the norm $\|u\|_\alpha$ by
\begin{eqnarray}\label{eq 2024 04 15 00}
\|u\|_\alpha^2=\sum_{k\in \mZ_*^d}|k|^{2\alpha}u_k^2,
\end{eqnarray}
where $u_k$ denotes the Fourier mode with wavenumber $k$.When $\alpha=0$, we also denote this norm $\|\cdot \|_\alpha$ by $\|\cdot \|$.$H^n= H^n(\mT^d, \R)\cap H$, where $H^n(\mT^2, \R)$ is the usual Sobolev space of order~$n\ge1$. We endow the space $H^n$ with the norm~$\|\cdot \|_n$.
Usually,   $\langle \cdot, \cdot \rangle$ denotes the inner product in $H.$




\smallskip
\noindent
$C_b(H)$ is the space of continuous functions. ~$C^1_b(H)$ is the space of functions~$f\in C_b(H)$ that are continuously Fr\'echet differentiable with bounded derivatives.
The Fr\'echet derivative of $f$ at  point  $w$ is denoted by $D  f (w)$
and   we usually also  write $D  f (w)\xi$ as $D _\xi f(w)$ for any $\xi\in H.$
Taken as a function of $w$, if  $g(w):=D_\xi f(w)$ is also Fr\'echet differentiable, then  for any $\zeta\in H$,
we usually denote  $D_\zeta g(w)$  by $D^2f(w)(\xi,\zeta).$

\smallskip
\noindent
$\cL(X,Y)$ is the space of bounded linear operators from Banach spaces $X$ into Banach space $Y$ endowed with the natural norm $\|\cdot\|_{\cL(X,Y)}$.
If there are no confusions, we always write  the operator  norm  $\|\cdot\|_{\cL(X,Y)}$
as $\|\cdot\|$.

\smallskip
\noindent
For any $t\geq 0$,   the filtration $\cF_{t}$ is defined by
\begin{eqnarray*}
\cF_{t}:=\sigma(W_{s} : s\leq t).
\end{eqnarray*}

\smallskip
\noindent
Throughout this paper, we set{
\begin{equation*}
\mmm=40 \Bbbk d (d+14\Bbbk)^2, \quad \nn=\lfloor  d/2+1\rfloor.
\end{equation*}
}
Without otherwise  specified statement, in this section,  we always assume that
$u=(u_t)_{t\geq 0}$ is the solution of (\ref{1-1}) with initial value $u_0\in H^\nn.$

\smallskip
\noindent
Since there are many constants appearing in the proof, we adopt  the following convention.
Without otherwise specified,  the letters $C,C_1,C_2,\cdots$ are  always  used to denote unessential constants that  may change from line to line  and  implicitly  depend on the data of the system (\ref{1-1}), i.e., $\nu,d,\Bbbk, \{b_k\}_{k\in \cZ_0},\mathbb{U}=|\cZ_0|$  and $
(c_{\mathbbm{i},\mathbbm{j}})_{1\leq \mathbbm{i}\leq d,0\leq \mathbbm{j} \leq \Bbbk}$.
Also, we usually do not explicitly indicate the dependencies on the parameters  $\nu,d,\Bbbk, \{b_k\}_{k\in \cZ_0},\mathbb{U}=|\cZ_0|$  and $
(c_{\mathbbm{i},\mathbbm{j}})_{1\leq \mathbbm{i}\leq d,0\leq \mathbbm{j} \leq \Bbbk}$ on every occasion.
\subsection{A priori estimates of the solutions}

{ In this subsection, unless otherwise specified,   we always assume that   $ \{u_t\}_{t\geq 0}\in C([0,\infty),H^\nn) $ is   a   solution of \eqref{1-1} with initial value $u_0\in H^\nn,$ i.e., for any $T>0,$  $ \{ u_t\}_{t\in [0,T]}\in C([0,T],H^\nn) $ is   a   solution of \eqref{1-1}. }
With regard to  $(u_t)_{ t\geq 0}$, we have the following a priori estimates.

Using the continuous embedding from $H^\nn$ to $L^\infty$,
and taking  similar   procedures   in \cite[Proposition 5]{MR20},
we have the following   $L^1$ contraction.{
\begin{lemma}\label{path L1 contr}
{ Let $\tau$ be a  bounded stopping time  with respect to  $\cF_t$}  and $u_t,v_t\in C([0,\tau),H^\nn)$ be two solutions of \eqref{1-1} with initial values $u_0$ and $v_0$, respectively. Then for every $0 \leq s \leq t<\tau$, almost surely, we have
\begin{equation*}
	\|u_t-v_t\|_{L^1} \leq\|u_s-v_s\|_{L^1}.
\end{equation*}
\end{lemma}

\begin{lemma}\label{L^p 2.2}
For any even number   $ p\geq 2$, there exist    constants  $\mathcal E_p>0,  C_p\in (1,\infty)$  which only  depend on $p$ and $\nu,d,\Bbbk, (b_{i})_{i\in \cZ_0},\mathbb{U}$   such that
for any ${ t\geq 1,}$ $K\geq 1$
and  $u_0\in H^\nn$, it holds that
\begin{eqnarray}
&&
\begin{split}
	\label{L-8-62}
	&  \mE \Big(\|u_{t}\|_{L^{p}}^{p} + \int_0^t \|u_r\|_{L^{p}}^{p}\dif r \Big) \leq C_p { (  \|u_0\|_{L^{p}}^{p}+t)},
\end{split}
\\ &&
\begin{split}
	&  \label{3030-1}
	\mP\Big( \|u_t\|^p_{L^p}+ \int_0^t  \|u_r\|_{L^p}^p \dif r  - \mathcal E_p  t \geq K
	\Big)
	\leq  \frac{C_p t^{49}  (t+\|u_0\|_{L^{100p}}^{100p})}  {(K+\mathcal E_p   t)^{100} }.
\end{split}
\end{eqnarray}
\end{lemma}

\begin{proof}
First, let us prove (\ref{L-8-62}) for $p\in 2\mN$.
With the help of  It\^o's formula, we obtain
\bae\label{L^p formula2}
&\|u_{t}\|_{L^{p}}^{p}\\
=& \|u_{0}\|_{L^{p}}^{p} + p\nu\int_{0}^{t}
\int_{\mT^d }u_r^{p-1}\Delta u_r\dif x\dif r\\
&+ p\sum_{k\in\cZ_{0}}b_{k}\int_{0}^{t}\langle u_{r}^{p -1}, e_{k}\rangle \dif W_{k}(r)
+ \frac12 p(p - 1)\sum_{k\in\cZ_{0}}b_{k}^{2}\int_{0}^{t} \langle u_{r}^{p -2}, e_{k}^{2} \rangle \dif r.
\eae
On the other hand, by \cite[Proposition A.1]{CGV14}, it holds that
\begin{eqnarray}
\label{28-1}
\int_{\mT^d }u(x)^{p-1}\big(-\Delta u(x)\big)\dif x\geq
C_{p}^{-1}\|u\|_{L^p}^p+\frac{1}{p}\|(-\Delta)^{1/2}u^{p/2}\|^2.
\end{eqnarray}
In the above,    $C_{p}\in (1,\infty)$ denotes    a positive constant that may   depend  on $p$ and $d.$
Combining the  above with (\ref{L^p formula2}), we arrive at that
\begin{eqnarray}
\label{pp0202-1}
\begin{split}
	& \|u_{t}\|_{L^{p}}^{p}+\int_{0}^t\|u_{r}\|_{L^{p}}^{p}\dif r
	\\ & \leq  C_{1,p}  \sum_{k\in\cZ_{0}}b_{k}\int_{0}^{t}\langle u_{r}^{p -1}, e_{k}\rangle \dif W_{k}(r)+C_{2,p} \int_{0}^t\|u_{r}\|_{L^{p-2}}^{p-2}\dif r,
\end{split}
\end{eqnarray}
where $C_{1,p},C_{2,p}$ are some  positive  constants depending on $p$ and
$\nu,d,\Bbbk, (b_{i})_{i\in \cZ_0},\mathbb{U}$,  and they may   change from line to line.  {   By  $u_0\in H^\nn$, Proposition \ref{wp} and $\|w\|_{L^p}\leq C_p\|w\|_\nn$,
the term
\begin{equation}
\label{p21-1}
C_{1,p}\sum_{k\in\cZ_{0}}b_{k}\int_{0}^{t}\langle u_{r}^{p -1}, e_{k}\rangle \dif W_{k}(r)
\end{equation}
in (\ref{pp0202-1}) is  a {{local martingale.}}
Thus, taking expectation on both sides of (\ref{pp0202-1})
,
it yields that\footnote{ Observe that $u_t\in C([0,\infty),H^\nn) .$
Define $\tau_n:=\inf\{t\geq 0, \int_0^t\|u_r\|_{L^{p-1}}^{2p-2}\dif r\geq n\}.$
Then, by (\ref{pp0202-1})--(\ref{p21-1}), for any $\cN>0$ and $n\in \mN$, one has
\begin{eqnarray*}
&& \mE \Big(\|u_{t\wedge \tau_n }\|_{L^{p}}^{p}\wedge \cN\Big)+\mE  \int_{0}^{t\wedge \tau_n} \|u_{r}\|_{L^{p}}^{p}\dif r
\leq C_{p}\mE \Big(\|u_{0}\|_{L^{p}}^{p} +\int_{0}^{t\wedge \tau_n} \|u_{r}\|_{L^{p-2}}^{p-2}\dif r\Big).
\end{eqnarray*}
In the above, first letting $n\rightarrow \infty$ and then letting $\cN\rightarrow \infty$, we get the desired result (\ref{p21-2}).
}   }
\begin{eqnarray}
\label{p21-2}
&& \mE \Big(\|u_{t}\|_{L^{p}}^{p}+\int_{0}^t\|u_{r}\|_{L^{p}}^{p}\dif r\Big)
\leq C_{p}\mE \Big(\|u_{0}\|_{L^{p}}^{p} +\int_{0}^t\| u_{r}\|_{L^{p-2}}^{p-2}\dif r\Big), \forall p\in 2\mN.
\end{eqnarray}
{ Setting $p=2$ in  the above, we get}
\begin{eqnarray}
\label{p0202-3}
&& \mE \Big(\|u_{t}\|_{L^{2}}^{2}+\int_{0}^t\|u_{r}\|_{L^{2}}^{2}\dif r\Big)
\leq C\mE \Big(\|u_{0}\|_{L^{2}}^{2} + t\Big).
\end{eqnarray}
Thus, by iterations,  it holds that
{\baee
&\mE \Big(\|u_{t}\|_{L^{p}}^{p}+\int_{0}^t\|u_{r}\|_{L^{p}}^{p}\dif r\Big)\\
\leq& C_{p}\mE \Big(\|u_{0}\|_{L^{p}}^{p}+\|u_{0}\|_{L^{p-2}}^{p-2} +\int_{0}^t\|u_{r}\|_{L^{p-2}}^{p-2}\dif r\Big)\\
\leq&  C_{p} \mE\Big(\|u_{0}\|_{L^{p}}^{p} +1+\int_{0}^t\|u_{r}\|_{L^2}^{2}\dif r\Big)\\
\leq&  C_{p} \Big(\|u_{0}\|_{L^{p}}^{p}+t),\quad \forall t\geq 1.
\eaee
}

Now, let us     prove \eqref{3030-1}.
By \eqref{pp0202-1}  and Young's inequality, we arrive at
\begin{eqnarray*}
\begin{split}
	& \|u_{t}\|_{L^{p}}^{p}+\int_{0}^t\|u_{r}\|_{L^{p}}^{p}\dif r
	\leq  C_{1,p}  \sum_{k\in\cZ_{0}}b_{k}\int_{0}^{t}\langle u_{r}^{p -1}, e_{k}\rangle \dif W_{k}(r)+C_{2,p} t.
\end{split}
\end{eqnarray*}
By  the above inequality,  H\"older's inequality, Burkholder-Davis-Gundy's inequality  and (\ref{L-8-62}), for any $\mathcal E>0, t\geq 1$ and $K\geq 1$ ,  one arrives at that
\begin{eqnarray*}
&&  \mP\Big( \|u_t\|^p_{L^p}+ \int_0^t  \|u_r\|_{L^p}^p \dif r  - \mathcal (\mathcal E+C_{2,p})  t \geq K
\Big)
\\ &&\leq  \mP\Big(C_{1,p}  \sum_{k\in\cZ_{0}}b_{k}\int_{0}^{t}\langle u_{r}^{p -1}, e_{k}\rangle \dif W_{k}(r)\geq \mathcal E   t+K  \Big)
\\ &&\leq \frac{ C_p \mE \Big[  \big(\sum_{k\in\cZ_{0}}b_{k}\int_{0}^{t}\langle u_{r}^{p -1}, e_{k}\rangle \dif W_{k}(r) \big)^{100}   \Big]}{(\mathcal E    t+K)^{100}  }
\\ &&\leq \frac{C_p \sum_{k\in \cZ_0}\mE  \Big( \int_0^t \langle u_{r}^{p -1}, e_{k}\rangle^2  \dif r\Big)^{50} }{(K+\mathcal E  t)^{100} }
\leq \frac{C_p \mE  \Big( \int_0^t \| u_{r}\|_{2p-2}^{2p-2} \dif r\Big)^{50} }{(K+\mathcal E  t)^{100} }
\\ && \leq \frac{C_p \mE  \Big[ \big(\int_0^t 1^{50/49}\dif r\big)^{49}
	\big(\int_0^t \| u_{r}\|_{2p-2}^{50(2p-2)} \dif r\big)\Big]   }{(K+\mathcal E  t)^{100} }
\\ &&\leq  \frac{C_p t^{49} \mE   \int_0^t \| u_{r}\|_{50(2p-2)}^{50(2p-2)} \dif r    }{(K+\mathcal E  t)^{16} }
\\ && \leq  \frac{C_p t^{49}  (t+\|u_0\|_{L^{100p-100}}^{100p-100})   }{(K+\mathcal E  t)^{100} }\leq  \frac{C_p t^{49}  (t+\|u_0\|_{L^{100p}}^{100p})   }{(K+\mathcal E  t)^{100} }.
\end{eqnarray*}
Setting $\mathcal E=C_{2,p}$ in the above, it yields the desired result (\ref{3030-1}) for $\mathcal E_p=2C_{2,p}$. The proof is complete.

\end{proof}

}

\begin{lemma}\label{div reg}
For any  $n\geq  \nn= \lfloor d/2+1\rfloor$,
there exist $m_n>\kappa_n>0$ depending on
$n,d,\Bbbk$ such that the following
\begin{align}\label{div}
\| \operatorname{div}A(u) \|_{n-2}^2\leq \eps \|u\|_{n}^2+C_{\eps,n} (\|u\|_{L^{m_n}}^{\kappa_n}+\|u\|_{L^{m_n}}^{m_n} )
\end{align}
holds for any  $\eps>0$ and $u\in H^n$,
where   $C_{\eps,n}>0$ is a constant  depending  on $\epsilon, n, d,\Bbbk
$ and $(c_{\mathbbm{i},\mathbbm{j}})_{1\leq \mathbbm{i}\leq d,0\leq \mathbbm{j} \leq \Bbbk}.$
Furthermore, one  can  also assume that
{\begin{eqnarray}
\label{14-1}
  m_n\leq  16 d  n \Bbbk+4n^2.
\end{eqnarray}}
\end{lemma}
\begin{proof}
{ For the case $n=1$, (\ref{div})--(\ref{14-1}) hold obviously and we omit the details.}
Therefore, we always assume $n\geq \max\{\nn,2\}$.

By direct calculation, we conclude that
\begin{align}
\nonumber &\| \operatorname{div}A(u) \|_{n-2}^2
\\ 	\nonumber
\leq&  \sum_{j=1}^d\| A_j(u) \|_{n-1}^2
\\		\nonumber
\leq &  C  \sum_{j=1}^\Bbbk	\sum_{\alpha=(\alpha_1,\cdots,\alpha_{n-1}) \in \mathcal{G}_{n-1}  } \|\partial_{x_{\alpha_1}} \cdots \partial_{x_{\alpha_{n-1}}} u^j\|^2
\\ 	\nonumber
\leq& C  \sum_{j=1}^\Bbbk
\sum_{a=(a_1,\cdots,a_{n-1}) \in \mathcal{H}_j  }  \int_{\mT^d }\prod_{k=1}^{n-1}  |\nabla^k u|^{2a_k}\cdot |u |^{2j-2\sum_{k=1}^{n-1}  a_k }\dif x
\\  \label{B3}
:=& C  \sum_{j=1}^\Bbbk
\sum_{a=(a_1,\cdots,a_{n-1}) \in \mathcal{H}_j  }  I_{j,a} ,
\end{align}
where for any $k,j\geq 1$
\begin{eqnarray*}
&& \mathcal G_{k}:=\{\alpha=(\alpha_1,\cdots,\alpha_k): 1\leq \alpha_i \leq d ,\alpha_i\in \mN \},\\
&&\mathcal{H}_j=\{ a=(a_1,\cdots,a_{n-1}): 0\leq a_k\leq {n-1},\forall 1\leq k\leq {n-1}, \sum_{k=1}^{n-1} ka_k={n-1},\sum_{k=1}^{n-1} a_k\leq j  \}.
\end{eqnarray*}

For any
{ $1\leq j\leq\Bbbk, a=(a_1,\cdots,a_{n-1}) \in \mathcal{H}_j$,} we will give  an estimate of $I_{j,a}$ in the next. By the definition of $\mathcal{H}_j$, at least one of { $a_k,1\leq k\leq  n-1$} is positive.  If $a_k\geq 1$, set
\begin{eqnarray*}
p_k=\frac{n}{ka_k}.
\end{eqnarray*}
If $a_k= 0$,
set
\begin{eqnarray*}
p_k=\infty.
\end{eqnarray*}
In this case, we  adopt the notation $0\cdot \infty=0$,  $r^0=1, \forall r\in \mR $ and
$\|f\|_{L^0}=1$  for any  function $f$ on $\mT^d.$

Observe that
\begin{equation*}
\sum_{k=1}^{n-1}\frac{1}{p_k}+\frac1n=\sum_{k=1}^{n-1}\frac{ka_k}{n}
+\frac1n=\frac{n-1}{n}+\frac1n=1.
\end{equation*}
Thus, by H\"older's inequality, 	we have
\bae\label{I_ja}
I_{j,a}=&  \int_{\mT^d }\prod_{k=1}^{n-1}  |\nabla^k u|^{2a_k}\cdot |u |^{2j-2\sum_{k=1}^{n-1} a_k }\dif x\\
\leq &  C
\prod_{k=1}^{n-1} \big\|\nabla ^k u\big\|_{L^{2a_kp_k}}^{2a_k}
\cdot \big\|  |u |^{j-\sum_{k=1}^{n-1} a_k } \big\|_{L^{2n}}^2.
\eae
Set $q=16dn\sum_{k=1}^{n-1}a_k\geq 16dn>1$
{  and
\begin{eqnarray*}
	\lambda_k= 1-\frac{q(2 n-d)}{q(2 n-d)+2 d } \cdot \frac{n-k}{n}\in (\frac{k}{n},1),~ 1\leq k\leq n-1.
	\end{eqnarray*}}
	With the help of  $p_k=\frac{ n}{a_k k},n\geq \frac{d+1}{2}$ and $q>0$,
	for any $1\leq k\leq n-1$ with $a_k>0$, it holds that
	$$
	\begin{aligned}
& \frac{2a_k p_k(kq+d)-dq}{[q(2n-d)+2d]a_k p_k} \\
&= \frac{2 k q+2 d}{q(2 n-d)+2 d}-\frac{ d q}{q(2 n-d)+2 d}
\cdot  \frac{k}{ n} \\
&= \frac{2 k q+2 d-q \frac{k d}{n}}{q(2 n-d)+2 d} \\
&= 1-\frac{q(2 n-d)}{q(2 n-d)+2 d }\cdot  \frac{n-k}{n}
\\ & =\lambda_k.
\end{aligned}
$$
Thus,  by Gagliardo-Nirenberg's  inequality,
we get
\bae\label{gn}
\|\nabla ^k u\|_{L^{2a_kp_k}}\leq C \|  u\|_n ^{\lambda_k}\| u\|^{1-\lambda_k}_{L^q}, \quad \forall 1\leq k\leq n-1 \text{ with } a_k>0.
\eae
%

Substituting \eqref{gn} into \eqref{I_ja}, one gets
\bae\label{Ija final}
I_{j,a} \leq & C
\prod_{k=1}^{n-1}
\|u \|_{n}^{2a_k \lambda_k}
\| u\|_{L^q}^{2a_k(1-\lambda_k)}\cdot \|  |u |^{j-\sum_{k=1}^{n-1} a_k } \|_{L^{2n}}^2
\\ \leq &  C\|u\|_{n}^{\sum_{k=1}^{n-1} 2a_k \lambda_k }
\| u\|_{L^q}^{\sum_{k=1}^{n-1 }2a_k(1-\lambda_k)}\cdot \|  |u |^{j-\sum_{k=1}^{n-1} a_k } \|_{L^{2n}}^2.
\eae
By direct calculations, $q=16dn\sum_{k=1}^{n-1}a_k\geq 16dn$ and $2n-d\geq 1,$  we see that
\baee
\bar\lambda:= &\sum_{k=1}^{n-1}2 a_k \lambda_k \\
= & \sum_{k=1}^{n-1}2 a_k\Big(1-\frac{q(2 n-d)}{q(2 n-d)+2 d } \cdot \frac{n-k}{n}\Big)
\\
=&  \sum_{k=1}^{n-1}2 a_k\Big(\frac{2d}{q(2 n-d)+2 d }+
\frac{q(2 n-d)k}{q(2 n-d)n+2 dn } \Big)
\\
=&  \frac{4d\sum_{k=1}^{n-1} a_k}{q(2 n-d)+2 d }+
\frac{2q(2 n-d)\sum_{k=1}^{n-1}k a_k}{q(2 n-d)n+2 dn }
\\
=&  \frac{4d\sum_{k=1}^{n-1} a_k}{q(2 n-d)+2 d }+
\frac{2q(2 n-d)(n-1)}{q(2 n-d)n+2 dn }
\\
< & \frac{4d\sum_{k=1}^{n-1} a_k}{( 16dn\sum_{k=1}^{n-1}a_k)\cdot  (2 n-d) }+
\frac{2(n-1)}{n}
<2-\frac{1}{n}.
\eaee
Substituting $\bar\lambda= \sum_{k=1}^{n-1}2 a_k \lambda_k $  into \eqref{Ija final},
also with the help of  $\bar\lambda\in (0,2)$
and Young's  inequality, for any $\eps>0$,
one has
\bae
\label{14-2}
I_{j,a} \leq  &\eps \|u\|_{n}^{2}+C_{\eps,n}
\Big( \| u\|_{L^q}^{2 \sum_{k=1}^{n-1}ka_k-\bar\lambda}\cdot \|  |u |^{j-\sum_{k=1}^{n-1} a_k } \|_{L^{2n}}^2\Big)^{2/(2-\bar\lambda)}.
\eae
Combining the above with (\ref{B3}), also noticing
that
\begin{equation*}2\sum_{k=1}^{n-1}ka_k= 2(n-1)\geq 2>\bar \lambda
\text{ and }j\geq \sum_{k=1}^{n-1} a_k, \quad \forall  a\in \mathcal{H}_j,
\end{equation*}
the proof of   (\ref{div})    is complete.	


Observe  that,   in   (\ref{14-2}),  it has  $1\leq j\leq \Bbbk,$
$q=16dn\sum_{k=1}^{n-1}a_k\in [16dn,16d n \Bbbk] $, $2-\bar{\lambda}\geq \frac{1}{n}$
and $\sum_{k=1}^{n-1}ka_k=n-1.$
So after some simple calculations, one arrives at  (\ref{14-1}).

\end{proof}

\begin{lemma}\label{priori H^m}
{
For any { $T\geq 1$} and $n\geq \nn$, let $(u_t)_{ t\geq 0}\in C([0,T], H^n )$ be a   solution of \eqref{1-1}. Then,
there exists a    $m\in \big(0, 16dn\Bbbk+4n^2\big)$  depending on  $n,d,\Bbbk$
such that
\bae\label{proiri H^m regu}
& \mathbb{E}\Big[ \sup _{0 \leq t \leq T}\|u_t\|_{{ n-1}}^2+\nu\int_0^T \|u_s\|_{{ n}}^2\dif s \Big]\leq \|u_0\|_{{ n-1}}^2+ { C\big(T+\left\|u_0\right\|_{L^m}^m\big) },
\eae
where $C$ is a constant depending on $n,\nu,d,\Bbbk, (b_{i})_{i\in \cZ_0},\mathbb{U}$ and $
(c_{\mathbbm{i},\mathbbm{j}})_{1\leq \mathbbm{i}\leq d,0\leq \mathbbm{j} \leq \Bbbk}$.
}
\end{lemma}
\begin{proof}
With the help of   Proposition \ref{wp},  using  the It\^{o}'s
formula  for  $\| u_t\|_{ n-1}^2$, we  get
\bae\label{H^m}
&\frac12\dif  \langle (-\Delta )^{( n-1)/2} u_t, (-\Delta )^{( n-1)/2} u_t   \rangle\\
=&  -\nu \big \langle  (-\Delta )^{(n+1)/2} u_t, (-\Delta )^{(n-1)/2} u_t  \big \rangle\dif t
-\big\langle (-\Delta )^{(n-1)/2}\text{div} A(u_t ), (-\Delta )^{(n-1)/2} u_t  \big \rangle\dif t\\
&+\sum_{i\in \cZ_{0}}b_i^2 \|e_i\|_{n-1}^2\dif t+d M_n(t)\\
=&  -\nu\|u_t\|_{n}^2\dif t
-\big\langle (-\Delta )^{(n-2)/2}\text{div} A(u_t ), (-\Delta )^{n/2} u_t   \big\rangle\dif t	\\
&+\frac{1}{2}\sum_{i\in \cZ_{0}}b_i^2 \|e_i\|_{n-1}^2\dif t+d M_n(t)\\
\leq&  -\nu\|u_t\|_{n}^2\dif t
-\| \text{div} A(u_t )\|_{n-2} \|  u_t \|_{n}\dif t	+\frac{1}{2}\sum_{i\in \cZ_{0}}b_i^2 \|e_i\|_{n-1}^2\dif t+d M_n(t),
\eae
where
\begin{equation*}
M_n(t):= \sum_{i\in \cZ_0}b_i\int_0^t \big\langle  (-\Delta )^{(n-1)/2} e_i  , (-\Delta )^{(n-1)/2} u_s   \big\rangle \dif W_i(s).
\end{equation*}
Integrating  (\ref{H^m}) from $0$ to $t$,  with the help of
 Lemma \ref{div reg}, for any $t\in [0,T]$, we   obtain
\bae\label{pre H^m}
\frac12 \| u_t\|_{n-1}^2+\frac12\nu\int_{0}^t\|u_s\|_{n}^2\dif s\leq \frac12 \| u_0\|_{n-1}^2+ C \int_{0}^t\|u_s\|_{L^m}^m\dif s+C t+\sup_{s\in [0,t]}M_n(s),
\eae
where $m\in \big(0, 16dn\Bbbk+4n^2\big)$    is a constant depending on  $n,d,\Bbbk.$
With the help of  Lemma \ref{L^p 2.2}, one has
\begin{eqnarray*}
&&   \mathbb E \sup_{t\in [0,T]}M_n(t)^2\leq C \mathbb E\int_0^T \sum_{i\in \cZ_{0}}  b_i^2\|e_i\|_{2n-2}^2\|u_t\|^2\dif t
\\ && \leq C \mathbb E \int_0^T \|u_t\|^2 \dif t\leq C (T+\|u_0\|^2).
\end{eqnarray*}
Thus, by Lemma \ref{L^p 2.2},
(\ref{pre H^m}) implies
the desired estimate \eqref{proiri H^m regu}.
The proof is complete.
\end{proof}


{
\begin{lemma}
\label{p17-4}
Recall  that $\mmm=40 \Bbbk d (d+14\Bbbk)^2$.
There exists a $\kappa_0$ only depending on $\nu,d,\Bbbk, (b_{i})_{i\in \cZ_0},\mathbb{U}$ such that
for any $\kappa\in (0,\kappa_0],$
$u_0\in H^{\nn+5}$ and $n\in \mN,$
it holds that
\begin{eqnarray}
	\label{74-1}
	\begin{split}
		& \mE \Big[ \exp\big\{\kappa \sum_{i=1}^n  \|u_i\|_{\nn+5}^2-K_\kappa\int_0^n \|u_s\|_{L^\mm}^\mm\dif s-K_\kappa n\big\} \Big]
		\\ & \leq  \exp\{a  \|u_0\|_{\nn+5}^2
		\},
	\end{split}
\end{eqnarray}
where $K_\kappa,a$
are some constants depending on
$\kappa$ and  $\nu,d,\Bbbk, (b_{i})_{i\in \cZ_0},\mathbb{U}$.
%
\end{lemma}
}
{
\begin{proof}
{Obviously, there exists  a  $\alpha=\alpha(\nu,d,\Bbbk, (b_{i})_{i\in \cZ_0},\mathbb{U})$ such  that
\begin{align}
	\label{p11p11-1}
	\frac12\nu\|u_t\|_{\nn+6}^2\geq
	\alpha \sum_{i\in \cZ_0} b_i^2\big\langle  (-\Delta )^{\nn+5} e_i  ,  u_t   \big\rangle^2.
\end{align}
At the beginning, we demonstrate the following  claim.

\textbf{Claim}.  For any  $\kappa<\frac12 \alpha,$ there exists a $\cC_\kappa>0$ such that
\begin{eqnarray}\label{induc expo}
\begin{split}
&\mE    \exp\Big\{ \kappa  \|u_t\|^2_{\nn+5}-\cC_\kappa \int_{0}^t \|u_s\|_{L^\mm}^\mm \dif s-\cC_\kappa \Big\}
\\ & \leq \exp\{ e^{- \nu  t} \kappa  \|u_{0}\|_{\nn+5}^2\},~~\forall u_0\in H^{\nn+5} \text{ and }t\geq 0.
\end{split}
\end{eqnarray}
In the first, we prove the
above Claim  for any $u_0\in H^{\nn+6}.$
 Using  the It\^{o}'s
formula  for  $\| u_t\|_{ n+5}^2$
through  similar arguments as that in   \eqref{H^m} and with the help of  Lemma \ref{div reg}, one arrives at   that
\bae\label{hn ito}
\dif \|u_t\|_{\nn+5}^2 \leq&  -\frac{3}{2}\nu\|u_t\|_{\nn+6}^2\dif t
+C (1+ \|  u_t \|_{L^\mm}^\mm)\dif t	+\sum_{i\in \cZ_{0}}b_i^2 \|e_i\|_{\nn+5}^2\dif t+2d M(t)\\
\leq&  -\frac{3}{2}\nu\|u_t\|_{\nn+6}^2\dif t	+C  \|  u_t \|_{L^\mm}^\mm\dif t+C\dif t+2d M_t\\
\eae
with
\begin{equation*}
	M_t:=(-1)^{\nn+5} \sum_{i\in \cZ_0}b_i\int_0^t \big\langle  (-\Delta )^{\nn+5} e_i  ,  u_s   \big\rangle \dif W_i(s).
\end{equation*}
%
%
%
Using (\ref{p11p11-1})--\eqref{hn ito} and the fact that $\|u_t\|_{\nn+6}\geq \|u_t\|_{\nn+5} $, we have
\bae\label{228}
\|u_t\|^2_{\nn+5}\leq& e^{-\nu t}\|u_{0}\|^2_{\nn+5}+C\int_{0}^t e^{-\nu (t-s)} \|u_s\|_{L^\mm}^\mm\dif s+C+ 2\int_{0}^t e^{-\nu (t-s)}\dif N_s,
\eae
where
$$N_s:=-\frac{\alpha}{2}
\int_0^s  \sum_{i\in \cZ_0} b_i^2\big\langle  (-\Delta )^{\nn+5} e_i  ,  u_r   \big\rangle^2 \dif r
+M_s.$$
From \cite[Lemma A.1]{Mat02},  we conclude that
\begin{equation}\label{OU expo mart}
	\begin{aligned}
		&\mP\Big\{
		\|u_t\|_{\nn+5}^2-
		e^{-\nu t}\|u_{0}\|^2_{\nn+5}-C\int_{0}^t e^{-\nu (t-s)} \|u_s\|_{L^\mm}^\mm\dif s-C\geq \frac{2K}{\alpha}
		\Big\}\\
		=&\mP\Big\{ \int_0^t e^{-\nu(t-s)}\dif N_s\geq \frac{K}{\alpha}  \Big\}\leq e^{-K}, ~~\forall K\geq 0.
	\end{aligned}
\end{equation}
Note now that if a random variable $X$ satisfies $\mP(X \geq  K) \leq \frac{1}{K^2}$
for all $K \geq 0,$
then $EX \leq  2.$
Thus, for any $\kappa\leq \frac{\alpha}{2}$, by (\ref{OU expo mart}), one has
\begin{eqnarray*}
	\mE \exp\{\kappa \|u_t\|_{\nn+5}^2-\kappa C \int_{0}^t e^{-\nu (t-s)} \|u_s\|_{L^\mm}^\mm\dif s  \}
	\leq C_\kappa \exp\{\kappa  e^{-\nu t}\|u_0\|^2_{\nn+5}\}.
\end{eqnarray*}
The above   implies that  Claim   (\ref{induc expo}) holds for any $u_0\in H^{\nn+6}.$

Now, we prove  the
Claim  (\ref{induc expo})  for  $u_0\in H^{\nn+5}.$
For any $u_0'\in H^{\nn+6}$,$ N\in \mN$
and $\kappa<\frac12 \alpha,$
since
(\ref{induc expo}) holds  for any   $u_0'\in H^{\nn+6},$
one has
\bae\label{pre exp hn}
&\mE \Big[ \exp\Big\{\kappa \|P_N u_t^{u_0'}\|_{\nn+5}^2 -
\cC_\kappa\int_0^t \|u_s^{u_0'}\|_{L^\mm}^\mm\dif s-\cC_\kappa \Big\}\Big]\leq  \exp\{e^{- \nu  t} \kappa   \|u_0'\|_{\nn+5}^2\}.
\eae
Noticing  the facts  that
\begin{eqnarray*}
	&& \Big|\|u_s^{u_0}\|_{L^\mm}^\mm-\|u_s^{u_0'}\|_{L^\mm}^\mm\Big|\leq C \|u_s^{u_0}-u_s^{u_0'}\|_{L^1} (\|u_s^{u_0}\|_{\nn}^{\mm-1}+\|u_s^{u_0'}\|_{\nn}^{\mm-1}+1),
	\\ &&
	\|P_N u_t^{u_0}-P_N u_t^{u_0'}\|_{\nn+5}
	\leq C_N \|u_t^{u_0}-u_t^{u_0'}\|_{L^1},
\end{eqnarray*}
letting    $u_0'\in H^{\nn+6}$ and $u_0'\rightarrow u_0$ in $H^{\nn+5}$ in \eqref{pre exp hn}, also  with the hlep of  Fatou's lemma and Lemma \ref{path L1 contr},  we  get
\baee
&\mE \Big[ \exp\Big\{\kappa \|P_N u_t^{u_0}\|_{\nn+5}^2 -\cC_\kappa\int_0^t \|u_s^{u_0}\|_{L^\mm}^\mm\dif s-\cC_\kappa \Big\}\Big]
\\ & \leq \exp\{\kappa e^{-\nu t}  \|u_0\|_{\nn+5}^2\}, ~\forall u_0\in H^{\nn+5}.
\eaee
In the above, letting $N\rightarrow \infty,$ we obtain the desired result (\ref{induc expo}).

In the end, we demonstrate a proof of (\ref{74-1}).
Set $c=\sum_{n=0}^\infty e^{-\nu n}.$
For any $\kappa\leq \frac{\alpha }{2c}$,
by  (\ref{induc expo}), one has
\begin{eqnarray*}
	& & \mE \Big[  \exp\Big\{ \kappa \sum_{i=1}^n \|u_i\|^2_{\nn+5}-\cC_{c\kappa}  \int_{0}^n \|u_s\|_{L^\mm}^\mm \dif s-\cC_{c\kappa } n  \Big\} \Big]
	\\ && =\mE \Big[  \exp\Big\{ \kappa \sum_{i=1}^n \|u_i\|^2_{\nn+5}-\cC_{c\kappa}  \int_{0}^n \|u_s\|_{L^\mm}^\mm \dif s-\cC_{c\kappa}  n  \Big\}\mid \mathscr F_{n-1}  \Big]
	\\ && \leq  \mE  \exp\Big\{ \kappa \sum_{i=1}^{n-1} \|u_i\|^2_{\nn+5}+\kappa e^{-\nu}\|u_{n-1}\|^2-\cC_{c\kappa}  \int_{0}^{n-1}  \|u_s\|_{L^\mm}^\mm \dif s-\cC_{c\kappa} (n-1)   \Big\}
\end{eqnarray*}
Applying this procedure repeatedly, one sees that  (\ref{74-1}) holds with $a=c \kappa$ and  $K_\kappa=\cC_{c\kappa}.$  

}

\end{proof}
}

{
}

		\subsection{ Well-posedness of SVSCL and Markov property of the semigroup $P_t$ }
		\label{2-2}

		Recall that
		$A(u)=\big(A_1(u),\cdots,A_d(u)\big)$ and  $
		A_i(u)=\sum_{j=0}^{\Bbbk}c_{i,j}u^j,i=1,\cdots,d.
		$
		
		We first define the local solution of \eqref{1-1}.
		For $k\in\mathbb N$ and $1\leq i\leq d$, let  $A_i^{(k)} \in C_c^{\infty}(\mathbb{R}^d,\mathbb R)$ be a function such that
		\begin{equation*}
			A_i^{(k)}(x)\Big|_{[-k, k]}=A_i(x).
		\end{equation*}
		For   $n\geq \nn=\lfloor  d/2+1\rfloor$, by \cite[Theorem 2.1]{H13}, with regard to the following equation
		\begin{eqnarray}\label{smooth trun}
			\left\{\begin{aligned}
				&\dif u^{(k)}+\operatorname{div} A^{(k)}\big(u^{(k)}\big) \dif t=\nu \Delta u^{(k)} \dif t+\dif \eta_t,\\
				& u_0^{(k)}=u_0 \in H^n,
			\end{aligned}\right.
		\end{eqnarray}
		there is a unique  solution $\{u^{(k)}_t\}_{t\geq0}\in C([0, \infty), H^n).$
		By the Sobolev embedding theorem, we have the continuous inclusion
		\bae\label{L inf incl}
		H^n\rightarrow L^\infty(\mathbb T^d).
		\eae
		For $|x|\leq l\wedge k$, it holds that
		\begin{align}  \label{sl}
			A_i^{(k)}(x)=  A_i^{(l)}(x)=A_i(x), ~i=1,\cdots,d. \end{align}
		Define the  following stopping times
		\begin{equation*}
			\tau^l_{u_0}=\inf \big\{t>0,\big\|u_t^{(l)}\big\|_{L^{\infty}}>l\big\}, \quad \forall l\in \mN.
		\end{equation*}
		By \eqref{sl},
		for every $0\leq t \leq \tau^k_{u_0} \wedge \tau^l_{u_0}$
		and  smooth function  $\phi$  on $\mT^d$ with $\int_{\mT^d} \phi(x)\dif x=0$,
		we have
		\baee
		\langle u_t^{(k)},\phi\rangle= &  \langle u_0,\phi\rangle   +
		\int_{0}^t \Big(\nu \langle   u_s^{(k)},\Delta \phi\rangle  +\sum_{i=1}^d \langle A^{(k)}(u_s^{(k)}),\partial_{x_i}\phi \rangle\Big)  \dif s+\langle \eta(t),\phi\rangle\\
		= &   \langle u_0,\phi\rangle   +
		\int_{0}^t \Big(\nu \langle   u_s^{(k)},\Delta \phi\rangle  +\sum_{i=1}^d \langle A^{(l)}(u_s^{(k)}),\partial_{x_i}\phi \rangle\Big)  \dif s+\langle \eta(t),\phi\rangle.
		\eaee
		By the uniqueness of the  solution of \eqref{smooth trun}, we conclude that
		\begin{eqnarray}
			\label{pp25-2}
			u_t^{(k)}=u_t^{(l)}, \quad \text {for } 0\leq t \leq \tau^k_{u_0} \wedge \tau^l_{u_0}.
		\end{eqnarray}
		For any $k<l$, assume that  $\tau_{u_0}^k> \tau_{u_0}^l$. First, by the definition of $\tau_{u_0}^l$, it holds that
		\begin{eqnarray}
			\label{pp25-1}
			\sup_{s\in [0,\tau_{u_0}^l]}\|u_s^{(l)}\|_{L^\infty}\geq l.
		\end{eqnarray}
		Secondly, by the definition of  $\tau^k_{u_0}$, also with the help of   $\tau^k_{u_0}> \tau^l_{u_0}\geq 0$, we get
		\begin{eqnarray*}
			\sup_{s\in [0,\tau^l_{u_0}]}\|u_s^{(k)}\|_{L^\infty}\leq  \sup_{s\in [0,\tau^k_{u_0}]}\|u_s^{(k)}\|_{L^\infty}\leq k.
		\end{eqnarray*}
		By  (\ref{pp25-2}) and  the assumption $\tau^k_{u_0}>\tau^l_{u_0}$,   the above inequality conflicts with (\ref{pp25-1}).
		Thus, for any   $k<l$, one has    $\tau^k_{u_0}\leq  \tau^l_{u_0}$.
		Define
		\begin{eqnarray}\label{maximal sol}
			\tau_{u_0}:=\sup_{k\in \mathbb N}\tau^k_{u_0}
		\end{eqnarray}
		and
		\begin{eqnarray}\label{local sol}
			u_t:=u_t^{(k)}, \quad 0\leq  t<\tau^k_{u_0}.
		\end{eqnarray}
		Thus,  for the equation (\ref{1-1}), we define a local solution  $u \in C([0, \tau_{u_0}), H^n)$   by the way above.
		
		For $u_0 \in H^\nn$, we have the following lemma for the corresponding local solution.
		\begin{lemma}\label{local nth dir}
			For any { $T\geq 1$}, let $(u_t)_{ t\geq 0}\in C([0,\tau_{u_0}\wedge T), H^\nn )$ be the local solution of \eqref{1-1}. One has
			\baee\label{223}
			& \mathbb{E}\Big[\sup _{0 \leq t \leq \tau_{u_0}\wedge T}\|u_t\|_{{ \nn}}^2\Big]\leq \|u_0\|_{{ \nn}}^2+ { C\big(T+\left\|u_0\right\|_{L^\mmm}^\mmm\big) }.
			\eaee
		\end{lemma}
		\begin{proof}
			{
				First, for $u_0^\prime\in H^{\nn+1}$ and $t\in [0,\tau_{u_0'}\wedge T).$ Using similar arguments as that in the proof of   Lemma \ref{priori H^m} and noticing the expression of $\mmm$, we get
				\bae\label{226n}
				& \mathbb{E}\Big[\sup _{0 \leq t \leq \tau_{u_0'}\wedge T}\|u_t^{u_0'}\|_{{ \nn}}^2\Big]\leq \|u_0'\|_{{ \nn}}^2+ { C\big(T+\left\|u_0'\right\|_{L^\mmm}^\mmm\big) }.
				\eae


				Secondly, for any $u_0\in H^{\nn},u_0'\in H^{\nn+1},$ $t\in[ 0, \tau_{u_0}\wedge \tau_{u_0\prime}\wedge T)$ and $N\in \mN$, we have
				\bae\label{l1 hn}
				&\|P_Nu_t^{u_0}\|_{\nn}^2\\
				\leq &  \|P_Nu_t^{u_0}-P_N u_t^{u_0'} \|_{\nn}^2+\|P_Nu_t^{u_0'}\|_{\nn}^2\\
				\leq&  C \sum_{|k|\leq N} \langle \partial_x^\nn (u_t^{u_0}-u_t^{u_0'}),e_k\rangle^2+\|u_t^{u_0'}\|_{\nn}^2\\
				\leq&  C_N  \sum_{|k|\leq N} \|  u_t^{u_0}-u_t^{u_0'}\|_{L^1}^2\|e_k\|_{L^\infty}^2+\|u_t^{u_0'}\|_{\nn}^2\\
				\leq& C_N \|u_0-u_0'\|^2_{L^1}+\|u_0'\|_{{ \nn}}^2+  C\int_0^t \|u_s^{u_0'}\|_{L^\mmm}^\mmm\dif s+ C t
				\\   &+ \sum_{i\in \cZ_0}|b_i| \sup_{t\in [0,T]}\big|\int_0^t \langle e_i, (-\Delta )^{\nn-1 }u_s^{u_0'}\rangle\dif W_i(s)\big| .\\
				\eae
				Where the last line above used  Lemma \ref{path L1 contr} and \eqref{pre H^m}. Substituting the estimates in Lemma \ref{priori H^m} into the inequality above, we conclude that
				\bae\label{new 232}
				&\mathbb E \Big[\sup_{t\in [0,\tau_{u_0}\wedge \tau_{u_0\prime}\wedge T)} \|P_Nu_t^{u_0}\|_{\nn}^2\Big]\\
				\leq& C_N \|u_0-u_0'\|^2_{L^1}+\|u_0'\|_{{ \nn}}^2+ { C\big(T+\left\|u_0'\right\|_{L^\mmm}^\mmm\big) }, \\
				\eae
				for any $N\in \mN$.

				Next, we will prove
				\begin{equation*}
					\mP(\tau_{u_0^\prime}=\infty)=1, \quad \forall u_0^\prime\in H^{\nn+1}.
				\end{equation*}
				For any $M>0$ and $k\in \mN$, using \eqref{226n},  we have
				\begin{eqnarray*}
					&& \mP(\tau_{u_0^\prime}<M)\leq  \mP(\tau^k_{u_0^\prime}<M )
					\\ && \leq \mP\Big(\sup_{s\in [0,\tau^k_{u_0^\prime})} \|u_s^{(k)}\|_{L^\infty}\geq k\text{ and }\tau^k_{u_0^\prime}<M\Big)
					\\ && \leq \frac{1}{k^2}  \mE \Big(\sup_{s\in [0,\tau^k_{u_0^\prime}\wedge M)} \|u_s^{(k)}\|_{L^\infty}^2 \Big)
					\\ && \leq \frac{C}{k^2}\big(\|u_0'\|_{{ \nn}}^2+ { C\big(M+\left\|u_0'\right\|_{L^\mmm}^\mmm\big) }\big).
				\end{eqnarray*}
				Letting $k\rightarrow \infty$ in the above, we conclude that
				\begin{equation*}
					\mP(\tau_{u_0^\prime}<M)=0.
				\end{equation*}
				Since $M>0$ is arbitrary,  we get $\mP(\tau_{u_0^\prime}=\infty)=1$ and $u_t^{u_0^\prime}\in C([0,\infty),H^{\nn})$ for $u_0^\prime \in H^{\nn+1}.$

				In \eqref{new 232},	
letting  $u_0'\in H^{\nn+1}$ and $u_0' \rightarrow u_0$ in $H^\nn$,
also with the help of $\mP(\tau_{u_0^\prime}=\infty)=1$,  we get
				\baee
				\mathbb E \Big[\sup_{t\in [0,\tau_{u_0}\wedge T)} \|P_Nu_t^{u_0}\|_{\nn}^2\Big] \leq \|u_0\|_{{ \nn}}^2+ { C\big(T+\left\|u_0\right\|_{L^\mmm}^\mmm\big) }.
				\eaee
				Letting $N\rightarrow \infty$ in the above and using monotone convergence theorem, we get
				\begin{eqnarray*}
					&& 	\mathbb E \Big[ \lim_{N\rightarrow \infty}\sup_{t\in [0,\tau_{u_0}\wedge T)} \|P_Nu_t^{u_0}\|_{\nn}^2\Big]
					=\mathbb E \Big[ \sup_N \sup_{t\in [0,\tau_{u_0}\wedge T)} \|P_Nu_t^{u_0}\|_{\nn}^2\Big]
					\\ && =\mathbb E \Big[  \sup_{t\in [0,\tau_{u_0}\wedge T)}\sup_N  \|P_Nu_t^{u_0}\|_{\nn}^2\Big]
					\\ && =\mathbb E \Big[  \sup_{t\in [0,\tau_{u_0}\wedge T)}  \|u_t^{u_0}\|_{\nn}^2\Big] \leq \|u_0\|_{{ \nn}}^2+ { C\big(T+\left\|u_0\right\|_{L^\mmm}^\mmm\big) }.
				\end{eqnarray*}
			{The last  inequality  in the above  is exactly the desired  result of  this lemma.
The proof is complete.
				}

			}
		\end{proof}

		\textbf{Now we are in a position to prove Proposition \ref{wp}. }

		\begin{proof}
			For $u_0 \in H^\nn$, the key is to prove that the local solution is global, i.e.,
			\begin{equation*}
				\mP(\tau_{u_0}=\infty)=1,
			\end{equation*}
			For any $M>0$ and $k\in \mN$, using Lemma \ref{local nth dir}, we have
			\begin{eqnarray*}
				&& \mP(\tau_{u_0}<M)\leq  \mP(\tau_{u_0}^k<M )
				\\ && \leq \mP(\sup_{s\in [0,\tau_{u_0}^k]} \|u_s^{(k)}\|_{L^\infty}\geq k\text{ and }\tau_{u_0}^k<M)
				\\  && \leq \mP(\sup_{s\in [0,\tau_{u_0}\wedge M]} \|  u_s \|_{L^\infty}\geq k  )
				\leq k^{-2} \mE \Big[\sup_{s\in [0,\tau_{u_0}\wedge M]} \|  u_s \|^2_{L^\infty} \Big]
				\\ && \leq C k^{-2} \mE \Big[\sup_{s\in [0,\tau_{u_0}\wedge M]} \|  u_s \|^2_{\nn} \Big]
				 \leq C k^{-2}\Big( \|u_0\|_{{ \nn}}^2+ { C\big(T+\left\|u_0\right\|_{L^\mmm}^\mmm\big) } \Big).
			\end{eqnarray*}
			Letting $k\rightarrow \infty$ in the above,  we conclude that
			\begin{equation*}
				\mP(\tau_{u_0}<M)=0.
			\end{equation*}
			Since $M>0$ is arbitrary,  we get $\mP(\tau_{u_0}=\infty)=1.$ The uniqueness of the solution is followed from Lemma \ref{path L1 contr}. The proof is completed.
		\end{proof}

		From Proposition \ref{wp}, we can define the transition semigroup $\left(P_t\right)_{t \geq 0}$ on $C_b\left(H^\nn\right)$ as following
		
		\begin{equation*}
			P_t \varphi\left(u_0\right):=\mathbb{E}_{u_0}[\varphi(u_t)], \quad t \geq 0, \quad u_0 \in H^\nn,
		\end{equation*}
		where $\mathbb{E}_{u_0}$ means that  $u_t$ starts from $u_0\in H^\nn. $
		The following proposition  can be proved by similar  arguments  in \cite[Corollary 1]{MR20}.
		\begin{proposition}
			The family $(P_t)_{t \geq 0}$ is a Feller semigroup and the process $(u_t)_{t \geq 0}$ is a strong Markov process in $H^\nn$.
		\end{proposition}
		
		Take $\varsigma \in \mathcal{P}(H^\nn)$, the dual operator $(P_t^*)_{t \geq 0}$ of $(P_t)_{t \geq 0}$ is defined  by
		\begin{equation*}
			P_t^* \varsigma(O):=\int_{H^\nn} \mathbb{P}_{u_0}(u_t \in \Gamma)  \varsigma(\dif u_0), \quad t \geq 0, \quad O \in \mathcal{B}(H^\nn),
		\end{equation*}
		and the empirical measure of  $(u_t)_{t\geq 0}$ is denoted   by
		\begin{equation*}
			R_T^* \varsigma(O):=\frac{1}{T} \int_0^T P_t^* \varsigma(O) \dif t.
		\end{equation*}

		\subsection{Elements of Malliavin calculus}
		\label{pp24}
		{  Let $\mathbb{U}=|\cZ_0|$ and denote the canonical basis of $\mR^{\mathbb{U}}$ by $\{ \vartheta_{j}\}_{  {j} \in \cZ_0}.$
			We define the linear operator $Q: \mR^{\mathbb{U}}\rightarrow H$ in the following way: for any $z=\sum_{{j} \in \cZ_0}z_{j}\vartheta_{j} \in \mR^{\mathbb{U}}$,
			\begin{eqnarray}
				\label{pp28-1}
				Qz=\sum_{{j} \in \cZ_0 }   b_{{j} } z_{j} e_{{j}}.
		\end{eqnarray}}
		{Without otherwise specified statement, in this section,  we always assume that
			$u=(u_t)_{t\geq 0}$ is the solution of (\ref{1-1}) with initial value $u_0\in H^\nn$.}
		For any $0\le s\le t$ and $\xi\in  H$, let $J_{s,t}\xi$ be the solution of the linearised problem:
		\begin{eqnarray}\label{10-1}
			\left\{
			\begin{split}
				&\partial_t  J_{s,t}\xi + \operatorname{div}A^\prime(u,J_{s,t}\xi)= \nu \Delta J_{s,t}\xi,
				\\  &J_{s,s}\xi=\xi,
			\end{split}
			\right.
		\end{eqnarray}
		where $A^\prime(u,v):= (  A_1^\prime (u) v,\cdots,A_d^\prime (u) v)$.
		For $u\in \mathbb R$, define \begin{equation*}F(u):=-\sum_{i=1}^d\partial_{x_i}A_i(u)=-\sum_{i=1}^dA_i'(u)\frac{\partial u}{\partial x_i}= -\operatorname{div}A(u).\end{equation*}
		For any $0\le t\le T$ and $\xi\in  H$, let~$K_{t,T}$  be the   adjoint of $J_{t,T}.$
		Then,
		\begin{eqnarray*}
			D F(u)v= -\operatorname{div} D A(u)v=-\operatorname{div}  A'(u)v
		\end{eqnarray*}
		and  $\varrho_t:=K_{t,T}\phi$ satisfies the following equation:
		\begin{eqnarray}
			\label{p0203-1}
			\partial_t \varrho_t=-\nu \Delta \varrho_t-(D F(u))^*\varrho_t,
		\end{eqnarray}
		where $ (D F(u))^*$ is the adjoint of $D  F(u)$, i.e, $\langle (D F(u))^*v,w \rangle=\langle v, D F(u)w \rangle.$
		Denote $J^{(2)}_{s,t}(\phi,\psi)$  by the second derivative of $u_t$ with respect to initial value  $u_0$ in the directions of $\phi$ and $\psi$. Then
		\begin{eqnarray}
			\label{0927-1}
			\left\{
			\begin{split}
				&\partial_t J^{(2)}_{s,t}(\phi,\psi)+ \operatorname{div}A^\prime\big(u,J^{(2)}_{s,t}(\phi,\psi)\big)=  \nu \Delta J^{(2)}_{s,t}(\phi,\psi)
				\\ &- \operatorname{div}A^{\prime\prime}\left(u,J_{s,t}
				\phi   J_{s,t}\psi   \right)\quad \text{ for } t>s,
				\\ &J^{(2)}_{s,s}(\phi,\psi)=0,
			\end{split}
			\right.
		\end{eqnarray}
		where $A^{\prime\prime}(u,v):= (  A_1^{\prime\prime} (u) v,\cdots,A_d^{\prime\prime} (u) v)$.
		For the  well-posedness of  equations  \eqref{10-1} and \eqref{0927-1}, one can refer to  \cite{LSU68}.
		
		For any $t>0$ and $v\in L^2([0,t];\R^{\mathbb{U}})$, where $\mathbb{U}=|\cZ_0|$, the Malliavin derivative of~$u_t$ in the direction~$v$ is defined by
		\begin{equation*}
			\cD^v u_t:=\lim_{\eps \to  0}\frac{1}{\e}
			\Big(\Phi(t,u_0,W+\eps \int_0^\cdot v\dd  s)-\Phi(t,u_0,W)\Big), \quad
		\end{equation*}
		where the limit holds almost surely.
		Then, $\cD^v u_s$ satisfies the following equation:
		\begin{eqnarray}
			\label{p17-1}
			\dif \cD^v u_s+ \operatorname{div}A^\prime(u_s ,\cD^v u_s)  = \nu \Delta \cD^v u_s\dif  s +
			Q  \dif \int_0^{s} v_r\dif r,\quad \forall s\in [0,t].
		\end{eqnarray}
		By the Riesz representation theorem, there is a linear operator \baee  \cD:L^2(\Omega,   H)\to L^2(\Omega; L^2([0,t ];\R^{\mathbb{U}})\otimes  H)\eaee such that
		\begin{equation}\label{2.3}
			\cD^v u_t  =\lag    \cD u,v \rag_{L^2([0,t];\R^{\mathbb{U}})},~\forall v\in L^2([0,t];\R^{\mathbb{U}}) .
		\end{equation}
		Actually, we have the following lemma.
		{  \begin{lemma}
				\label{17-1}
				For  any     $v\in L^2([0,t];\R^{\mathbb{U}})$,
				we have
				\begin{align*}
					\cD^v u_t&=\int_0^{t}  J_{r,t}Q v_r \dif r.
				\end{align*}
				Hence,   we also have
				\begin{eqnarray*}
					\cD_r^iu_t=J_{r,t}  Q \theta_i, ~\forall r\in [0,t], i=1,\cdots,\mathbb{U},
				\end{eqnarray*}
				where  the linearization  $J_{r,t}\xi$ is the solution of  \eqref{10-1},
				$Q$ is given by \eqref{pp28-1}, and  $\{\theta_i\}_{i=1}^\mathbb{U}$ is the standard basis of $\mR^{\mathbb{U}}.$
				Here and below, we adopt the notation $\cD_r^iF:=(\cD F)^i(r)$,
				that is
				$\cD_r^i$ denotes the $i$th component of $\cD F$ evaluated at time $r.$
				%
		\end{lemma}}

		For any $s\leq t$,   define  the linear operator
		$\cA_{s,t}v: L^2([s,t];\mR^{\mathbb{U}})\rightarrow H  $
		by
		\begin{eqnarray}
			\label{15-1}
			\cA_{s,t}v:=\int_{s}^{t}
			J_{r,t}Q v_r  \dif r, v\in L^2([s,t];\mR^{\mathbb{U}}).
		\end{eqnarray}
		For any $s < t$, let
		$\cA_{s,t}^*: H\rightarrow  L^2([s,t];\mR^{\mathbb{U}})$ be the adjoint of $\cA_{s,t}$  defined in the above. We observe that
		%
		%
		\begin{eqnarray*}
			(\cA_{s,t}^*\phi)(r)=Q^* J_{r,t}^*\phi=Q^*K_{r,t}\phi,
		\end{eqnarray*}
		where
		$Q^*: H  \rightarrow  \mathbb R^{\mathbb{U}} $  is the adjoint of $Q$ defined in (\ref{pp28-1}).
		The Malliavin matrix  $\cM_{s,t}:H\rightarrow H $ is defined by
		\begin{eqnarray}
			\label{200-3}
			\cM_{s,t}\phi &:=& \cA_{s,t}\cA_{s,t}^*\phi.
		\end{eqnarray}
		By direct calculations, we have
		\begin{eqnarray}
			\label{200-1}
			\langle \cM_{s,t}\phi,\phi\rangle
			=\sum_{j\in \cZ_0 }\int_s^t \langle K_{r,t}\phi, e_j\rangle^2 \dif r.
		\end{eqnarray}
		Recall that
		$
		\mmm=40 \Bbbk d(d+14\Bbbk)^2.
		$
		Now we list some estimates with regard to $J_{s,t}\xi,J_{s,t}^{(2)}(\phi,\psi)$ etc for
		$,\xi,\phi,\psi\in H$.
		\begin{lemma}
			\label{expo J}
			{    With  probability $1$, the following
				\begin{eqnarray}\label{87-10}
					\|J_{s,t}\xi\|^2+ \int_{s}^t\|J_{s,r}\xi\|_1^2\dif r
					\leq
					C \|\xi\|^2 \exp\Big\{C \int_s^t  ( \| u_r\|^{\mmm}_{L^\mmm}+1)  \dif r \Big\}
				\end{eqnarray}
				holds for any  $\xi\in H$ and $0\leq s\leq t$,
				where the constant  $C$    depends on $\nu,d,\Bbbk,\mathbb{U}$
				and   $
				(c_{\mathbbm{i},\mathbbm{j}})_{1\leq \mathbbm{i}\leq d,0\leq \mathbbm{j} \leq \Bbbk}.$
			}
		\end{lemma}
		
		\begin{proof}
			By direct calculations, we get
			\begin{eqnarray}
				\label{pp22-1}
				\partial_t \|J_{s,t}\xi\|^2=-2\nu \|J_{s,t}\xi\|_1^2
				-2\left\langle  \operatorname{div}A^\prime(u_t,J_{s,t}\xi),J_{s,t}\xi \right\rangle.
			\end{eqnarray}
			Set \begin{eqnarray*}
				p =
				\left\{
				\begin{split}
					&  4, \quad  \quad \quad  d=1,2,
					\\
					& \frac{2d-1}{d-2}, \quad d\geq 3.
				\end{split}
				\right.
			\end{eqnarray*}
			For the second term on  the right side of (\ref{pp22-1}), using H\"older's inequality, one arrives at that
			\baee
			&\big|-2 \left\langle  \operatorname{div}A^\prime(u_t,J_{s,t}\xi),J_{s,t}\xi \right\rangle\big|\\
			\leq &2 \sum_{i=1}^d \big| \left\langle \partial_{x_i}\left(  A_i^\prime(u_t)J_{s,t}\xi \right)  ,J_{s,t}\xi \right\rangle\big| \\
			=&2 \sum_{i=1}^d \big|\left\langle   A_i^\prime(u_t)J_{s,t}\xi ,\partial_{x_i}J_{s,t}\xi \right\rangle\big| \\
			\leq& 2 \sum_{i=1}^d \| A_i^{\prime}(u_t)\|_{L^\frac{2p}{p-2}} \| J_{s,t}\xi \|_{L^p}\| J_{s,t}\xi \|_1 \\
			:=& I.
			\eaee

			For the case  $d=1,2$, by  Gagliardo-Nirenberg's ineqaulity, we have
			\begin{align}\label{d geq 3 I}
				& \|J_{s,t}\xi \|_{L^4 } \leq
				C\| J_{s,t}\xi\|_1^{\lambda}\| J_{s,t}\xi\|^{1-\lambda},
			\end{align}
			where $\lambda=\frac{d}{4}.$
			Thus, for $d=1,2$, it holds that
			\begin{align}
				\nonumber   I \leq &
				C\sum_{i=1}^d \|  A_i^\prime(u)\|_{L^4}\|J_{s,t}\xi \|_{1 }^{1+\lambda}
				\|J_{s,t}\xi\|^{1-\lambda}
				\\  \nonumber  \leq& C \sum_{i=1}^d  \|  A_i^\prime(u_t)\|_{L^4}^{2/(1-\lambda)} \|J_{s,t}\xi\|^{2}+\frac12 \nu \|J_{s,t}\xi \|_{1 }^2
				\\
				\label{pp22-2}
				\leq & C( \| u_t\|^{\mmm}_{L^\mmm}+1) \| J_{s,t}\xi \|^{2}  +\frac12\nu\| J_{s,t}\xi \|_1^2.
			\end{align}
			For the case $d\geq 3$,  we have
			\begin{align}
				\nonumber 	I\leq&C  \sum_{i=1}^d \| A_i^{\prime}(u_t)\|_{L^\frac{2p}{p-2}} \| J_{s,t}\xi \|_{L^p}\|J_{s,t}\xi \|_1
				\\   \nonumber
				\leq& C ( \| u_t\|^{\Bbbk}_{L^\frac{2\Bbbk p}{p-2}} +1)\| J_{s,t}\xi \|_{L^p}\| J_{s,t}\xi \|_1 \\  \nonumber
				\leq& C  ( \| u_t\|^{\Bbbk}_{L^\frac{2\Bbbk p}{p-2}} +1) \| J_{s,t}\xi \|^{1-\lambda_{p,d}} \| J_{s,t}\xi \|_1^{1+ \lambda_{p,d}}\\ \label{pp22-3}
				\leq & C( \| u_t\|^{\mmm}_{L^\mmm}+1) \| J_{s,t}\xi \|^{2}  +\frac12\nu\| J_{s,t}\xi \|_1^2,
			\end{align}
			where we have used the Gagliardo-Nirenberg inequality
			\begin{equation*}
				\|J_{s,t}\xi\|_{L^p}\leq C \|J_{s,t}\xi \|^{1-\lambda_{p,d}} \| J_{s,t}\xi \|_{1}^ {\lambda_{p,d}}, ~
				\lambda_{p,d}=\frac{d(p-2)}{2p}=\frac{3d}{2(2d-1)}.
			\end{equation*}
			
			With the help of (\ref{pp22-2}) and (\ref{pp22-3}),  for $d\geq 1$, we obtain
			\begin{align*}
				\partial_t  \| J_{s,t}\xi\|^{2}  \leq& -\frac12\nu \|J_{s,t}\xi\|_1^{2}+ C  ( \| u_t\|^{\mmm}_{L^\mmm}+1) \|J_{s,t} \xi \|^{2}.
			\end{align*}
			Thus, using the Gr\"ownwall ineuqlity,  the proof of (\ref{87-10}) is complete.
		\end{proof}
		By similar arguments to those  in \cite[Proposition 5]{MR20}, we have the $L^1$ contraction of the Jacobian.
		\begin{lemma}[$L^1$ contraction of $J$]\label{L^1}
			With probability one,  we have
			\baee
			\left\| J_{s, t}\xi\right\|_{L^1}\leq \|\xi\|_{L^1}, ~\forall \xi\in H \text{ and }0\leq s\leq t.
			\eaee
		\end{lemma}
		\begin{proof}
			For $\eta>0$, we define a continuous approximation of the sign function as follow
			\begin{equation*}
				\operatorname{sign}_\eta(z):= \begin{cases}\frac{z}{\eta}, & z \in[-\eta, \eta], \\ 1, & z \geq \eta,\\ -1, & z \leq \eta.\end{cases}
			\end{equation*}
			{  	Then we have the {  following}   continuously differentiable approximation of the absolute value:}
			\begin{equation*}
				|f|_\eta:=\int_0^f \operatorname{sign}_\eta(z) \dif z, \quad \forall  f \in \mathbb{R} .
			\end{equation*}
			For $0 \leq s \leq t$, we have
			\bae\label{l11}
			& \int_{\mathbb{T}^d}|J_{s,t}\xi|_\eta \dif x-\int_{\mathbb{T}^d}|\xi|_\eta \dif x=\int_{\mathbb{T}^d} \int_s^t \frac{\dif}{\dif r}|J_{s,r}\xi|_\eta \dif r \dif x \\
			& =\int_{\mathbb{T}^d} \int_s^t \frac{\dif}{\dif r}J_{s,r}\xi \operatorname{sign}_\eta\big(J_{s,r}\xi\big) \dif r \dif x \\
			& =\int_s^t \int_{\mathbb{T}^d}\sum_{i=1}^d\left(A_i^\prime(u_r)J_{s,r}\xi-\nu \partial_{x_i}J_{s,r}\xi\right) \partial_{x_i}\operatorname{sign}_\eta(J_{s,r}\xi) \dif x \dif r\\
			& =\sum_{i=1}^d\int_s^t \int_{\mathbb{T}^d}\big(A_i^\prime(u_r)J_{s,r}\xi-\nu \partial_{x_i}J_{s,r}\xi\big) \partial_{x_i}J_{s,r}\xi \frac{1}{\eta} \mathbf{1}_{|J_{s,r}\xi| \leq \eta} \dif x \dif r \\
			& \leq \sum_{i=1}^d \int_s^t \int_{\mathbb{T}^d}A_i^\prime(u_r)J_{s,r}\xi \partial_{x_i}J_{s,r}\xi \frac{1}{\eta} \mathbf{1}_{|J_{s,r}\xi| \leq \eta} \dif x \dif r.
			\eae
			By Proposition \ref{wp}, $u_r$ belong to $C\left([s, t], L^\infty(\mathbb{T}^d)\right)$ almost surely. Then we have
			\baee
			&A_i^\prime(u_r)J_{s,r}\xi \partial_{x_i}J_{s,r}\xi \frac{1}{\eta} \mathbf{1}_{|J_{s,r}\xi| \leq \eta}
			\leq  C\Big(1+\sup _{r \in[s, t]}\|u_r\|_{L^{\infty}}^{\Bbbk-1}\Big) \left|\nabla J_{s,r}\xi\right|.
			\eaee
			Thus we conclude that
			\baee
			&\int_s^t \int_{\mathbb{T}^d} \left|\nabla J_{s,r}\xi\right|\dif x \dif r \leq C \int_s^t\left\| J_{s,r}\xi\right\|_1 \dif r \leq C (t-s)^\frac12\Big( \int_s^t\left\| J_{s,r}\xi\right\|_1^2 \dif r\Big)^\frac12 \\
			\leq&  C(t-s)^\frac12 \|\xi\| \exp\Big\{C \int_s^t( \| u\|^{\mmm}_{L^\mmm}+1) \dif r \Big\} <\infty,\quad \mathbb P\text{-a.s.},
			\eaee
			where we used Lemma \ref{expo J} in the last line. By the dominated convergence theorem, one has
			$$
			\begin{aligned}
				& \lim _{\eta \rightarrow 0} \int_s^t \int_{\mathbb{T}^d}A_i^\prime(u_r)J_{s,r}\xi \partial_{x_i}J_{s,r}\xi \frac{1}{\eta} \mathbf{1}_{|J_{s,r}\xi| \leq \eta}\dif x \dif r \\
				& \quad=\int_s^t \int_{\mathbb{T}^d} \lim _{\eta \rightarrow 0}A_i^\prime(u_r)J_{s,r}\xi \partial_{x_i}J_{s,r}\xi \frac{1}{\eta} \mathbf{1}_{|J_{s,r}\xi| \leq \eta} \dif x \dif r=0 .
			\end{aligned}
			$$

			Noticing that $|\cdot|_\eta$ increases to $|\cdot|$ as $\eta$ decreases, by the monotone convergence theorem and   \eqref{l11}, one has
			\baee
			&\|J_{s,t}\xi\|_{L^1}=\lim _{\eta \rightarrow 0} \int_{\mathbb{T}^d}|J_{s,t}\xi|_\eta \dif x \leq \lim _{\eta \rightarrow 0} \int_{\mathbb{T}^d}|\xi|_\eta \dif x=\|\xi\|_{L^1}.
			\eaee
			We complete the proof.
		\end{proof}
		
		\begin{lemma}
			\label{15-4}
			{
				With probability one, it holds that
				\begin{eqnarray}\label{integrable J L^2}
					\begin{split}
						&\|J_{s,t}\xi\|^2+ \int_{s}^t\|J_{s,r}\xi\|_1^2\dif r
						\\ & \leq C \|\xi\|^2+  C\|\xi\|^2 \int_{s}^t( \| u_r\|^{\mmm}_{L^\mmm}+1) \dif r,
						~\forall \xi\in H \text{ and }0\leq s\leq t,
					\end{split}
				\end{eqnarray}
				where the constant  $C$    depends on $\nu,d,\Bbbk, (b_{i})_{i\in \cZ_0},\mathbb{U}$
				and   $
				(c_{\mathbbm{i},\mathbbm{j}})_{1\leq \mathbbm{i}\leq d,0\leq \mathbbm{j} \leq \Bbbk}.$
			}
		\end{lemma}
		
		\begin{proof}
			Set $p=\frac{2d+1}{d}\in (2,3]$ and $\lambda=\lambda_{p,d}=\frac{2d(p-1)}{p(2+d)}=\frac{2d(d+1)}{(2d+1)(d+2)}\in (1/3,1-\frac{1}{4d})$.
			Observe that
			\begin{eqnarray}
				\nonumber  & &
				\big| \left\langle  \operatorname{div}A^\prime(u,J_{s,t}\xi),J_{s,t}\xi \right\rangle\big| \\
				\nonumber && =\big| \sum_{i=1}^d \left\langle \partial_{x_i}\left(  A_i^\prime(u)J_{s,t}\xi \right)  ,J_{s,t}\xi \right\rangle\big| =\big| \sum_{i=1}^d \left\langle   A_i^\prime(u)J_{s,t}\xi ,\partial_{x_i}J_{s,t}\xi \right\rangle\big|
				\\
				\nonumber && \leq \sum_{i=1}^d \| A_i^{\prime}(u)\|_{L^\frac{2p}{p-2}} \| J_{s,t}\xi \|_{L^p}\| \nabla J_{s,t}\xi \|
				\leq C \sum_{i=1}^d (1+\| u\|^{\Bbbk-1}_{L^\frac{2(\Bbbk-1)p}{p-2}}) \| J_{s,t}\xi \|_{L^p}\| \nabla J_{s,t}\xi \|\\
				\nonumber  && \leq  C \sum_{i=1}^d (1+\| u\|^{\Bbbk-1}_{L^\frac{2(\Bbbk-1)p}{p-2}}) \| J_{s,t}\xi \|_{L^1}^{1-\lambda} \| \nabla J_{s,t}\xi \|^{1+ \lambda}\\
				\nonumber  && \leq    C \sum_{i=1}^d (1+\| u\|^{2(\Bbbk-1)/(1-\lambda)}_{L^\frac{2(\Bbbk-1)p}{p-2}}) \| \xi \|^{2}  +\frac12\nu\| \nabla J_{s,t}\xi \| ^2
				\\ \label{p1107-1}  && \leq    C  (1+\| u\|^{\mmm}_{L^\mmm}) \| \xi \|^{2}  +\frac12\nu\| \nabla J_{s,t}\xi \| ^2.
			\end{eqnarray}
			In the above,  for  the second inequality,   we have used the Gagliardo-Nirenberg's inequality
			$
			\|J_{s,t}\xi\|_{L^p}\leq C \|J_{s,t}\xi\|_{L^1}^{1-\lambda} \| J_{s,t}\xi \|_{1}^ {\lambda}\leq C \|J_{s,t}\xi \|_{L^1}^{1-\lambda} \|\nabla J_{s,t}\xi\|^ {\lambda};
			$
			for  the  third inequality, we have used Lemma \ref{L^1} and Young's inequality;
			and  in the last inequality, we have used the fact
			\begin{equation*}
				\mmm \geq \max\{\frac{2(\Bbbk-1)p}{p-2},\frac{2(\Bbbk-1)}{(1-\lambda)}\}.
			\end{equation*}
			Then using the chain rule for   $\frac{\dif  \| J_{s,t}\xi\|^2   }{\dif t }$,
			also  with the help of (\ref{p1107-1}),  we get  the desired result (\ref{integrable J L^2}).
		\end{proof}
		
		\begin{lemma}\label{L^p J}
			With probability one,   the following  pathwise estimate
			\bae
			&\|J_{s,t}\xi\|_{L^8}
			\leq   C \|\xi\|_{L^8}  \exp\Big\{C  \int_s^t ( \| u_r \|^{\mmm}_{L^{\mmm}}+1 )  \dif r \Big\}
			\eae
			holds for any  $0\leq s\leq t   \text{ and  } \xi \in H \text{ with  }\int_{\mT^d}\xi(x)\dif x=0$,
			where $C$  is a constant   depending on  $\nu,d,\Bbbk,(b_j)_{j\in \cZ_0},\mathbb{U}$
			and $
			(c_{\mathbbm{i},\mathbbm{j}})_{1\leq \mathbbm{i}\leq d,0\leq \mathbbm{j} \leq \Bbbk}$.
			
		\end{lemma}
		\begin{proof}
			Set $q=\frac{2d+1}{d}\in (2,3]$ and $\lambda=\lambda_{q,d}=\frac{d(q-2)}{2q}=\frac{d}{4d+2}\in [\frac{1}{6},\frac{1}{4})$.   Observe that
			\baee
			& \Big| \big \langle  \operatorname{div}A^\prime(u,J_{s,t}\xi),(J_{s,t}\xi)^{7}  \big\rangle\Big| \\
			=&\Big| \sum_{i=1}^d \big \langle \partial_{x_i}\left(  A_i^\prime(u)J_{s,t}\xi \right)  , (J_{s,t}\xi)^7  \big \rangle\Big|
			\\
			=&7 \Big| \sum_{i=1}^d \big\langle   A_i^\prime(u), (J_{s,t}\xi)^6  \partial_{x_i}J_{s,t}\xi \big\rangle\Big|
			\\
			\leq&C \sum_{i=1}^d  \| A_i^\prime(u) \|_{L^{\frac{2q}{q-2}}} \big\|(J_{s,t}\xi)^4 \big\|_{L^q} \big\|(J_{s,t}\xi)^{3} \partial_{x_i}J_{s,t}\xi\big\|
			\\
			\leq&C \sum_{i=1}^d (1+ \| u \|^{\Bbbk-1}_{L^{\frac{2q(\Bbbk-1)}{q-2}}}) \|f\|_{L^q} \|\nabla f\|,
			\\
			\eaee
			where $f(x):=(J_{s,t}\xi)^{4}(x)$.
			By  the Gagliardo-Nirenberg's inequality
			\begin{equation*}
				\|f\|_{L^q}\leq C \|f\|^{1-\lambda} \| f\|_{1}^ {\lambda} \leq C\|f\|^{1-\lambda} \|\nabla f\|^ {\lambda}
			\end{equation*}
			and   in view of   $m\geq 6d \Bbbk$, we get
			\begin{eqnarray}
				\nonumber & &\Big| \big\langle  \operatorname{div}A^\prime(u,J_{s,t}\xi),(J_{s,t}\xi)^{7} \big\rangle\Big| \\
				\nonumber  & & \leq C\sum_{i=1}^d  (1+\| u \|^{\Bbbk-1}_{L^{\frac{2q(\Bbbk-1)}{q-2}}}) \|f\|^{1-\lambda} \|\nabla f\|^{1+\lambda} \\
				\nonumber & & \leq C\big(1+\sum_{i=1}^d  \| u \|^{\Bbbk-1}_{L^{\frac{2q(\Bbbk-1)}{q-2}}}\big)^{\frac{2}{1-\lambda}}\|f\|^{2}+\frac12\nu \|\nabla f\|^{2} \\
				\label{p1107-2}
				& &  \leq C ( \| u \|^{\mmm}_{L^{\mmm}}+1 )\|(J_{s,t}\xi)^4 \|^{2}+\frac{\nu }{32}  \big\|(J_{s,t}\xi)^4 \big\|_1^{2}.
			\end{eqnarray}

			With the help of    \cite[Proposition A.1]{CGV14}, it holds that
			\begin{eqnarray}
				\label{p28-1}
				\begin{split}
					& \int_{\mT^d }w(x)^{p-1}\big(-\Delta w(x)\big)\dif x
					\geq
					C_{p}^{-1}\|w\|_{L^p}^p+\frac{1}{p}\|(-\Delta)^{1/2}w^{p/2}\|^2,
				\end{split}
			\end{eqnarray}
			where $w(x)=J_{s,t}\xi(x),p=8$ and $C_p\in (1,\infty)$ is a constant depending on $p,d.$
			Finally, using It\^o's formula for $\|J_{s,t}\xi\|_{L^8}^8$,  also with the help of (\ref{p28-1}) and (\ref{p1107-2}),  we conclude that
			\baee
			&\|J_{s,t}\xi\|^8_{L^8}+\int_s^t \big\|(J_{s,r}\xi)^{4}\big\|_1^{2} \dif r \leq C \int_s^t ( \| u_r \|^{\mmm}_{L^{\mmm}}+1 )\|(J_{s,r}\xi)^4 \|^{2} \dif r, ~\forall 0\leq s\leq  t.
			\eaee
			The above inequality implies the desired result. The proof is completed.

		\end{proof}
		
		Recall that $\mmm=40 \Bbbk d(d+14\Bbbk)^2.$
		\begin{lemma}\label{2.18 heissan}
			Almost surely, we have
			\begin{eqnarray}\label{87-12}
				\begin{split}
					&
					\|J_{s,t}^{(2)}(\phi,\psi)\|^2
					\\ & \leq C
					\|\phi\|_{L^8}^2 \|\psi\|^2
					\exp\Big\{C\int_s^t (\|u_r\|_{L^\mmm}^\mmm+1)\dif r \Big\},
					~\forall \phi,\psi\in H \text{ and } 0\leq s\leq t,
				\end{split}
			\end{eqnarray}
			where $C$ is a constant depending on  $\nu,d,\Bbbk, (b_{i})_{i\in \cZ_0},\mathbb{U}$  and $
			(c_{\mathbbm{i},\mathbbm{j}})_{1\leq \mathbbm{i}\leq d,0\leq \mathbbm{j} \leq \Bbbk}$.
		\end{lemma}
		\begin{proof}
		Multiplying $J^{(2)}_{s,t}(\phi,\psi)$ on both sides of \eqref{0927-1} and integrating on $\mathbb T^d$, one has
			\bae\label{hessian J12}
			&\frac12 \frac{\dif}{\dif t}\big\|J^{(2)}_{s,t}(\phi,\psi)\big\|^2+\nu\big\|\nabla J^{(2)}_{s,t}(\phi,\psi)\big\|^2\\
			\leq&\Big| \Big\langle  \operatorname{div}A^\prime\big(u,J^{(2)}_{s,t}(\phi,\psi)\big),J^{(2)}_{s,t}(\phi,\psi) \Big\rangle+  \Big\langle\operatorname{div}A^{\prime\prime}\big(u,J_{s,t}\phi   J_{s,t}\psi\big),  J^{(2)}_{s,t}(\phi,\psi) \Big\rangle\Big| \\
			=&\Big| \sum_{i=1}^d \left\langle \partial_{x_i}\big[  A_i^\prime(u)J^{(2)}_{s,t}(\phi,\psi)+A_i^{\prime\prime}(u)J_{s,t}\phi   J_{s,t}\psi \big]  ,J^{(2)}_{s,t}(\phi,\psi) \right\rangle\Big| \\
			\leq &\Big| \sum_{i=1}^d \left\langle A_i^\prime(u)J^{(2)}_{s,t}(\phi,\psi),\partial_{x_i} J^{(2)}_{s,t}(\phi,\psi) \right\rangle
			\Big|+\Big|\sum_{i=1}^d \left\langle A_i^{\prime\prime}(u)J_{s,t}\phi J_{s,t}\psi   ,\partial_{x_i} J^{(2)}_{s,t}(\phi,\psi) \right\rangle\Big|
			\\
			:=& J_1+J_2.
			\eae
			Set  $p=\frac{2(d+2)}{d},q=d+2$ and $\lambda=\frac{d(p-2)}{2p}=\frac{4d}{4d+8}$.
			For the term $J_1$, we have
			\bae
			J_1=& \Big| \sum_{i=1}^d \int_{\mT^d} A_i^{\prime}(u) J^{(2)}_{s,t}(\phi,\psi)\partial_{x_i} J^{(2)}_{s,t}(\phi,\psi) \dif x
			\Big|
			\\
			\leq&C  (1+\left\|u\right\|_{L^{(\Bbbk-1)q}}^{\Bbbk-1}) \big\|J^{(2)}_{s,t}(\phi,\psi)\big\|_{L^{p}}\big\|\nabla J^{(2)}_{s,t}(\phi,\psi)\big\|\\
			\leq&C(1+  \left\|u\right\|_{L^{(\Bbbk-1)q}}^{\Bbbk-1}) \big\|J^{(2)}_{s,t}(\phi,\psi)\big\|^{1-\lambda}\big\|\nabla J^{(2)}_{s,t}(\phi,\psi)\big\|^{1+\lambda}\\
			\eae
			where the last inequality  follows  from  Gagliardo-Nirenberg's  inequality. Then, by  Young's inequality, we have
			\bae\label{hessian J1}
			J_1 \leq&C\big( 1+\left\|u\right\|_{L^{(\Bbbk-1)q}}^{\frac{2(\Bbbk-1)}{1-\lambda}}\big) \big\|J^{(2)}_{s,t}(\phi,\psi)\big\|^{2}+\frac12\nu \big\|\nabla J^{(2)}_{s,t}(\phi,\psi)\big\|^{2}\\
			\leq&C \big(1+\left\|u\right\|_{L^{(\Bbbk-1)q}}^{2(\Bbbk-1)/(1-\lambda)} \big) \big\|J^{(2)}_{s,t}(\phi,\psi)\big\|^{2}+\frac12\nu \big\|\nabla J^{(2)}_{s,t}(\phi,\psi)\big\|^{2}.
			\eae
			
			Now consider the term $J_2$. By H\"older's inequality, one arrives at that
			\bae
			&J_2 \\
			\leq& \sum_{i=1}^d \left\|  A_i^{\prime\prime}(u) \right\|_{L^{8}}\left\| J_{s,t}\phi \right\|_{L^8} \left\|J_{s,t}\psi  \right\|_{L^4}  \big\| \nabla J^{(2)}_{s,t}(\phi,\psi)\big\|\\
			\leq& C (1+\left\| u \right\|_{L^{8(\Bbbk-2)}}^{\Bbbk-2})\left\| J_{s,t}\phi \right\|^2_{L^8} \left\|J_{s,t}\psi\right\|^2_{L^4}+\frac12 \nu  \big\| \nabla J^{(2)}_{s,t}(\phi,\psi)\big\|^2.\\
			\eae
			
			Combining the above estimates of $J_1,J_2$ with (\ref{hessian J12}), we conclude that
			\begin{align*}
				\frac{\dif  \| J^{(2)}_{s,t}(\phi,\psi)\|^2}{\dif  t} \leq &  C \big(1+ \left\|u_t \right\|_{L^{(\Bbbk-1)q}}^{2(\Bbbk-1)/(1-\lambda)} \big) \big\|J^{(2)}_{s,t}(\phi,\psi)\big\|^{2}
				\\ &+C (1+\left\| u \right\|_{L^{8(\Bbbk-2)}}^{\Bbbk-2}) \left\| J_{s,t}\phi \right\|^2_{L^8} \left\|J_{s,t}\psi  \right\|^2_{L^4}.
			\end{align*}
			Furthermore, in view of $\mmm=40 \Bbbk d(d+14\Bbbk)^2$  and Lemmas   \ref{15-4},  \ref{L^p J}, the above inequality implies the desired result. The proof is complete.

		\end{proof}

		Recall that $P_N$ is the orthogonal projection from $H$ into $H_N=\text{span}\{e_j; j\in Z_{\ast}^d, |j|\leq N\}$ and $Q_N=I-P_N$.
		For any $N\in \mN,t\geq 0$ and $\xi\in H$,  denote $\xi_t^h:=Q_NJ_{0,t}\xi,
		\xi_t^\iota :=P_NJ_{0,t}\xi$ and $\xi_t:=J_{0,t}\xi$.
		\begin{lemma}
			\label{48-1}
			With probability one,  for  any $\xi\in H$,  $t\in [0,1]$ and $N\geq 1$, one has
			\begin{eqnarray}
				\label{10-2}
				\begin{split}
					& \|\xi_t^h\|^2\leq  e^{-\nu N^2 t}
					\|Q_N\xi\|^2
					+\frac{C \|\xi \|^2 \exp\Big\{C \int_0^t \|u_r\|_{L^\mmm}^\mmm \dif r\Big\} }{N},
				\end{split}
			\end{eqnarray}
			where $C$  is a constant     depending on  $\nu,d,\Bbbk,(b_j)_{j\in \cZ_0},\mathbb{U}$
			and $(c_{\mathbbm{i},\mathbbm{j}})_{1\leq \mathbbm{i}\leq d,0\leq \mathbbm{j} \leq \Bbbk}$.
		\end{lemma}
		
		\begin{proof}
			By direct calculations, it holds that
			\begin{eqnarray}
				\label{p111-1}
				\frac{\dif  \|\xi_t^h\|^2}{\dif t}\leq  -2\nu \|\xi_t^h\|_1^2
				{  + 2\sum_{i=1}^d \Big|  \big\langle \partial_{x_i} \big(A_i'(u_t)\xi_t\big), \xi_t^h\big \rangle\Big|.}
			\end{eqnarray}
			Set $p=\frac{8d}{4d-1}\in (2,\infty)$.By Gagliardo-Nirenberg's  inequality, it holds that
			\begin{eqnarray*}
				\|w\|_{L^p}\leq C\|w\|_1^{1/8} \|w\|^{7/8}.
			\end{eqnarray*}
			Therefore, by H\"older's inequality,  we have
			\begin{eqnarray}
				\label{p05-1}
				\begin{split}
					& \big|\big\langle \partial_{x_i} \big(A_i'(u_t)\xi_t\big), \xi_t^h\big \rangle\big|
					= \big|\big\langle  A_i'(u_t)\xi_t, \partial_{x_i}(\xi_t^h) \big \rangle\big|
					\\ &\leq \| A_i'(u_t)\|_{L^{2p/(p-2)}}\| \xi_t\|_{L^p}\| \xi_t^h\|_1
					\\   &  \leq C\| A_i'(u_t)\|_{L^{2p/(p-2)}} \|\xi_t \|_1^{1/8} \|\xi_t\|^{7/8} \| \xi_t^h\|_1
					\\ &\leq   \frac{\nu}{4}\| \xi_t^h\|_1^2+
					C \| A_i'(u_t)\|_{L^{2p/(p-2)}}^2  \|\xi_t \|_1^{1/4}  \|\xi_t\|^{7/4}.
				\end{split}
			\end{eqnarray}
			Combining the above estimate with  (\ref{p111-1}),
			also with the help of Lemma \ref{15-4}, for any $t\in [0,1]$, one arrives at that
			
			\
			
			\begin{eqnarray*}
				&&  \|\xi_t^h\|^2
				\\ && \leq  e^{-\nu N^2 t}
				\|Q_N\xi\|^2
				+\sum_{i=1}^d \int_0^t \exp\{-\nu N^2(t-s)\}\| A_i'(u_s)\|_{L^{2p/(p-2)}}^2  \|\xi_s \|_1^{1/4}  \|\xi_s\|^{7/4} \dif s
				\\ &&\leq   e^{-\nu N^2 t}
				\|Q_N\xi\|^2+
				C \sum_{i=1}^d \Big(\int_0^t \exp\{-2 \nu N^2(t-s)\}
				\dif s \Big)^{1/2}
				\\ &&\quad\quad \times
				\Big(\int_0^t \| A_i'(u_s)\|_{L^{2p/(p-2)}}^8    \dif s \Big)^{1/4}
				\Big(\int_0^t \|\xi_s \|_1^2     \dif s \Big)^{1/8} t^{1/8}
				\sup_{s\in [0,t]}\|\xi_s\|^{7/4}
				\\ && \leq e^{-\nu N^2 t}
				\|Q_N\xi\|^2
				+\frac{C \|\xi \|^2  \exp\Big\{C \int_0^t \|u_r\|_{L^\mmm}^\mmm \dif r\Big\} }{N}.
			\end{eqnarray*}
			The proof is complete.

		\end{proof}

		\begin{lemma}
			\label{1340-2}
			With probability $1$,
			for any $t\in [0,1]$ and   $\xi\in H$ with $\xi_0^\iota=0$,  we have
			\begin{eqnarray}
				\label{16-1}
				\begin{split}
					& \|\xi_t^\iota\|^2\leq  C\exp\Big\{C\sum_{i=1}^d  \int_0^t \| A_i'(u_s)\|_{L^{2p/(p-2)}}^{16/7} \dif s  \Big\}
					\\ & \quad \quad \quad\quad  \times \Big( \sum_{i=1}^d \int_0^t
					\| A_i'(u_s)\|_{L^{2p/(p-2)}}^2  \|\xi_s^h  \|_1^{1/4} \|\xi_s^h \|^{7 /4} \dif s\Big),
				\end{split}
			\end{eqnarray}
			where $p=\frac{8d}{4d-1}$,  $C$  is  a   constant   depending on
			$\nu,d , \{b_j\}_{j\in \cZ_0},\mathbb{U}$
			and $
			(c_{\mathbbm{i},\mathbbm{j}})_{1\leq \mathbbm{i}\leq d,0\leq \mathbbm{j} \leq \Bbbk}$. Combining the above  inequality with Lemma  \textup{\ref{48-1}}, for any $\xi\in H$ and $N\geq 1$,
			we have
			\begin{eqnarray}
				\label{16-2}
				\begin{split}
					& \|J_{0,1}Q_N \xi\|^2
					\leq
					\frac{C\|\xi\|^2}{N^{1/2}}
					\exp\Big\{C \int_0^1
					\|u_s\|_{L^\mmm}^\mmm d  s  \Big\}.
				\end{split}
			\end{eqnarray}
		\end{lemma}
		\begin{proof}
			First, we give a proof of (\ref{16-1}). By direct calculations, it holds that
			\begin{eqnarray}
				\label{p111-1}
				\frac{\dif  \|\xi_t^l\|^2}{\dif t}\leq  -2\nu \|\xi_t^l\|_1^2
				{  + 2\sum_{i=1}^d \Big|  \big\langle \partial_{x_i} \big(A_i'(u_t)\xi_t\big), \xi_t^l\big \rangle\Big|.}
			\end{eqnarray}
			With similar arguments as that  in (\ref{p05-1}), by Young's inequality,  we conclude that
			\begin{eqnarray*}
				&& \big|\big\langle \partial_{x_i} \big(A_i'(u_t)\xi_t^\iota \big), \xi^\iota_t \big \rangle\big|\leq
				C\| A_i'(u_t)\|_{L^{2p/(p-2)}} \| \xi^\iota_t  \|_{L^p} \| \partial_{x_i}  \xi^\iota_t   \|
				\\ &&\leq
				C\| A_i'(u_t)\|_{L^{2p/(p-2)}}  \|\xi_t^\iota \|^{7/8} \| \xi_t^\iota \|_1^{9/8}
				\\ &&\leq  \frac{\nu}{4}\| \xi_t^\iota \|_1^{2}
				+C\| A_i'(u_t)\|_{L^{2p/(p-2)}}^{16/7}   \|\xi_t^\iota \|^2
			\end{eqnarray*}
			and
			\begin{eqnarray*}
				&& \big|\big\langle \partial_{x_i} \big(A_i'(u_t)\xi_t^h \big), \xi_t^\iota \big \rangle\big|
				= \big|\big\langle A_i'(u_t)\xi_t^h, \partial_{x_i} \big(\xi_t^\iota\big) \big \rangle\big|
				\\ &&
				\leq \| A_i'(u_t)\|_{L^{2p/(p-2)}}\| \xi_t^h \|_{L^p}  \| \xi_t^\iota \|_1
				\\  && \leq  C\| A_i'(u_t)\|_{L^{2p/(p-2)}} \|\xi_t^h  \|_1^{1/8} \|\xi_t^h \|^{7/8} \|\xi_t^\iota \|_1
				\\ &&\leq \frac{\nu}{4}\| \xi_t^\iota \|_1^{2}
				+ C\| A_i'(u_t)\|_{L^{2p/(p-2)}}^2  \|\xi_t^h  \|_1^{1/4} \|\xi_t^h \|^{7/4}.
			\end{eqnarray*}
			Applying the chain rule  to
			$\|\xi_t^\iota\|^2$,
			by the above two  estimates, we arrive at
			\begin{eqnarray*}
				\|\xi_t^\iota\|^2\leq C\exp\Big\{C\sum_{i=1}^d  \int_0^t \| A_i'(u_s)\|_{L^{2p/(p-2)}}^{16/7} \dif s  \Big\}\Big(\sum_{i=1}^d\int_0^t
				\| A_i'(u_s)\|_{L^{2p/(p-2)}}^2  \|\xi_s^h  \|_1^{1/4} \|\xi_s^h \|^{7 /4} \dif s\Big).
			\end{eqnarray*}
			
			Now, we give a proof of (\ref{16-2}).
			Let $\tilde \xi=Q_N \xi.$
			By  {  Lemma \ref{15-4},  Lemma  \ref{48-1} }and (\ref{16-1}), for any $t\in [0,1]$, we get
			\begin{eqnarray*}
				&& \|P_N J_{0,t} \tilde  \xi \|^2
				\\ && \leq C\exp\Big\{C \sum_{i=1}^d \int_0^t \| A_i'(u_s)\|_{L^{2p/(p-2)}}^{16/7}\dif s  \Big\}\\
				&&\quad\quad\quad\quad \times \int_0^t  \sum_{i=1}^d
				\| A_i'(u_s)\|_{L^{2p/(p-2)}}^2  \|Q_N J_{0,s} \tilde \xi  \|_1^{1/4} \| Q_N J_{0,s} \tilde \xi  \|^{7/4} \dif s
				\\ && \leq C\exp\Big\{C  \int_0^t \| u_s\|_{L^{\mmm}}^\mmm \dif  s  \Big\}
				\Big(\sum_{i=1}^d \int_0^t
				\| A_i'(u_s)\|_{L^{2p/(p-2)}}^{32/7} \dif s  \Big)^{7/16}  \Big(\int_0^t \|Q_N J_{0,s} \tilde \xi  \|_1^2 \dif s \Big)^{1/8}
				\\ && \quad\quad\quad\quad \times
				\Big(\int_0^t \| Q_N J_{0,s} \tilde \xi  \|^4 \dif s\Big)^{7/16}
				\\ &&\leq C\| \tilde \xi \|^2 \exp\Big\{C \int_0^t \| u_s\|_{L^{\mmm}}^\mmm \dif  s  \Big\}
				\\ && \quad\quad\quad\quad  \times
				\Bigg(\int_0^t {  \exp\{-2\nu N^2 s\}}
				\dif s
				+  \frac{C    \exp\Big\{C \int_0^t \|u_r\|_{L^\mmm}^\mmm \dif r\Big\} }{N^2 }
				\Bigg)^{7/16}.
			\end{eqnarray*}				Setting   $t=1$ in the above, we get
			\begin{eqnarray*}
				\|P_N J_{0,1}Q_N\xi \|^2
				\leq  \frac{C\|\xi\|^2 \exp\Big\{C \int_0^1\|u_r\|_{L^\mmm}^\mmm  \dif r  \Big\}
				}{N^{1/2}}.
			\end{eqnarray*}
			Combing the above inequality with Lemma \ref{48-1}, we get the desired result (\ref{16-2}).
			
		\end{proof}

		Using the similar arguments as that in   \cite[Section~4.8]{HM-2006}  or   \cite[Lemma~A.6]{FGRT15}, we have the following lemma.
		\begin{lemma}
			\label{L:2.2}
			There is a  constant $C=C(\nu,d,\Bbbk,\{b_j\}_{j\in \cZ_0},\mathbb{U})>0$ such that  for any
			$0\le s<t$ and  $\beta>0$, we have
			\begin{gather}
				\label{2.7}  \|\cA_{s,t}\|^2_{\cL(L^2([s,t];\R^{\mathbb{U}}),
					{H})} \le  C \int_{s}^{t}\|J_{r,t}\|_{\cL(H,{H})}^2\dif r,
				\\   \|\cA_{s,t}^*(\cM_{s,t}+\beta\I)^{-1/2}\|_{\cL({H},
					L^2([s,t];\R^{\mathbb{U}}))} \le  1,\label{2.8}
				\\
				\|(\cM_{s,t}+\beta\I)^{-1/2}\cA_{s,t}\|_{\cL(L^2([s,t];\R^{\mathbb{U}}),
					{H})} \le  1,\label{2.9}
				\\  \|(\cM_{s,t}+\beta\I)^{-1/2}\|_{\cL({H},{H})}
				\le  \beta^{-1/2}, \label{2.10}
				\\  \|(\cM_{s,t}+\beta\I)^{-1}\|_{\cL({H},{H})}
				\le  \beta^{-1}.\label{2.11}
			\end{gather}
		\end{lemma}

		\section{The invertibility of the Malliavin matrix $\cM_{0,t}.$}\label{sec invert}
		\label{S:3}

		For $T>0$,  recall that  $\{u_r\}_{ r\in [0,T]} $ is  the solution of   equation (\ref{1-1}) with initial value $u_0\in \tilde H^\nn$ and   $
		\langle \cM_{0,T}\phi, \phi  \rangle
		=\sum_{i\in \cZ_0}\int_0^T b_i^2 \langle J_{r,T}e_i,\phi\rangle^2\dif r, \forall\phi\in \tilde H
		$.  The aim of this section is to prove the following two propositions.

		\begin{proposition}
			\label{1-66}
			Under the Condition \textup{\ref{16-5}}, for any $T>0, \alpha\in (0,1],N\in \mN$ and  $u_0\in \tilde H^{\nn+5}, $   one has
			\begin{eqnarray}
				\label{1-2}
				\mP\Big(\inf_{\phi\in \cS_{\alpha,N}}\langle
				\cM_{0,T}\phi,\phi\rangle=0 \Big)=0,
			\end{eqnarray}
			where $\cS_{\alpha,N}:=\{\phi\in \tilde H:\|P_N\phi\|\geq \alpha, \|\phi\|=1\}.$
		\end{proposition}

		However, this proposition is insufficient for our proof of Proposition \ref{3-11} and a stronger version is required. Before presenting this stronger version, we introduce some necessary notation.
		{   For $ \alpha\in (0,1], u_0\in \tilde  H^{\nn+5},N\in \mN,\mathfrak{R}>0$ and $\eps>0$, let\footnote{Note that $\cM_{0,t}$ is the Malliavin matrix of $u_t$ which is the solution of equation (\ref{1-1})  at time $t$ with  initial value is $u_0.$
				Therefore, $\cM_{0,1}$ also depends on $u_0$. }
			\begin{eqnarray}
				\label{1-5}
				X^{u_0,\alpha,N}=\inf_{ \phi \in \cS_{\alpha,N}}
				\langle \cM_{0,1}\phi,\phi\rangle.
			\end{eqnarray}
			and denote
			\begin{eqnarray}
				&&
				\label{p25-2}
				r(\eps,\alpha, \mathfrak{R},N):=\sup_{{u_0\in H^{\nn+5}:
						\|u_0\|_{\nn+5}  <  \mathfrak{R}}}\mP(  X^{u_0,\alpha,N}<\eps).
			\end{eqnarray}
		} 	
		
		{
			\begin{proposition}
				\label{3-8}
				For any $\eps>0, \alpha \in (0,1], \mathfrak{R}>0$ and $N\in \mN$,
				regarding    $ r(\eps,\alpha,\mathfrak{R},N)$  as  a function   of $\eps$,
				we have
				\begin{eqnarray}
					\label{p25-4}
					\lim_{\eps\rightarrow 0}r(\eps,\alpha,\mathfrak{R},N)=0.
				\end{eqnarray}
			\end{proposition}
		}
		
		This section is organized as follows.
		In subsection \ref{s3-1}, we give a proof of Proposition \ref{1-66}
		and in   subsection \ref{s3-2},  we give a proof of Proposition \ref{3-8}.
		
		\subsection{Proof of Proposition \ref{1-66}}
		\label{s3-1}
		
		Under the Condition \ref{16-5},  in this subsection,  we will prove the following stronger result than
		Proposition \ref{1-66}  for later use:
		\begin{eqnarray}
			\label{86-1}
			\mP\Big(\omega: \inf_{\phi\in \cS_{\alpha,N}}\sum_{i\in \cZ_0}b_i^2\int_{T/2}^T \langle K_{r,T}\phi, e_i\rangle^2 \dif r =0 \Big)=0.
		\end{eqnarray}
		
		First, we list a proposition and two lemmas. Then, we  conclude a proof of  (\ref{86-1}).
		The following proposition is taken from \cite[Theorem 7.1]{HM-2011}.
		\begin{proposition}
			\label{p10-1}
			Let $\{W_k(t)\}_{k=1}^{\mathbb{U}}$
			be a family of i.i.d. standard Wiener processes on interval $[0,T]$ and, for
			every multi-index $\alpha= (\alpha_1,\cdots,\alpha_\mathbb{U})$, define $W_\alpha = W_{1}^{\alpha_1}\cdots W_{\mathbb{U}}^{\alpha_{\mathbb{U}}}$  with the convention
			that $W_\alpha = 1$ if $\alpha = \emptyset$. Let furthermore $A_\alpha$ be a family of (not necessarily adapted)
			stochastic processes with the property that there exists $n \geq 0$ such that $A_\alpha  = 0$
			whenever ${  |\alpha|:=\sum_{i=1}^\mathbb U \alpha_i>n}$ and set $Z_A(t) =
			\sum_\alpha A_\alpha(t)W_\alpha(t),t\in [0,T]$.
			Then, there exists a family of events $Oscm_W^{n,\eps},\eps\in (0,1)$ depending only on
			$n,\eps,T$  and  $W=\{W_k(t),t\in [0,T]\}_{k=1}^{\mathbb{U}}$  such that the followings hold:
			
			(1) On the event $(Oscm_W^{n,\eps})^c$, it has
			\baee
			\|Z_{A}\|_{L^\infty}\leq \eps
			\Rightarrow
			\Bigg\{
			\begin{split}
				\sup_{\alpha}\| A_\alpha \|_{L^\infty}\leq \eps^{3^{-n}},
				& \\
				\text{ or }  \sup_{\alpha}\| A_\alpha \|_{Lip}\geq \eps^{-3^{-(n+1)}},&
			\end{split}
			\eaee
			where $\|A_\alpha\|_{L^\infty}=\sup_{t\in [0,T]}|A_\alpha(t)|$ and   $\| A_\alpha \|_{Lip}:=\sup_{s\neq t, s,t\in [0,T]}\frac{|A_\alpha(t)-A_\alpha(s)|}{|t-s|}.$

			(2) {  $\mP(Oscm_W^{n,\eps} )\leq C_{p,n,T}\eps^p,\forall \eps \in (0,1)$ and $p>0.$}
		\end{proposition}
		The following lemma is a direct result of Proposition \ref{p10-1} after some simple arguments.
		\begin{lemma}
			\label{p10-3}
			Let $\{W_k(t),t\in [0,T]\}_{k=1}^{\mathbb{U}}$ be a family of i.i.d. standard Wiener processes and, for
			every multi-index $\alpha= (\alpha_1,\cdots,\alpha_d)$, define $W_\alpha  = W_{1}^{\alpha_1}\cdots W_{\mathbb{U}}^{\alpha_{\mathbb{U}}} $ with the convention that $W_\alpha = 1$ if $\alpha = \emptyset$. Let furthermore $A_\alpha$ be a family of (not necessarily adapted) stochastic processes with the property that there exists $n \geq 0$ such that $A_\alpha = 0$ whenever $|\alpha|>n$ and set $Z_A(t) =\sum_\alpha A_\alpha(t)W_\alpha(t)$. Then, there exists a $\tilde \Omega $ with $\mP(\tilde \Omega )=1$ only depending on $n,T$ and $\{W_k\}_{k=1}^{\mathbb{U}}$, such that the implication
			\begin{align*}
				\|Z_{A}\|_{L^\infty} =0 \text{ and }\sup_{\alpha}\| A_\alpha \|_{Lip}<\infty
				\Rightarrow
				\sup_{\alpha}\| A_\alpha \|_{L^\infty}=0
			\end{align*}
			holds on $\omega\in \tilde \Omega$.
		\end{lemma}
		\begin{proof}
			For any  $M\in \mN$, obviously, there exists a $k_M\in \mN$ such that
			\begin{eqnarray}
				\label{pp03-1}
				\big(\frac{1}{k_M}\big)^{-3^{-(n+1)}}>M,\quad  \big(\frac{1}{k_M}\big)^{3^{-n}}<\frac{1}{M}.
			\end{eqnarray}
			Let $ \Omega^0:= \bigcup_{M=1}^\infty \bigcap_{k\geq k_M} Oscm_W^{n,\frac{1}{k}}$,
			where   $Oscm_W^{n,\eps},\eps>0$ are events  depending only on
			$n,\eps,T$ and    $W=\{W_k(t),t\in [0,T]\}_{k=1}^{\mathbb{U}}$ that are given by Proposition   \ref{p10-1}.
			For any $M\in \mN$  and $p>0$, using Proposition \ref{p10-1}, it holds that
			\begin{eqnarray*}
				\mP\big(\bigcap_{k\geq k_M} Oscm_W^{n,\frac{1}{k}} \big )\leq C_{p,n,T}
				\big( \frac{1}{k}\big)^p, \forall k\geq k_M.
			\end{eqnarray*}
			Thus, letting $k\rightarrow \infty$ in the above, we conclude that $ \mP\big(\bigcap_{k\geq k_M} Oscm_W^{n,\frac{1}{k}} \big )=0$. Furthermore, it also has $\mP\big(\Omega^0 )=0$.

			Assume that  $A_\alpha$ is  a family of  stochastic processes with the property that there exists $n \geq 0$ such that $A_\alpha = 0$ whenever $|\alpha|>n$ and  $Z_A(t) =\sum_\alpha A_\alpha(t)W_\alpha(t)$.
			For any   $M\in \mN$, denote
			\begin{eqnarray*}
				\cA_M:=\big\{\omega\in \Omega:  \|Z_{A}\|_{L^\infty}(\omega)=0, ~\sup_{\alpha}\| A_\alpha \|_{Lip}(\omega)< M  \text{ and }   \sup_{\alpha}\| A_\alpha \|_{L^\infty}> \frac{1}{M}\big\}.
			\end{eqnarray*}
			Assume that
			$\omega \in \cA_M.$
			Then, for any $k\geq k_M$, by (\ref{pp03-1}), it holds that
			\begin{eqnarray*}
				\sup_{\alpha}\| A_\alpha \|_{Lip}(\omega)< \big(\frac{1}{k}\big)^{-3^{-(n+1)}} \text{ and }
				\sup_{\alpha}\| A_\alpha \|_{L^\infty}>\big(\frac{1}{k}\big)^{3^{-n}}.
			\end{eqnarray*}
			Therefore,    with the help of  Proposition \ref{p10-1}, we conclude that
			\begin{eqnarray*}
				\cA_M  \subseteq   \bigcap_{k\geq k_M} Oscm_W^{n,\frac{1}{k}},
			\end{eqnarray*}
			which implies  $\cup_{M=1}^\infty  \cA_M  \subseteq \Omega^0 $.

			Setting   $\tilde \Omega=\Omega \setminus  \Omega^0$, we complete the proof by    $\mP\big(\Omega^0 )=0$ and the fact  $\cup_{M=1}^\infty  \cA_M  \subseteq \Omega^0   $.
		\end{proof}
		
		{ \begin{lemma}
				\label{p24-10}
				Denote $f_t=f_t(x) =u_{0}-\int_{0}^t\operatorname{div} A(u_s)\dif s+\int_{0}^t \nu\Delta u_s\dif s.$
				For any $u_0\in \tilde H^{\nn+5}, T>0, 1\leq i\leq d, \phi\in \tilde  H,\lambda \in \mN \cup\{0\}$ and  smooth function $g$ on $\mT^d$, the following function:
				\begin{eqnarray*}
					t\in [0,T]\rightarrow \big\langle K_{t,T}\phi,\partial_{x_i}\big(f_t^\lambda (x)
					g (x) \big)
					\big\rangle
				\end{eqnarray*}
				is differentiable, $\mP$-a.s. Furthermore, we have
				\begin{eqnarray*}
					\mP\Big(\sup_{\phi \in \tilde H: \|\phi\|\leq 1}\sup_{t\in [0,T]}\Big|\frac{\dif }{\dif t}\big\langle K_{t,T}\phi,\partial_{x_i}\big(f_t^\lambda (x)
					g (x) \big)
					\big\rangle \Big|<\infty\Big)=1.
				\end{eqnarray*}
		\end{lemma}}
		\begin{proof}
			
			Since  $F(u)= -\text{div} A(u)$, the following:
			\begin{eqnarray}
				\label{pp0203-8}
				\begin{split}
					&  \|  DF(u)v\|
					\\ &  \leq  C(1+\|u\|_{\nn+1}^\Bbbk)\Big(
					\int_{\mT^d}|v(x)|^2\dif x
					+\sum_{i=1}^d \int_{\mT^d}|\partial_{x_i}v(x)|^2\dif x
					\Big)^{1/2}
					,
				\end{split}
			\end{eqnarray}
			holds for any $u\in \tilde H $ and  function $v$ on $\mT^d.$
			Assume $\|\phi\|\leq1$ and denote $K_{t,T}\phi$ by $\varrho_t.$

			For the case $\lambda=0$,
			by \eqref{L inf incl}, (\ref{p0203-1}), (\ref{pp0203-8}) and Lemma  \ref{15-4}, we have
			\begin{eqnarray*}
				&&  \Big| \partial_t \big\langle \varrho_t,\partial_{x_i}
				g (x) \big)
				\big\rangle\Big|
				=
				\Big|  \big\langle
				\nu \Delta \varrho_t+(DF(u_t))^*\varrho_t,
				\partial_{x_i}
				g (x) \big\rangle\Big|
				\\ && \leq
				C_g \|\varrho_t\|
				(1+\|u_t\|_{\nn+1}^\Bbbk)
				\\  && \leq
				{ C_{T,g} \sup_{s\in [0,T]}(1+\|u_s\|_{\nn}^\mmm)}
				\cdot (1+\|u_t\|_{\nn+1}^\Bbbk), ~\forall t\in [0,T],
			\end{eqnarray*}
			where  $C_g$ is a constant depending on $g$,
			$\nu,d,\Bbbk, \{b_k\}_{k\in \cZ_0},\mathbb{U}$,
			$(c_{\mathbbm{i},\mathbbm{j}})_{1\leq \mathbbm{i}\leq d,0\leq \mathbbm{j} \leq \Bbbk}$
			and
			$C_{T,g}$ is a constant depending on $T,g$,
			$\nu,d,\Bbbk, \{b_k\}_{k\in \cZ_0},\mathbb{U}$,
			$(c_{\mathbbm{i},\mathbbm{j}})_{1\leq \mathbbm{i}\leq d,0\leq \mathbbm{j} \leq \Bbbk}.$
			In this case, the proof of this lemma is completed by the above inequality and Proposition \ref{wp}.
			So, we always assume that $\lambda\geq 1.$
			
			Before we give a proof of this lemma for
			$\lambda\geq 1$,   we demonstrate two     estimates  on $f_t.$
			With the help of \eqref{L inf incl} and Lemma \ref{div reg},
			for some $\mm_1=\mm_1(\nn,d,\Bbbk)\geq  1$, one has
			\begin{eqnarray}
				\nonumber  &&  \|\partial_{x_i}(f_t^\lambda g)\|_2+\|\partial_{x_i}(f_t^\lambda g)\|_{1}
				+ \|\partial_{x_i}(f_t^\lambda g)\|
				\\ \nonumber  &&\leq   C_{\lambda,g} (1+\|f_t\|_{L^\infty}^\lambda)\Big(1+\sum_{i,j,k=1}^d\big(\|\partial_{x_i}f_t \|_{L^\infty}^3 +\|\partial_{x_i}\partial_{x_j}f_t\|_{L^\infty}^3 + \|\partial_{x_i}\partial_{x_j}\partial_{x_k}f_t\|_{L^\infty}^3\big)\Big)
				\\  \nonumber &&\leq   C_{\lambda,g} \big(1+\|f_t\|_{\nn+3}\big)^{\lambda+3}
				\\ \nonumber &&\leq    C_{T,\lambda,g}
				\big(1+\sup_{s\in [0,T]}\|u_s\|_{\nn+5}+\sup_{s\in [0,T]} \|\text{div}A(u_s) \|_{\nn+3}\big)^{\lambda+3}
				\\  \nonumber &&\leq   C_{T, \lambda,g}
				\sup_{s\in [0,T]}(1+\|u_s\|_{\nn+5}+\|u_s\|_{L^\infty}^{\mm_1})^{\lambda+3}
				\\  \label{p0203-5}&&\leq  C_{T,\lambda,g}
				\sup_{s\in [0,T]}(1+\|u_s\|_{\nn+5})^{\mm_1(\lambda+3)},
				~\forall t\in [0,T]\text{ and }1\leq i\leq d,
			\end{eqnarray}
			where  $C_{T,\lambda,g}$ is a constant depending on $T,\lambda,g$ and $\nu,d,\Bbbk, \{b_k\}_{k\in \cZ_0},\mathbb{U}$,
			$(c_{\mathbbm{i},\mathbbm{j}})_{1\leq \mathbbm{i}\leq d,0\leq \mathbbm{j} \leq \Bbbk}$.
			By Lemma \ref{div reg}, \eqref{L inf incl}  and  the definition of $f_t$,  for some $\mm_2=\mm_2(\nn,d,\Bbbk)>1$  one gets
			\begin{eqnarray}
				\nonumber  &&  \|f_t\|_{L^\infty}+\|\partial_t f_t\|_{L^\infty}
				+\|\partial_{x_i} f_t\|_{L^\infty}
				+\|\partial_t \partial_{x_i}f_t\|_{L^\infty}
				\\  \nonumber &&\leq  C_{T}
				\sup_{s\in [0,T]}\big( \| \operatorname{div} A(u_s)\|_{L^\infty}+\| \Delta u_s\|_{L^\infty}
				+ \|\partial_{x_i}   \operatorname{div} A(u_s)\|_{L^\infty}+\|\partial_{x_i}   \Delta u_s\|_{L^\infty}
				\big)
				\\  \nonumber && \leq  C_{T}
				\sup_{s\in [0,T]}\big( \| \operatorname{div} A(u_s)\|_{\nn }+\| u_s\|_{\nn+2}+  \| \operatorname{div} A(u_s)\|_{\nn+1 }+\| u_s\|_{\nn+3}   \big)
				\\  \nonumber && \leq  C_{T}
				\sup_{s\in [0,T]}\big( 1+ \|u_s\|_{\nn+3 }+\|u_s\|_{L^\infty}^{\mm_2} \big)
				\\  \label{p0203-7} && \leq  C_{T}
				\sup_{s\in [0,T]}\big( 1+ \|u_s\|_{\nn+3 } \big)^{\mm_2 }, \quad \forall t\in [0,T],1\leq i\leq d.
			\end{eqnarray}

			Now, we will give an estimate of $\big| \big\langle  \partial_t\varrho_t,  \partial_{x_i}\big(f_t^\lambda
			g \big)
			\big\rangle\big|.$
			by
			(\ref{pp0203-8})--(\ref{p0203-5})   and  Lemma \ref{15-4},  for any
			for any $t\in [0,T]$,  we have
			\begin{eqnarray}
				\nonumber   && \Big| \big\langle \partial_t \varrho_t,\partial_{x_i}\big(f_t^\lambda(x)
				g(x) \big)
				\big\rangle\Big|
				=\Big| \big\langle  \nu \Delta \varrho_t+(DF(u_t))^*\varrho_t,
				\partial_{x_i}\big(f_t^\lambda
				g  \big)
				\big\rangle \Big|
				\\   \nonumber &&\leq  C \|\varrho_t \|\cdot \|\partial_{x_i}\big(f_t^\lambda
				g\big) \|_2+\Big| \big\langle \varrho_t,
				DF(u_t)\big(\partial_{x_i}(f_t^\lambda
				g )\big)
				\big\rangle \Big|
				\\  \nonumber &&\leq   C \|\varrho_t \| \|\partial_{x_i}\big(f_t^\lambda
				g \big) \|_2
				+ C \|\varrho_t \| (1+\|u_t\|_{\nn+1}^\Bbbk)\big(\|\partial_{x_i} (f_t^\lambda g)\|+ \|\partial_{x_i} (f_t^\lambda g)\|_1\big)
				\\   \nonumber && \leq  C_{T,\lambda,g} \sup_{s\in [0,T]}(1+\|u_s\|_{L^\mmm}^\mmm)
				\cdot  (1+\|u_t\|_{\nn+1}^\Bbbk) \cdot \sup_{s\in [0,T]}(1+\|u_s\|_{\nn+5})^{\mm_1(\lambda+3)}
				\\  \label{p0203-2}&&\leq   C_{T,g,\lambda}
				\sup_{s\in [0,T]}(1+\|u_s\|_{\nn+5}))^{\mmm+\Bbbk+\mm_1(\lambda+3)}, \quad \forall t\in [0,T],
			\end{eqnarray}
			where  $C_{T,g,\lambda}$ is a constant depending on $T,g,\lambda$ and $\nu,d,\Bbbk, \{b_k\}_{k\in \cZ_0},\mathbb{U}$,
			$(c_{\mathbbm{i},\mathbbm{j}})_{1\leq \mathbbm{i}\leq d,0\leq \mathbbm{j} \leq \Bbbk}$.

			Now we give an estimate of   $ \big\langle  \varrho_t,\partial_t  \partial_{x_i}\big(f_t^\lambda
			g \big)
			\big\rangle. $
			By  (\ref{p0203-7}),
			Lemma \ref{15-4} and Lemma \ref{div reg},
			we arrive at
			{ \begin{eqnarray*}
					&& \Big| \big\langle  \varrho_t,\partial_t  \partial_{x_i}\big(f_t^\lambda
					g \big)
					\big\rangle\Big|
					\\ && \leq  C_{g} \sup_{s\in [0,T]}\|\varrho_s\| \big( \|\partial_s \partial_{x_i} f_s^\lambda\|_{L^\infty}+ \|\partial_s f_s^\lambda\|_{L^\infty} \big)
					\\ && \leq  C_{\lambda,g} \sup_{s\in [0,T]}\|\varrho_s\|
					(1+\|f_s\|_{L^\infty}^\lambda)(\|\partial_{x_i}f_s\|_{L^\infty}
					+\|\partial_{s}f_s\|_{L^\infty}+\|\partial_s \partial_{x_i}f_s\|_{L^\infty}+\|\partial_s f_s \cdot \partial_{x_i}f_s  \|_{L^\infty} )
					\\ && \leq  C_{T,\lambda,g} \sup_{s\in [0,T]}(1+\|u_s\|_{L^\mmm}^\mmm)
					\cdot  \sup_{s\in [0,T]}\big( 1+ \|u_s\|_{\nn+3 } \big)^{\mm_2(\lambda+2) }
					\\ &&\leq  C_{T,\lambda,g}  \sup_{s\in [0,T]}\big( 1+ \|u_s\|_{\nn+3 } \big)^{\mmm+\mm_2(\lambda+2) }.
			\end{eqnarray*}}

			The proof is completed by combining the above inequality, (\ref{p0203-2}), and Proposition \ref{wp}.

		\end{proof}
		\textbf{Now we are in a position to demonstrate a proof  of   Proposition \ref{1-66}.} By direct calculations, we have
		\begin{eqnarray*}
			\langle \cM_{0,T}\phi,\phi\rangle=\sum_{j\in \cZ_0} \int_0^T \langle K_{t,T}\phi,e_j   \rangle^2 \dif t=\sum_{j\in \cZ_0} \int_0^T \langle \phi,J_{t,T}e_j   \rangle^2 \dif t.
		\end{eqnarray*}
		Let   $\tilde \Omega$  be  set that
		is given by Lemma  \ref{p10-3}. Then, $\mP( \tilde \Omega)=1$.
		Assume that
		\begin{eqnarray*}
			\omega\in \big\{\omega: \inf_{\phi\in \cS_{\alpha,N}}\langle
			\cM_{0,T}\phi,\phi\rangle=0\big\} \cap \tilde \Omega.
		\end{eqnarray*}
		Then,  for some $\phi\in \tilde H$ with
		\begin{eqnarray}
			\label{1-3}
			\|P_N\phi\|\geq \alpha,
		\end{eqnarray}
		one has
		\begin{eqnarray*}
			\langle K_{t,T}\phi,e_j   \rangle(\omega)=0,\quad \forall t\in [T/2,T].
		\end{eqnarray*}
		Assume that  we have proved
		\begin{eqnarray}
			\label{p10-2}
			\langle K_{t,T}\phi, e_k  \rangle(\omega)=0, \quad \forall t\in [T/2,T]
			\text{ and } k\in \cZ_{n-1}.
		\end{eqnarray}
		In the following, we will prove that
		\begin{eqnarray*}
			\langle K_{t,T}\phi, e_k \rangle(\omega)=0,\quad \forall t\in [T/2,T]
			\text{ and }  k \in \cZ_{n}.
		\end{eqnarray*}
		Recall that   $F(u)= -\operatorname{div} A(u)$
		and
		$\varrho_t=K_{t,T}\phi$ satisfies the following equation:
		\begin{eqnarray*}
			\partial_t \varrho_t=-\nu \Delta \varrho_t-(DF(u))^*\varrho_t,
		\end{eqnarray*}
		where $(DF(u))^*$ is the adjoint of $DF(u)$, i.e, $\langle (DF(u))^*v,w \rangle=\langle v, DF(u)w \rangle.$
		In view of
		\begin{eqnarray*}
			DF(u)v= -\text{div} DA(u)v=-\text{div} A'(u)v,
		\end{eqnarray*}
		we take derivative with respect to $t$ in (\ref{p10-2})  and get
		\begin{eqnarray*}
			\langle -\nu \Delta \varrho_t-(DF(u_t))^*\varrho_t, e_k\rangle=0, \quad \forall  t\in [T/2,T] \text{ and }k\in \cZ_{n-1},
		\end{eqnarray*}
		i.e.,
		\begin{eqnarray*}
			\langle \varrho_t, DF(u_t)e_k\rangle=0,\quad \forall t\in [T/2,T]\text{ and }k\in \cZ_{n-1}.
		\end{eqnarray*}
		Thus, for any $t\in [T/2,T]$ and $k\in \cZ_{n-1}$, we have
		\begin{eqnarray}
			\label{p26-1}
			\Big\langle \varrho_t, \sum_{i=1}^d\partial_{x_i}\big( \sum_{j=1}^{\Bbbk -1}j \cdot  c_{i,j}u^{j-1}e_k+c_i\Bbbk  u^{\Bbbk -1} e_k \big)\Big\rangle=0.
		\end{eqnarray}
		Let
		\begin{eqnarray*}
			\cA:=\big\{(\alpha_j)_{j\in \cZ_0}:\alpha_j\geq 0,\forall j\in \cZ_0 \text{ and }\sum_{j\in \cZ_0}\alpha_j=\Bbbk-1\big\}
		\end{eqnarray*}
		and  $f_t =u_{0}-\int_{0}^t\operatorname{div} A(u_s)\dif s+\int_{0}^t \nu\Delta u_s\dif s.$
		Substituting $u_t=f_t+\sum_{j\in \cZ_0}b_je_jW_j(t),t\in [0,T]$ into the  equation (\ref{p26-1}),  also with the help of Lemma \ref{p24-10}, one arrives at
		\baee
		0=& A_0(t)+\sum_{|\alpha|\leq \Bbbk-2 }A_\alpha (t)W_\alpha(t)+
		\Big\langle \varrho_t, \sum_{i=1}^d\partial_{x_i}\Big(   c_{i} \Bbbk  \cdot\big(  \sum_{j\in \cZ_0}b_je_jW_j(t) \big)^{\Bbbk-1}e_k \Big) \Big\rangle
		\\ =& A_0(t)+\sum_{|\alpha|\leq \Bbbk-2 }A_\alpha (t)W_\alpha(t)
		\\   &+
		\sum_{i=1}^d
		\sum_{\alpha=(\alpha_j)_{j\in \cZ_0}\in \cA}\prod_{j\in \cZ_0} \Bigg[ \binom{\Bbbk-1}{\alpha_j}
		\Big\langle \varrho_t,\partial_{x_i}\Big( c_i\Bbbk \cdot  \big( b_{j}^{\alpha_j}e_j^{\alpha_j}\big) e_k \Big)\Big\rangle
		W_j^{\alpha_j}(t)\Bigg]
		\\ =& A_0(t)+\sum_{|\alpha|\leq \Bbbk-1 }A_\alpha (t)W_\alpha(t)
		\\ &+
		\sum_{\alpha=(\alpha_j)_{j\in \cZ_0}\in \cA}\prod_{j\in \cZ_0} \binom{\Bbbk-1}{\alpha_j}
		\cdot
		\Big\langle \varrho_t,\sum_{i=1}^d c_i  \Bbbk  \partial_{x_i}\Big( \prod_{j\in \cZ_0}\big( b_{j}^{\alpha_j}e_j^{\alpha_j}\big) e_k \Big)\Big\rangle
		\cdot \prod_{j\in \cZ_0}W_j^{\alpha_j}(t),
		\eaee
		{  where $W_\alpha(t)=\prod_{j\in \cZ_0}W_j(t)^{\alpha_j}$,
			$|\alpha|=\sum_{j\in \cZ_0}\alpha_j $ and
			$A_0(t),A_\alpha(t)$ are some  processes such that for all $|\alpha|\leq \Bbbk-1$,}
		\begin{eqnarray*}
			\mP\Big(\sup_{\phi \in \tilde H :\|\phi\|\leq 1}\sup_{t\in [0,T]}\Big| \frac{\dif }{\dif t }A_0(t)\Big|+ \sup_{\phi \in \tilde H :\|\phi\|\leq 1} \sup_{t\in [0,T]}\Big| \frac{\dif }{\dif t }A_\alpha(t)\Big|<\infty\Big)=1.
		\end{eqnarray*}
		Therefore, by Lemma \ref{p10-3} and Lemma \ref{p24-10}, for any $\alpha=(\alpha_j)_{j\in \cZ_0}\in \cA$, we obtain
		\begin{eqnarray}
			\label{p26-2}
			\Big\langle \varrho_t,\sum_{i=1}^d  c_{i}   \partial_{x_i}\Big( \big(\prod_{j\in \cZ_0}e_j^{\alpha_j}\big) \cdot e_k \Big)\Big\rangle=0,~\forall t\in [T/2,T]\text{ and }k\in \cZ_{n-1}.
		\end{eqnarray}
		Observe that $\sin(\ell \cdot x),\cos(\ell \cdot x),\ell \in \mathbb{L}$ can be written as
		linear combinations
		of  elements in the following set:
		\begin{eqnarray*}
			\Big\{ \prod_{j\in \cZ_0} e_j^{\alpha_j}:
			j\in \cZ_0,\alpha_j\geq 0, \sum_{j\in \cZ_0} \alpha_j=\Bbbk-1
			\Big\}.
		\end{eqnarray*}
		Consequently, by (\ref{p26-2}) and the symmetry of the set $\cZ_{n-1}$(
		i.e.,  $k \in\cZ_{n-1}$ implies $-k \in\cZ_{n-1}$),  the followings:
		\begin{eqnarray*}
			&& \Big\langle \varrho_t,\sum_{i=1}^d  c_{i}   \partial_{x_i} \big(\sin(\ell \cdot x)\cdot \sin(k \cdot x) \big)\Big\rangle=0,
			\\ && \Big\langle \varrho_t,\sum_{i=1}^d  c_{i}   \partial_{x_i} \big(\sin(\ell \cdot x)\cdot \cos(k \cdot x) \big)\Big\rangle=0,
			\\ && \Big\langle \varrho_t,\sum_{i=1}^d  c_{i}   \partial_{x_i} \big(\cos(\ell \cdot x)\cdot \sin(k \cdot x) \big)\Big\rangle=0,
			\\ && \Big\langle \varrho_t,\sum_{i=1}^d  c_{i}   \partial_{x_i} \big(\cos(\ell \cdot x)\cdot \cos(k \cdot x) \big)\Big\rangle=0.
		\end{eqnarray*}
		hold for any  $t\in [T/2,T], k\in \cZ_{n-1}$ and $\ell \in \mathbb{L}.$
		By the above equalities,  for any  $t\in [T/2,T], k\in \cZ_{n-1}$ and $\ell \in \mathbb{L}$,  one arrives at
		\begin{eqnarray*}
			&&  \sum_{i=1}^d c_i(k_i+\ell_i)\Big \langle \varrho_t,\sin\big(k\cdot x+\ell\cdot x\big)\Big \rangle
			= -\Big\langle \varrho_t,\sum_{i=1}^d c_i\partial_{x_i}\cos\big(k\cdot x+\ell\cdot x\big) \Big \rangle
			\\ &&=   \Big\langle \varrho_t,\sum_{i=1}^d c_i\partial_{x_i}
			\big( \sin(k\cdot x) \sin  (\ell \cdot x) -\cos(k\cdot x)   \cos(\ell \cdot x)  \big)  \Big \rangle=0.
		\end{eqnarray*}
		With similar arguments,  for any  $t\in [T/2,T], k\in \cZ_{n-1}$ and $\ell \in \mathbb{L}$ ,  we also have
		\begin{eqnarray*}
			&& \sum_{i=1}^d c_i(k_i+\ell_i) \Big \langle \varrho_t,\sin\big(k\cdot x+\ell\cdot x\big)\Big \rangle=0.
		\end{eqnarray*}
		Thus, for   any $ t\in [T/2,T], k\in \cZ_{n-1}$ and $\ell\in \mathbb{L}$ such that
		$\sum_{i=1}^dc_i(k_i+\ell_i)\neq 0$,
		\begin{eqnarray*}
			\Big \langle \varrho_t, \sin\big(k\cdot x+\ell\cdot x\big)\Big\rangle =  \Big \langle \varrho_t, \cos\big(k\cdot x+\ell\cdot x\big)\Big\rangle =0.
		\end{eqnarray*}
		By the symmetry of $\cZ_{n-1},\mathbb{L}$ and the definition of $\cZ_n$,  the above implies
		\begin{eqnarray*}
			\langle \varrho_t,e_{k} \rangle=0,\quad \forall k\in \cZ_n \text{ and }t\in [T/2,T].
		\end{eqnarray*}
		By the above arguments, we arrive at
		\begin{eqnarray*}
			&& \langle \varrho_t,e_k\rangle=0,~\forall t\in [T/2,T] \text{ and } k\in \cZ_{n-1}
			\\ &&  \Rightarrow \langle \varrho_t,e_k\rangle=0,~\forall t\in [T/2,T]\text{ and }  k\in \cZ_{n}.
		\end{eqnarray*}
		By the definition of $\tilde H^\nn$, we conclude that
		\begin{eqnarray*}
			\langle \varrho_t,e_k \rangle
			= 0,\quad
			\forall t\in [T/2,T] \text{ and }  k  \in \cup_{n=0}^\infty \cZ_n
		\end{eqnarray*}
		and
		\begin{eqnarray*}
			\langle \varrho_t,\phi  \rangle
			= 0,\quad
			\forall t\in [T/2,T] \text{ and } \phi \in \tilde H.
		\end{eqnarray*}
		Setting  $t=T$ in the above,  one sees that
		$
		\phi=0,
		$
		which contradicts with (\ref{1-3}).

		\subsection{ 	{Proof of Proposition \ref{3-8}}   }
		\label{s3-2}

		Assume that  the  (\ref{p25-4}) were  wrong,  then   there exist  sequences   $\{u_0^{(k)} \}_{k\geq1} \subseteq  \{
		w\in  \tilde  H^{\nn+5}:  \|w\|_{\nn+5}<\mathfrak{R}
		\}, \{\eps_k\}_{k\geq1}\subseteq (0,1)$  and a positive number $\delta_0$  such that
		\begin{eqnarray}
			\label{D-1}
			\lim_{k\rightarrow \infty} \mP(  X^{u_0^{(k)},\alpha,N}<\eps_k)\geq \delta_0>0  \text{ and  } \lim_{k\rightarrow \infty}\eps_k=0.
		\end{eqnarray}
		Our  strategy    is to find something
		contradicts  with (\ref{D-1}).

		Since $\tilde H^{\nn+5}$ is a Hilbert  space, there exists a subsequence $\{u_0^{(n_k)},k\geq 1\}$ of  $\{u_0^{(k)},k\geq 1\}$ and an element  $u_0^{(0)}\in \tilde H^{\nn+5}$
		such that
		$u_0^{(n_k)}$
		converges weakly to  $u_0^{(0)}$
		in $\tilde H^{\nn+5}.$
		Therefore,  with regard to  $u_0^{(0)}$, it holds that
		\begin{eqnarray*}
			\|u_0^{(0)}\|_{\nn+5}\leq \liminf_{k\rightarrow \infty}\|u_0^{(k)}\|_{\nn+5}\leq \mathfrak{R}.
		\end{eqnarray*}
		
		For the convenience of writing,   we still denote this subsequence $\{u_0^{(n_k)},k\geq 1\}$   by $\{u_0^{(k)},k\geq 1\}.$
		Considering the equation (\ref{1-1}), when  $u_t|_{t=0}=u_0^{(k)}(k\geq 0)$,
		we denote its solution by $u_t^{(k)}$. For $k\in \mN \cup\{0\},\xi\in \tilde H$ and $s\in [0,\infty)$, let    $ J_{s,t}^{(k)}\xi, t\geq s$ be the solution to  the following equation:
		\begin{eqnarray*}
			\left\{
			\begin{split}
				&  \partial_t   J_{s,t}^{(k)}\xi-   \nu \Delta J_{s,t}^{(k)} \xi +\sum_{i=1}^d \partial_{x_i}\big(A_i'(u_t^{(k)})J_{s,t}^{(k)} \xi  \big) =0,
				\\ &   J_{s,s}^{(k)} \xi = \xi.
			\end{split}
			\right.
		\end{eqnarray*}
		
		As before,  $C$  denotes   a constant depending
		$\nu,d,\Bbbk, \{b_k\}_{k\in \cZ_0},\mathbb{U},
		(c_{\mathbbm{i},\mathbbm{j}})_{1\leq \mathbbm{i}\leq d,0\leq \mathbbm{j} \leq \Bbbk}.$
		$C_{\mathfrak{R}}$  denotes  a constant depending on $\mathfrak{R}$
		and   $\nu,d,\Bbbk, \{b_k\}_{k\in \cZ_0},\mathbb{U},(c_{\mathbbm{i},\mathbbm{j}})_{1\leq \mathbbm{i}\leq d,0\leq \mathbbm{j} \leq \Bbbk}$.
		The values of these  constants may change from line to line. Recall that
		$
		\mmm=40 \Bbbk d(d+14\Bbbk)^2.
		$
		\begin{lemma}
			\label{p06-2}
			With probability one,  for any $0\leq s\leq t\leq 1,\xi\in \tilde H$ and $k\in \mN$, one has
			\begin{eqnarray*}
				&& \| J_{s,t}^{(k)}\xi-J_{s,t}^{(0)}\xi\|^2
				\\ && \leq C\exp\Big\{
				C \int_s^t \|u_r^{(0)}\|_{L^\mmm}^\mmm+\|u_r^{(k)}\|_{L^\mmm}^\mmm  \dif r
				\Big\}
				\Big(\int_s^t  \|u_r^{(k)}-u_r^{(0)}\|_{L^1} \dif r \Big)^{\frac{1}{8d}}\|\xi\|^2.
			\end{eqnarray*}
		\end{lemma}
		\begin{proof}
			By direct calculations, we have
			\begin{eqnarray}
				\nonumber &&\frac{\dif \| J_{s,t}^{(k)}\xi-J_{s,t}^{(0)}\xi\|^2 }{\dif t}\\
				\nonumber &&  = -2\nu \|J_{s,t}^{(k)}\xi-J_{s,t}^{(0)}\xi\|_1^2
				\\   \nonumber & &\quad\quad {  - 2\sum_{i=1}^d  \Big\langle J_{s,t}^{(k)}\xi-J_{s,t}^{(0)}\xi,
					\partial_{x_i}\big(A_i'(u_t^{(k)})J_{s,t}^{(k)} \xi  \big)- \partial_{x_i}\big(A_i'(u_t^{(0)})J_{s,t}^{(0)} \xi  \big)   \Big  \rangle}
				\\  \nonumber &&  =  -2\nu \|J_{s,t}^{(k)}\xi-J_{s,t}^{(0)}\xi\|_1^2
				\\   \nonumber &&\quad\quad -  \sum_{i=1}^d2 \Big\langle J_{s,t}^{(k)}\xi-J_{s,t}^{(0)}\xi,
				\partial_{x_i}\big(A_i'(u_t^{(k)})J_{s,t}^{(k)} \xi  \big)-\partial_{x_i}\big(A_i'(u_t^{(0)})J_{s,t}^{(k)} \xi  \big)   \Big  \rangle
				\\   \nonumber &&\quad\quad - \sum_{i=1}^d  2 \Big\langle J_{s,t}^{(k)}\xi-J_{s,t}^{(0)}\xi,
				\partial_{x_i}\big(A_i'(u_t^{(0)})J_{s,t}^{(k)} \xi  \big)- \partial_{x_i}\big(A_i'(u_t^{(0)})J_{s,t}^{(0)} \xi  \big)   \Big  \rangle
				\\ \label{p05-2}  && :=I_1+\sum_{i=1}^d I_{2,i}+\sum_{i=1}^d I_{3,i}.
			\end{eqnarray}
			
			%
			Let $p=\frac{8d}{4d-1}\in (2,3)$.By Gagliardo-Nirenberg's inequality, it holds that
			\begin{eqnarray}
				\label{p1203-3}
				\|w\|_{L^p}\leq C\|w\|_1^{1/8}\|w\|^{7/8},\forall w\in H.
			\end{eqnarray}
			Thus, by H\"older's inequality, one arrives at
			\begin{align*}
				|I_{2,i}|   \leq &     C \|J_{s,t}^{(k)}\xi-J_{s,t}^{(0)}\xi\|_1
				\|A_i'(u_t^{(k)})-A_i'(u_t^{(0)})\|_{L^{2p/(p-2)}}
				\|J_{s,t}^{(k)}\xi\|_{L^p}
				\\  \leq &     C  \|J_{s,t}^{(k)}\xi-J_{s,t}^{(0)}\xi\|_1
				\|A_i'(u_t^{(k)})-A_i'(u_t^{(0)})\|_{L^{2p/(p-2)}}
				\|J_{s,t}^{(k)}\xi\|_{1}^{1/8} \|J_{s,t}^{(k)}\xi\|^{7/8}
				\\   \leq &    \frac{\nu}{4} \|J_{s,t}^{(k)}\xi-J_{s,t}^{(0)}\xi\|_1^2+
				C \|A_i'(u_t^{(k)})-A_i'(u_t^{(0)})\|_{L^{2p/(p-2)}}^2
				\|J_{s,t}^{(k)}\xi\|_{1}^{1/4} \|J_{s,t}^{(k)}\xi\|^{7/4}
			\end{align*}
			and
			\begin{align*}
				|I_{3,i}|  \leq &
				C \|J_{s,t}^{(k)}\xi-J_{s,t}^{(0)}\xi\|_1
				\|A_i'(u_t^{(0)})\|_{L^{2p/(p-2)}}
				\|J_{s,t}^{(k)}\xi-J_{s,t}^{(0)}\xi\|_{L^p}
				\\  \leq &    C \|J_{s,t}^{(k)}\xi-J_{s,t}^{(0)}\xi\|_1^{9/8}
				\|A_i'(u_t^{(0)})\|_{L^{2p/(p-2)}}
				\|J_{s,t}^{(k)}\xi-J_{s,t}^{(0)}\xi\|_{}^{7/8}
				\\  \leq &    \frac{\nu}{4}  \|J_{s,t}^{(k)}\xi-J_{s,t}^{(0)}\xi\|_1^{2}
				+ \|A_i'(u_t^{(0)})\|_{L^{2p/(p-2)}}^{16/7}
				\|J_{s,t}^{(k)}\xi-J_{s,t}^{(0)}\xi\|_{}^{2}.
			\end{align*}
			Combining the estimates of $I_{2,i},I_{3,i},i=1,\cdots,d$ with (\ref{p05-2}), and invoking H\"older's inequality again
			, we derive
			\begin{eqnarray}
				\nonumber && \| J_{s,t}^{(k)}\xi-J_{s,t}^{(0)}\xi\|^2
				\\ \nonumber && \leq C\exp\Big\{
				C \int_s^t \|u_r^{(0)}\|_{L^\mmm}^\mmm+\|u_r^{(k)}\|_{L^\mmm}^\mmm  \dif r
				\Big\}
				\big(\sum_{i=1}^d \int_s^t  \|A_i'(u_r^{(k)})-A_i'(u_r^{(0)})\|_{L^{2p/(p-2)}}^4 \dif r \big)^{1/2}
				\\  \label{p06-1} &&\quad\quad\quad\quad \times \Big(\int_s^t \|J_{s,r}^{(k)}\xi\|_1\dif r \Big)^{1/4}\Big(\int_s^t \|J_{s,r}^{(k)}\xi\|^7  \dif r \Big)^{1/4}.
			\end{eqnarray}
			In view of Cauchy-Schwarz's inequality,  it holds that 
			\begin{eqnarray}
				\nonumber &&  \|A_i'(u_r^{(k)})-A_i'(u_r^{(0)})\|_{L^{2p/(p-2)}}^{4 }
				=\Big(\int_{\mT^d}\big(A_i'(u_r^{(k)})-A_i'(u_r^{(0)})\big)^{2p/(p-2)}\dif x\Big)^{(2p-4)/p}
				\\ \label{p1203-2} &&\leq C  \Big(\int_{\mT^d}|u_r^{(k)}-u_r^{(0)}|^{1/2}
				(1+|u_r^{(k)}|^{n}+|u_r^{(0)}|^{n})
				\dif x\Big)^{(2p-4)/p}
				\\  \nonumber && \leq C  \|u_r^{(k)}-u_{r}^{(0)} \|_{L^1}^{(p-2)/p }\Big(\int_{\mT^d} (1+|u_r^{(k)}|^{2n} +|u_r^{(0)}|^{2n})\dif x\Big)^{(p-2)/p},
			\end{eqnarray}
			where $n=\max\{\frac{2p(\Bbbk-1)}{p-2}-\frac{1}{2},0\}.$
			The H\"older's inequality further bounds the difference by
			\begin{eqnarray*}
				&& \int_s^t \|A_i'(u_r^{(k)})-A_i'(u_r^{(0)})\|_{L^{2p/(p-2)}}^{4 }\dif r
				\\ &&\leq C \Big(\int_s^t  \|u_r^{(k)}-u_{r}^{(0)} \|_{L^1}^{(p-2)/p} \big(1+\|u_r^{(k)}\|_{L^{2n}}^{2n(p-2)/p}
				+
				\|u_r^{(0)}\|_{L^{2n}}^{2n(p-2)/p }\big)   \dif r\Big)
				\\ &&\leq    C \Big(\int_s^t  \|u_r^{(k)}-u_{r}^{(0)} \|_{L^1} \dif r\Big)^{(p-2)/p}\Big(\int_s^t (1+\|u_r^{(k)}\|_{L^{2n}}^{n(p-2)}
				+
				\|u_r^{(0)}\|_{L^{2n}}^{n(p-2)})   \dif r\Big)^{2/p}.
			\end{eqnarray*}
			Combining the above inequality with (\ref{p06-1}), in view of
			$\mmm=40 \Bbbk d(d+14\Bbbk)^2$ and Lemma \ref{15-4},
			one arrives at that
			\begin{eqnarray*}
				&& \| J_{s,t}^{(k)}\xi-J_{s,t}^{(0)}\xi\|^2
				\\ && \leq C\exp\Big\{
				C \int_s^t \|u_r^{(0)}\|_{L^\mmm}^\mmm+\|u_r^{(k)}\|_{L^\mmm}^\mmm  \dif r
				\Big\}
				\Big(\int_s^t  \|u_r^{(k)}-u_r^{(0)}\|_{L^1} \dif r \Big)^{\frac{p-2}{2p} }
				\\ && \quad\quad\quad\quad \times\Big( \int_s^t (1+\|u_r^{(k)}\|_{L^{2n}}^{(n(p-2)}
				+
				\|u_r^{(0)}\|_{L^{2n}}^{n (p-2)}) \dif r\Big)^{1/p}\|\xi\|^2
				\\ && \leq C\exp\Big\{
				C \int_s^t \big(\|u_r^{(0)}\|_{L^\mmm}^\mmm+\|u_r^{(k)}\|_{L^\mmm}^\mmm\big)  \dif r
				\Big\}
				\Big(\int_s^t  \|u_r^{(k)}-u_r^{(0)}\|_{L^1} \dif r \Big)^{\frac{1}{8d}}\|\xi\|^2.
			\end{eqnarray*}
			The proof is complete.
		\end{proof}

		\begin{lemma}
			\label{E-5}
			For any $t\in [0,1]$, $k\in \mN$ and $M\geq \max\{|j|,   j\in \cZ_0\}$,  one has
			\begin{eqnarray*}
				&& \|Q_M u_t^{(k)} \|^2\leq  e^{-\nu M^2t }
				\|Q_Mu_0^{(k)}\|^2
				+\frac{C (\int_0^t \|u_s^{(k)}\|_{L^\mmm}^{\mm}\dif s +1) }{M }
			\end{eqnarray*}
			and
			\begin{eqnarray}
				\label{p1203-1}
				\int_0^t \|Q_M u_r^{(k)} \|_1^2\dif r\leq C\Big (\int_0^t \|u_s^{(k)}\|_{L^{\mm}}^{{\mm}}\dif s +1\Big).
			\end{eqnarray}
		\end{lemma}
		
		\begin{proof}
			By direct calculations,
			one has
			\begin{eqnarray*}
				&& \sum_{i=1}^d \big| \big\langle  \partial_{x_i}A_i(u_t^{(k)}),Q_M u_t^{(k)}\big\rangle\big| = \sum_{i=1}^d \big| \big\langle  A_i(u_t^{(k)}),\partial_{x_i}\big(Q_M u_t^{(k)}\big) \big\rangle\big|
				\\ &&  \leq
				C(1+\|u_t^{(k)}\|_{L^{2\Bbbk}}^{2\Bbbk})  \|Q_M u_t^{(k)}\|_1
				\\ &&\leq  \frac{\nu}{4}\|Q_M u_t^{(k)}\|_1^2+ C (1+\|u_t^{(k)}\|_{L^{4\Bbbk}}^{4\Bbbk}),
			\end{eqnarray*}
			The above inequality  implies
			\begin{eqnarray*}
				\frac{\dif}{\dif t}  \|Q_M u_t^{(k)}\|^2\leq -\nu \|Q_M u_t^{(k)}\|_1^2
				+C (1+\|u_t^{(k)}\|_{L^{4\Bbbk}}^{4\Bbbk}).
			\end{eqnarray*}
			Therefore, in view of $m\geq {8\Bbbk}$ and  $\|Q_M u_t^{(k)}\|_1^2\geq M^2 \|Q_M u_t^{(k)}\|^2$,
			one arrives at  (\ref{p1203-1}) and{
				\begin{eqnarray*}
					&& \|Q_Mu_t^{(k)}\|^2\leq  e^{-\nu M^2t }
					\|Q_Mu_0^{(k)}\|^2
					+C \int_0^t e^{-\nu M^2(t-s )}  (1+\|u_s^{(k)}\|_{L^{4\Bbbk}}^{{4\Bbbk}})\dif s
					\\ &&\leq   e^{-\nu M^2t }
					\|Q_Mu_0^{(k)}\|^2
					+C \Big(\int_0^t e^{-2\nu M^2(t-s )} \dif s \Big)^{1/2}
					\Big(\int_0^t (1+\|u_s^{(k)}\|_{L^{4\Bbbk}}^{{8\Bbbk}})\dif s \Big)^{1/2}
					\\ &&\leq  e^{-\nu M^2t }
					\|Q_Mu_0^{(k)}\|^2
					+\frac{C \big(\int_0^t \|u_s^{(k)}\|_{L^\mmm}^{\mmm}\dif s +1\big) }{M }.
			\end{eqnarray*}}
			The proof is complete.
		\end{proof}

		\begin{lemma}
			\label{24-1}
			For any $t\in [0,1]$ and $k,M\in \mN$,
			one has
			\begin{eqnarray*}
				&& \|P_M u_t^{(k)}-P_M u_t^{(0)}\|^2
				\\ &&  \leq
				C \exp\Big\{C \int_0^t \big(\|u_r^{(k)}\|_{L^\mmm}^\mmm+ \|u_r^{(0)}\|_{L^\mmm}^\mmm\big)\dif r \Big\}
				\Big[\|P_M u_0^{(k)}-P_M u_0^{(0)}\|^2 +\frac{1}{\sqrt{M}}   \Big].
			\end{eqnarray*}
		\end{lemma}
		\begin{proof}
			One easily sees that
			\begin{eqnarray}
				\label{E-3}
				\begin{split}
					& \dif  \|P_M u_t^{(k)}-P_M u_t^{(0)}\|^2 = - 2\nu \|P_M u_t^{(k)}-P_M u_t^{(0)}\|_1^2\dif t
					\\ &\quad  {  - 2 \sum_{i=1}^d   \langle P_M u_t^{(k)}-P_M u_t^{(0)},
						\partial_{x_i}A_i(u_t^{(k)})- \partial_{x_i}A_i(u_t^{(0)})     \rangle\dif t.}
				\end{split}
			\end{eqnarray}
			By direct calculations,  for  $p=\frac{8d}{4d-1}\in (2,3)$ and $n=\frac{2p}{p-2}\Bbbk=8 d \Bbbk$,  one has
			\begin{eqnarray*}
				&& \sum_{i=1}^d   \big| \langle P_M u_t^{(k)}-P_M u_t^{(0)},
				\partial_{x_i}A_i(u_t^{(k)})- \partial_{x_i}A_i(u_t^{(0)})     \rangle\big|
				\\ &&= \sum_{i=1}^d  \big| \langle \partial_{x_i} \big(P_M u_t^{(k)}\big)-
				\partial_{x_i} \big(P_M u_t^{(0)}\big),
				A_i(u_t^{(k)})- A_i(u_t^{(0)})     \rangle\big|
				\\ &&\leq
				C \sum_{i=1}^d  \|P_M u_t^{(k)}-P_M u_t^{(0)}\|_1
				\|u_t^{(k)}- u_t^{(0)} \|_{L^p}
				\cdot \big\|1+|u_t^{(k)}|^{\Bbbk}+|u_t^{(0)}|^{\Bbbk}\big\|_{L^{2p/(p-2)}}
				\\
				&&\leq C \sum_{i=1}^d \|P_M u_t^{(k)}-P_M u_t^{(0)}\|_1
				\|P_Mu_t^{(k)}-P_M u_t^{(0)} \|_{L^p}
				\cdot (1+\|u_t^{(k)}\|_{L^n}^n+\|u_t^{(0)}\|_{L^n}^n)
				\\ && \quad + C\sum_{i=1}^d  \|P_M u_t^{(k)}-P_M u_t^{(0)}\|_1
				\|Q_Mu_t^{(k)}- Q_Mu_t^{(0)} \|_{L^p}
				\cdot(1+\|u_t^{(k)}\|_{L^n}^n+\|u_t^{(0)}\|_{L^n}^n)
				\\ &&=\sum_{i=1}^d  I_{1,i}+\sum_{i=1}^d I_{2,i}.
			\end{eqnarray*}
			By (\ref{p1203-3}), for any $1\leq i\leq d$,  it holds that
			\begin{eqnarray*}
				&& I_{1,i} \leq  C \|P_M u_t^{(k)}-P_M u_t^{(0)}\|_1^{9/8}
				\|P_Mu_t^{(k)}-P_M u_t^{(0)} \|^{7/8}
				\cdot(1+\|u_t^{(k)}\|_{L^n}^n+\|u_t^{(0)}\|_{L^n}^n)
				\\ &&\leq  \frac{\nu}{4}\|P_M u_t^{(k)}-P_M u_t^{(0)}\|_1^{2}
				\\ && \quad\quad+ \|P_Mu_t^{(k)}-P_M u_t^{(0)} \|^{2}
				\cdot  (1+\|u_t^{(k)}\|_{L^n}^n+\|u_t^{(0)}\|_{L^n}^n)^{16/7}.
			\end{eqnarray*}
			and
			\begin{eqnarray*}
				&& I_{2,i} \leq C\|P_M u_t^{(k)}-P_M u_t^{(0)}\|_1
				\|Q_Mu_t^{(k)}- Q_Mu_t^{(0)} \|_{1}^{1/8}\|Q_Mu_t^{(k)}- Q_Mu_t^{(0)} \|^{7/8}
				\\ &&  \times
				(1+\|u_t^{(k)}\|_{L^n}^n+\|u_t^{(0)}\|_{L^n}^n)
				\\ && \leq \frac{\nu }{4}\|P_M u_t^{(k)}-P_M u_t^{(0)}\|_1^2
				+C \|Q_Mu_t^{(k)}- Q_Mu_t^{(0)} \|_{1}^{1/4}\|Q_Mu_t^{(k)}- Q_Mu_t^{(0)} \|^{7/4}
				\\ && \quad\quad \times
				(1+\|u_t^{(k)}\|_{L^n}^n+\|u_t^{(0)}\|_{L^n}^n)^2.
			\end{eqnarray*}
			Therefore,  by Lemma \ref{E-5} and the fact  $\mmm=40 \Bbbk d(d+14\Bbbk)^2$,  for any $t\in [0,1]$,  one gets
			\begin{eqnarray*}
				&&  \|P_M u_t^{(k)}-P_M u_t^{(0)}\|^2
				{  \exp\Big\{-C \int_0^t \big(\|u_r^{(k)}\|_{L^n}^n+ \|u_r^{(0)}\|_{L^n}^n+1\big)^{16/7}\dif r \Big\}}
				\\ && \leq   \int_0^t \|Q_Mu_s^{(k)}- Q_Mu_s^{(0)} \|_1^{1/4}\|Q_Mu_s^{(k)}- Q_Mu_s^{(0)} \|^{7/4}
				(\|u_s^{(k)}\|_{L^n}^n+ \|u_s^{(0)}\|_{L^n}^n+1)^3  \dif s
				\\ && \quad\quad+ \|P_M u_0^{(k)}-P_M u_0^{(0)}\|^2
				\\ && \leq \Big(\int_0^t \|Q_Mu_s^{(k)}- Q_Mu_s^{(0)} \|_1^2\dif s \Big)^{1/8}
				\Big( \int_0^t \|Q_Mu_s^{(k)}- Q_Mu_s^{(0)} \|^7\dif s \Big)^{1/4}
				\\ && \quad \quad  \times  \big(\int_0^t  (\|u_s^{(k)}\|_{L^n}^n+ \|u_s^{(0)}\|_{L^n}^n+1)^{24/5} \dif s\big)^{5/8}
				+ \|P_M u_0^{(k)}-P_M u_0^{(0)}\|^2
				\\ &&\leq C_\mathfrak{R}  \exp\Big\{C \int_0^t \big(\|u_r^{(k)}\|_{L^\mmm}^\mmm+ \|u_r^{(0)}\|_{L^\mmm}^\mmm\Big)\dif r
				\big\}\Big( \int_0^t (\|Q_Mu_s^{(k)}\|^7+\|Q_Mu_s^{(0)} \|^7)\dif s \Big)^{1/4}
				\\ && \quad
				\quad +\|P_M u_0^{(k)}-P_M u_0^{(0)}\|^2
				\\ &&\leq C_\mathfrak{R}  \exp\Big\{C \int_0^t \big(\|u_r^{(k)}\|_{L^\mmm}^\mmm+ \|u_r^{(0)}\|_{L^\mmm}^\mmm\big)\dif r
				\Big\}   \Big[\frac{1}{\sqrt{M}}+\|P_M u_0^{(k)}-P_M u_0^{(0)}\|^2\Big].
			\end{eqnarray*}
			The above inequality yields  the desired result and the proof is complete.
		\end{proof}
		
			\textbf{Now we are in a position to complete the   proof of  (\ref{p25-4})  in Proposition  \ref{3-8}.}
			With the help of Lemma   \ref{p06-2},  for any  $ r\in [\frac{1}{2},1]$ and  { $\phi \in \tilde H$} with $\|\phi\|=1$, we have
			\begin{eqnarray}
				\nonumber  && \|J_{r,1}^{(k)}\phi-J_{r,1}^{(0)}\phi \|^2 \leq
				C e^{C \int_{0}^1
					(\|u_s^{(k)}\|_{L^\mm}^{\mm}
					+
					\|u_s^{(0)}\|_{L^\mm}^{\mm})\dif r} \cdot \sup_{t \in [\frac{1}{2},1]}\|u_t^{(k)}-u_t^{(0)}\|^{\frac{1}{8d}}.
				\\ \label{p06-3}
			\end{eqnarray}
			With the help of Lemmas \ref{E-5} and \ref{24-1}, for any $t\in [1/2,1]$, one gets
			\begin{eqnarray*}
				\|u_t^{(k)}-u_t^{(0)}\|^2\leq
				C_\mathfrak{R}  \exp\Big\{C \int_0^1 \big(\|u_s^{(k)}\|_{L^\mm}^\mm+ \|u_s^{(0)}\|_{L^\mm}^\mm\big)\dif r \Big\}
				\Big[\|P_M u_0^{(k)}-P_M u_0^{(0)}\|^2 +\frac{1}{\sqrt{M}}   \Big].
			\end{eqnarray*}
			Therefore, for any $j\in \cZ_0$ and $r\in [1/2,1]$, we have
			\begin{eqnarray}
				\label{p06-4}
				\begin{split}
					& \|J_{r,1}^{(k)}e_j-J_{r,1}^{(0)}e_j \|^2\leq   C_\mathfrak{R}  \exp\Big\{C \int_0^1 \big(\|u_s^{(k)}\|_{L^\mm}^\mm+ \|u_s^{(0)}\|_{L^\mm}^\mm\big)\dif r \Big\}
					\\
					& \quad\quad\quad\quad \times
					\Big(\|P_M u_0^{(k)}-P_M u_0^{(0)}\|^{\frac{1}{8d}} +M^{-\frac{1}{32d}}    \Big).
				\end{split}
			\end{eqnarray}
			Note that
			\bae\label{p1203-4}
			\langle \phi,
			&J_{r,1}^{(k)}  e_j\rangle^2
			=\Big(\langle \phi,
			J_{r,1}^{(0)} e_j\rangle  +\langle \phi,
			J_{r,1}^{(k)} e_j-J_{r,1}^{(0)}e_j \rangle\Big)^2
			\\  \geq& \frac{1}{2}\langle \phi,
			J_{r,1}^{(0)} e_j\rangle ^2-3\langle \phi,
			J_{r,1}^{(k)} e_j-J_{r,1}^{(0)}e_j \rangle^2
			\\    \geq&
			\frac{1}{2}\langle \phi,
			J_{r,1}^{(0)} e_j\rangle^2-3\| J_{r,1}^{(k)}e_j-J_{r,1}^{(0)}e_j\|^2,
			~{  \forall \phi  \in \tilde H \text{ with }\|\phi\|\leq 1.}
			\eae
			Also recall that $K_{r,t}$ is the adjoint of $J_{r,t}$.
			It follows from the above inequality, (\ref{p06-4}) and (\ref{p1203-4})  that
			\begin{eqnarray*}
				&& \mP\Big(\inf_{\phi\in \cS_{\alpha,N}}\sum_{j\in \cZ_0}\int_{0}^1 \langle K_{r,1}^{(k)}\phi,
				e_j\rangle^2\dif r<\eps_k\Big)
				=\mP\Big(\inf_{\phi\in \cS_{\alpha,N}}\sum_{j\in \cZ_0}\int_{0}^1 \langle \phi,
				J_{r,1}^{(k)} e_j\rangle^2\dif r<\eps_k\Big)
				\\ &&\leq \mP\Big(\frac{1}{2}\inf_{\phi\in \cS_{\alpha,N} }\sum_{j\in \cZ_0}\int_{1/2}^1 \langle \phi,
				J_{r,1}^{(0)} e_j\rangle^2\dif r<\eps_k
				+3 \sum_{j\in \cZ_0}  \sup_{r\in [1/2,1]}\|J_{r,1}^{(k)}e_j-J_{r,1}^{(0)}e_j\|^2  \Big)
				\\ &&\leq
				\mP\Bigg(\frac{1}{2}\inf_{\phi\in \cS_{\alpha,N} }\sum_{j\in \cZ_0}\int_{1/2}^1 \langle \phi,
				J_{r,1}^{(0)} e_j\rangle^2\dif r<\eps_k
				\\ && \quad\quad  + C_\mathfrak{R} \exp\Big\{C \int_0^1 \big(\|u_r^{(k)}\|_{L^\mm}^\mm+ \|u_r^{(0)}\|_{L^\mm}^\mm\big)\dif r \Big\}
				\big(\|P_M u_0^{(k)}-P_M u_0^{(0)}\|^{\frac{1}{8d}} +M^{-\frac{1}{32d}}    \big)    \Bigg).
			\end{eqnarray*}
			Therefore, for any $k\geq 1, M\in \mN, \mathcal{K}>0$, we deduce that
			\begin{eqnarray}
				\nonumber  && \mP\Big(\inf_{\phi\in \cS_{\alpha,N}}\sum_{j\in \cZ_0}\int_{0}^1 \langle K_{r,1}^{(k)}\phi,
				e_j\rangle^2\dif r<\eps_k\Big)
				\\  \nonumber  &&\leq
				\mP\Big(\inf_{\phi\in \cS_{\alpha,N}}\sum_{j\in \cZ_0} \int_{1/2}^1 \langle K_{r,1}^{(0)}\phi,
				e_j\rangle^2\dif r<2 \eps_k+
				C_\mathfrak{R} e^{C\cK }   \big(\|P_M u_0^{(k)}-P_M u_0^{(0)}\|^{\frac{1}{8d}} +M^{-\frac{1}{32d}}    \big)   \Big)
				\\  \nonumber   && \quad \quad\quad\quad \quad \quad\quad\quad  +\mP\Big( \int_0^1 \big(\|u_r^{(k)}\|_{L^\mm}^\mm+ \|u_r^{(0)}\|_{L^\mm}^\mm\big)\dif r >\cK  \Big)
				\\ \nonumber &&
				\leq
				\mP\Big(\inf_{\phi\in \cS_{\alpha,N}}\sum_{j\in \cZ_0} \int_{1/2}^1 \langle K_{r,1}^{(0)}\phi,
				e_j\rangle^2\dif r<2 \eps_k+
				C_\mathfrak{R} e^{C\cK }   \big(\|P_M u_0^{(k)}-P_M u_0^{(0)}\|^{\frac{1}{8d}} +M^{-\frac{1}{32d}}    \big) \Big)
				\\  \label{2032}    && \quad \quad\quad\quad \quad \quad\quad\quad  +\frac{\sup_{k\in \mN}\mE
					\int_0^1 \big(\|u_r^{(k)}\|_{L^\mm}^\mm+ \|u_r^{(0)}\|_{L^\mm}^\mm\big)\dif r
				}{\cK }.
			\end{eqnarray}
			Letting $k\rightarrow \infty$ in (\ref{2032}), by (\ref{D-1}) and Lemma \ref{L^p 2.2},  one sees that
			\begin{eqnarray}
				\label{B-1}
				&& \delta_0\leq \mP\Big(\inf_{\phi\in\cS_{\alpha,N}}\sum_{j\in \cZ_0} \int_{1/2}^1 \langle K_{r,1}^{(0)}\phi,
				e_j\rangle^2\dif r \leq
				\frac{C_\mathfrak{R} e^{C \cK}   }{M^{1/(32d)}}\Big)  +\frac{C_\mathfrak{R}    }{\mathcal{K}}.
			\end{eqnarray}
			In (\ref{B-1}),   first letting $M\rightarrow \infty$ and then letting $\mathcal{K} \rightarrow \infty$,
			we conclude that
			\begin{eqnarray}
				\label{88-7}
				\delta_0 \leq \mP\Big(\inf_{\phi\in\cS_{\alpha,N}}\sum_{j\in \cZ_0} \int_{1/2}^1 \langle
				K_{r,1}^{(0)}\phi,
				e_j\rangle^2\dif r=0\Big).
			\end{eqnarray}
			On the other hand,  (\ref{86-1}) implies that
			\begin{eqnarray*}
				&& \mP \Big(\inf_{\phi\in \cS_{\alpha,N}}\sum_{j\in \cZ_0} \int_{1/2}^1 \langle
				K_{r,1}^{(0)}\phi,
				e_j\rangle^2\dif r=0 \Big)
				=0.
			\end{eqnarray*}
			It conflicts with (\ref{88-7}) and the proof is  complete.

			\section{Proof of Proposition \ref{3-11}}\label{proof of 1.6}
			
				The proof of Proposition \ref{3-11} is established using a localized method. Below, we provide a detailed explanation of its main ideas.
				Let $\chi:\mR\rightarrow [0,1]$ be a smooth function such that
				\begin{eqnarray}
					\label{p1110-3}
					&& \chi(x)=
					\left\{
					\begin{split}
						&  1,  \quad x\in  (-\infty, -4],
						\\
						&  0, \quad x\geq -2,
						\\ &
						\in (0,1),  \quad x\in [-4,-2]
					\end{split}
					\right.
					\text{ and }  | \chi'(x)|\leq 1,\quad \forall x\in \mR.
				\end{eqnarray}
				Recall that
				$ \mmm=40 \Bbbk d(d+14\Bbbk)^2$.  For any $\Upsilon >0,u_0\in \tilde H^{\nn+5}$ and $n\in \mN$,
				we define
				\begin{equation*}
					T_{n}^{u_0} := \Pi_{i=1}^n   \chi \Big(\int_0^i\|u_r^{u_0}\|_{L^\mmm}^\mmm\dif r-\Upsilon i\Big).
				\end{equation*}
				{   Define $ B_{  H^{\nn+5}}(\mathfrak {R}):=\{u\in \tilde H^\nn:  \|u\|_{\nn+5}< \mathfrak {R}\}$ for any $\mathfrak {R}>0.$}
				By (\ref{3030-1}),
				there exists a $\mathcal E_\mmm>0$ such that
				\begin{eqnarray}
					\label{p0202-1}
					\begin{split}
						& \mP\Big( \|u_t\|^\mmm_{L^\mmm}+ \int_0^t  \|u_r\|_{L^\mmm}^\mmm \dif r  - \mathcal E_\mmm  t \geq K
						\Big)
						\\ & \leq \frac{C_\mmm t^{49}  (t+\|u_0\|_{L^{100\mmm}}^{100\mmm})   }{(K+\mathcal E_\mmm  t)^{100} }
						\\ & \leq \frac{C_{\mathcal R}  t^{50}  }{(K+\mathcal E_\mmm  t)^{100} },\quad \forall t\geq 1,K\geq 1\text{ and  }u_0\in B_{ H^{\nn+5}}(\mathfrak {R}),
					\end{split}
				\end{eqnarray}
				where $C_{\mathcal R}$ is a constant depending on $\mathcal R,m,\nu,d,\Bbbk, \{b_k\}_{k\in \cZ_0},\mathbb{U},(c_{\mathbbm{i},\mathbbm{j}})_{1\leq \mathbbm{i}\leq d,0\leq \mathbbm{j} \leq \Bbbk}$ and it eventually depends  on
				$\mathcal R,\nu,d,\Bbbk, \{b_k\}_{k\in \cZ_0},\mathbb{U},(c_{\mathbbm{i},\mathbbm{j}})_{1\leq \mathbbm{i}\leq d,0\leq \mathbbm{j} \leq \Bbbk}.$

				For any  $u_0,u_0'\in B_{ H^{\nn+5}}(\mathfrak {R})$ and $f\in C_b^1(\tilde H)$,
				observe that
				\begin{eqnarray*}
					\mE  f(u_n^{u_0})- \mE f\big(u_n^{u_0^\prime}\big)=I_1+I_2+I_3,
				\end{eqnarray*}
				where  $(u_t^{u_0})_{ t\geq 0}$ is   the  solution of \eqref{1-1} with initial value $u_0$ and
				\begin{align*}
					I_1 :=&  \mE \big[ f(u_n^{u_0})-f(u_n^{u_0})T_{n}^{u_0}\big],
					\\
					I_2:=& \mE \big[ f(u_n^{u_0}) T_n^{u_0}\big] -\mE \big[ f\big(u_n^{u_0^\prime}\big) T_n^{u_0'}\big],
					\\
					I_3 :=&  \mE \big[ f\big(u_n^{u_0^\prime}\big)T_n^{u_0'}-f\big(u_n^{u_0^\prime}\big)\big].
				\end{align*}
				
				For any $\aleph \geq 3 $ and  $\Upsilon=(\aleph+1)\mathcal E_\mmm+5$,   by (\ref{p0202-1}), it holds that
				\bae\label{smallness error local vs global}
				I_1\leq & \|f\|_{L^\infty} \sum_{i=1}^n\mP\Big( \int_0^i\|u_r^{u_0}\|_{L^\mmm}^\mmm\dif r-\Upsilon i\geq -4 \Big)
				\\ \leq  &  \|f\|_{L^\infty} \sum_{i=1}^n\mP\Big( \int_0^i\|u_r^{u_0}\|_{L^\mmm}^\mmm\dif r-\mathcal{E}_\mmm  i\geq \Upsilon i-\mathcal{E}_\mmm  i-4 \Big)
				\\ \leq &   \|f\|_{L^\infty} \sum_{i=1}^n
				\frac{C_{\mathcal R}  i^{50}  }{(\Upsilon i-\mathcal E_\mmm-4+\mathcal E_\mmm  i)^{100} }
				\\ \leq  &
				\|f\|_{L^\infty} \sum_{i=1}^n
				\frac{C_{\mathcal R}  i^{50}  }{(\aleph \mathcal E_\mmm i +\mathcal E_\mmm  i)^{100} }
				= \|f\|_{L^\infty} \frac{C_{\mathcal  R}}{( \aleph  \mathcal E_\mmm  +\mathcal E_\mmm)^{100} }\sum_{i=1}^n
				\frac{1 }{ i^{50}}.
				\eae
				For any $\eps>0$,	take  $\aleph=\aleph(\eps)\geq 3$  be  big enough
				and  let  $\Upsilon=(\aleph+1)\mathcal E_\mmm+5$. Then,  we  arrive at
				\bae\label{I1 grad}
				I_1\leq \frac{\eps}{3}.
				\eae
				With similar arguments,  we also have
				\bae\label{I3 grad}
				I_3\leq \frac{\eps}{3}.
				\eae
				The main difficulty lies in the    estimate of   $I_2.$

				In order to estimate $I_2$,
				we need to  give a  gradient estimate of $K_nf(u_0)$, where
				\begin{equation*}
					K_nf(u_0):=\mE \big[ f(u_n^{u_0}) T_n^{u_0}\big]=\mE \big[ f(u_n^{u_0}) \Pi_{i=1}^n \chi\big(\int_0^i\|u_r^{u_0}\|_{L^\mmm}^\mmm\dif r-\Upsilon  i \big)\big].\end{equation*}
				For any $\xi\in \tilde  H$, observe that
				\bae\label{Kn}
				D_{\xi}K_nf(u_0)
				=J_1 + J_2,
				\eae
				where for $1\leq k\leq n$
				\begin{align*}
					J_1 =&  \mE \left[ (D f)(u_n^{u_0})J_{0,n}\xi \cdot T_n^{u_0}\right],
					\\  J_2 =&  \sum_{k=1}^n  \mE \Big[ f(u_n^{u_0})
					\Pi_{i=1,i\neq k }^n \chi \Big(\int_0^i\|u_r^{u_0}\|_{L^\mmm}^\mmm\dif r-\Upsilon i\Big)
					\\ & \quad  \times \chi'\Big(\int_0^k\|u_r^{u_0}\|_{L^\mmm}^\mmm\dif r-\Upsilon k\Big)\mmm \int_0^k \left\langle (u^{u_0}_r)^{\mmm-1},  J_{0,r}\xi  \right\rangle \dif r \Big]\\
					=& \mmm\sum_{k=1}^n  \mE \Big( f(u_n^{u_0})  T^{u_0}_{n,k}
					\int_0^k \big\langle (u^{u_0}_r)^{\mmm-1},  J_{0,r}\xi  \big\rangle \dif r \Big).
				\end{align*}
				{ In the above,}
				\begin{eqnarray*}
					T^{u_0}_{n,k}:=\Pi_{i=1,i\neq k }^n \chi \Big(\int_0^i\|u_r^{u_0}\|_{L^\mmm}^\mmm\dif r-\Upsilon i\Big)
					\cdot  \chi'\Big(\int_0^k\|u_r^{u_0}\|_{L^\mmm}^\mmm\dif r-\Upsilon k\Big).
				\end{eqnarray*}
				Similar to   the papers~\cite{HM-2006, HM-2011},
				we  approximate the perturbation
				$J_{0,t}\xi$ caused by the variation of the initial condition with a variation, $\cA_{0,t}v=\cD^v w_t $, of the noise by an appropriate   process $v$. Denote by $\rho_t$ the  residual error between~$J_{0,t}\xi$    and~$\cA_{0,t}v$:
				\begin{align*}
					\rho_t=J_{0,t}\xi-\cA_{0,t}v.
				\end{align*}
				Then, it holds that
				\bae\label{div J1}
				J_1 =&
				\mE \left[ (D f)(u_n^{u_0}) A_{0, n}v \cdot  T_n^{u_0}\right] + \mE \left[ (D f)(u_n^{u_0}) \rho_{n}\cdot   T_n^{u_0}\right]
				\\ :=& J_{11} + J_{12},
				\eae
				and
				\bae\label{div J2}
				J_2 =&
				\mmm\sum_{k=1}^n  \mE \Big( f(u_n^{u_0})  T^{u_0}_{n,k}
				\int_0^k \big\langle ( u^{u_0}_r)^{\mmm-1},  A_{0,r}v  \big\rangle \dif r \Big)\\
				&+\mmm\sum_{k=1}^n  \mE \Big( f(u_n^{u_0})  T^{u_0}_{n,k}
				\int_0^k \big\langle ( u^{u_0}_r)^{\mmm-1},  \rho_r  \big\rangle \dif r \Big)\\
				:=& J_{21} + J_{22}.
				\eae
				From the integration by part formula  in the Malliavin calculus,  we have
				\begin{equation*}
					J_{11} + J_{21} = \mathbb E\big[\mathcal{D}^{v}\big(f(u_{n}^{u_{0}}) T_n^{u_0}\big)\big] = \mathbb{E}\bigg[f(u_{n}^{u_{0}})T_n^{u_0}\int_{0}^{n}v(s) dW(s)\bigg].
				\end{equation*}
				In the above,  the integral $\int_0^{n} v(s)\dif W(s)$ is interpreted as the Skohorod integral.

				In summary, the estimate of $I_2$ is derived through gradient estimates of $D_\xi  K_nf(u_0).$
				To achieve this, it is necessary to select an appropriate direction $v$ and establish moment estimates for  $\rho_t$   and the non-adapted integral   $\int_0^{t} v(s)\dif W(s).$
				These tasks are addressed in Subsections \ref{p1216-5}--\ref{skorohod}. Finally, in Subsection \ref{sec 3.1}, we complete the proof of Proposition \ref{3-11}.

				\subsection{The choice of $v$.}
				\label{p1216-5}
				
				In this section, we always assume that $\|\xi\|=1$ and
				$u_t$ is the solution of (\ref{1-1}) with initial value
				$u_0\in B_{ H^{\nn+5}}(\mathfrak {R})$.
				We work with the perturbation $v$ which is given by $0$ on all intervals of the type $({n+1},{n+2}), n\in 2\mN$, and by $v_{{n},{{n+1}}} \in L^2([n,{n+1}],H), n\in 2\mN$
				on the remaining intervals.
				{  For any $n\in 2\mN$, the infinitesimal variations $v_{n,n+1}$  is defined  by}
				\begin{eqnarray}
					\label{p-1}
					\begin{split}
						& v_{{n},{{n+1}}}(r):=\cA^*_{n,{n+1}}
						{(\cM_{n,{n+1}}+\beta\mI)}^{-1}J_{n,{n+1}}\rho_{{n}} ,
						~ r\in [{n},{{n+1}}],
						\\
						& v_{{{n+1}},{{n+2}}  }(r):=0,~ r\in ({{n+1}},{{n+2}}).
					\end{split}
				\end{eqnarray}
				where $\rho_{n} $ is the residual of the infinitesimal
				displacement at time $n$ which has not yet been compensated
				by $v$, i.e., $\rho_{n} =J_{0,n }\xi-\cA_{0,n}v_{0,{n}}.$
				Here and after, we use  $v_{a,b}$  to denote the function $v$ when
				we restrict its  domain to be $[a,b]$  and
				the constant $\beta$  in  (\ref{p-1})
				will be decided later.
				Obviously,    $\rho_0=J_{0,0}\xi-\cA_{0,0}v=\xi.$
				
				Similar to \cite{HM-2006}, we have the following  recursions  for $\rho_{n}$:
				\begin{eqnarray*}
					\rho_{{n+2} }&=&{J}_{{n+1},{n+2}}\beta
					(\cM_{n,{n+1}}+\beta \mI)^{-1}{J}_{n,{n+1}}
					\rho_{{n}}, \quad \forall n\in 2\mN.
				\end{eqnarray*}

				\subsection{The control of $\rho_n$}
				\label{p1216-4}
				

				For each $n\in \mN,u_0\in \tilde H^{\nn+5}$ and $\Upsilon>0$,  the random variable $\tilde T_{n}^{u_0}=\tilde T_{n}^{u_0}(\Upsilon)$
				is defined  by
				\begin{eqnarray}
					\label{p1110-2}
					\tilde T_{n}^{u_0}=\Pi_{i=1}^n   \chi(\int_0^i\|u_r^{u_0}\|_{L^\mmm}^\mmm \dif r-\Upsilon  i-2 ).
				\end{eqnarray}
				The aim of this subsection is to prove the following proposition.
				\begin{proposition}
					\label{4-2}
					For any $\gamma_0,\mathfrak {R}$ and $\Upsilon>0$ ,  there exists
					a constant $\beta=\beta(\gamma_0,\mathfrak {R}, \Upsilon)>0$ depending on $\gamma_0,\mathfrak {R}, \Upsilon $ and   $\nu,d,\Bbbk,(b_j)_{j\in \cZ_0},\mathbb{U},(c_{\mathbbm{i},\mathbbm{j}})_{1\leq \mathbbm{i}\leq d,0\leq \mathbbm{j} \leq \Bbbk}$
					such that
					if we define the direction $v$  according to  \eqref{p-1},
					then  
					the followings
					\begin{eqnarray}
						\label{15-1}
						&& \mE_{u_0} \left[  \|\rho_{t}\|^{80} \tilde T_{2n+1}^{u_0} \right] \leq C_{\gamma_0,\mathfrak {R}, \Upsilon}  \exp\{-\gamma_0 n\},\quad t\in [2n,2n+1],
						\\
						\label{15-2}
						&&   \mE_{u_0} \left[    \|\rho_{t}\|^{80}\tilde T_{2n+2}^{u_0} \right]
						\leq C_{\gamma_0,\mathfrak {R}, \Upsilon} \exp\{-\gamma_0 n\}, \quad t\in [2n+1,2n+2],
						\\ \label{15-3} &&  {  \mE_{u_0} \|\cD_s^i \rho_{2n}\|^{40}  \tilde T_{2n}^{u_0} \leq C_{\gamma_0,\mathfrak {R}, \Upsilon} \exp\{-\gamma_0 n\}, \quad  s\in (2l,2l+1),
						}
					\end{eqnarray}
					hold for any $1\leq i\leq \mathbb{U}$, $u_0\in \tilde H^{\nn+5}$ with $\|u_0\|_{\nn+5}  \leq \mathfrak {R}$ and  $n,l\in \mN$ with $l \leq n-1$,
					where $C_{\gamma_0,\mathfrak {R}, \Upsilon}$ is a constant depending on $\gamma_0,\mathfrak {R}, \Upsilon, \nu,d,\Bbbk,(b_j)_{j\in \cZ_0},\mathbb{U},(c_{\mathbbm{i},\mathbbm{j}})_{1\leq \mathbbm{i}\leq d,0\leq \mathbbm{j} \leq \Bbbk}$.
				\end{proposition}
				In order to  prove this proposition, we need to truncate $\rho_t$ into the low/high frequency parts. For any $n\in 2\mN, \beta>0$ and $N\in \mN$, let $\cR_{n,{n+1}}^\beta=\beta(\cM_{n,{n+1}}+\beta\mI)^{-1}$,  one easily sees that
				$
				\|\cR_{n,{n+1}}^\beta\|\leq1$.Then
				\begin{eqnarray}
					\nonumber  && \rho_{{n+2}}
					= J_{{n+1},{n+2}}
					\cR_{n,{n+1}}^\beta J_{n,{n+1}}\rho_{n}
					\\  \nonumber && = J_{{n+1},{n+2}}Q_N\cR_{n,{n+1}}^\beta J_{n,{n+1}}\rho_{n}
					+J_{{n+1},{n+2}}P_N\cR_{n,{n+1}}^\beta J_{n,{n+1}}\rho_{n}
					\\  \label{4-1} && :=  \rho_{{n+2}}^{(1)}+ \rho_{{n+2}}^{(2)}.
				\end{eqnarray}
				Therefore, by Lemma \ref{15-4},   one has
				\begin{eqnarray}
					\label{p1103-1}
					\begin{split}
						& \|\rho_{n+2}\|^{80}
						\\ & \leq  C e^{C \int_n^{n+2}\|u_r\|_{L^\mmm}^\mmm\dif r}(\|J_{{n+1},{n+2}}Q_N\|^{80}+\|P_N\cR_{n,{n+1}}^\beta\|^{80})
						\|\rho_{n}\|^{80}.
					\end{split}
				\end{eqnarray}
				Furthermore,  with the help of   (\ref{16-2}),
				we also have
				\begin{eqnarray}
					\label{p1104-1}
					\|\rho_{n+2}\|^{80}\leq  C e^{C \int_n^{n+2}\|u_r\|_{L^\mmm}^\mmm\dif r}\Big(\frac{1}{N^{20}}+\|P_N\cR_{n,{n+1}}^\beta\|^{80}\Big)\|\rho_{n}\|^{80},
				\end{eqnarray}

				Denote
				\begin{eqnarray}
					\label{p05-3}
					\zeta_n:=\frac{1}{N^{20}} +\|P_N\cR_{n,{n+1}}^\beta\|^{80}.
				\end{eqnarray}
				Then, by (\ref{p1104-1}) {  and $\|\rho_0\|=\|\xi\|=1$}, we have
				\begin{equation}\label{error-theta}
					\|\rho_{2n+2}\|^{80}\leq C^n  e^{C\int_0^{2 n+2}\|u_r\|_{L^\mmm}^\mmm \dif r}
					\prod_{i=0}^n\zeta_{2i},
				\end{equation}
				{	where the constant  $C$    depends on $\nu,d,\Bbbk, (b_{i})_{i\in \cZ_0},\mathbb{U}$
					and   $
					(c_{\mathbbm{i},\mathbbm{j}})_{1\leq \mathbbm{i}\leq d,0\leq \mathbbm{j} \leq \Bbbk}.$}
				
				Let  \begin{equation*}A_\eps=A_{\eps,u_0,\alpha,N}=\{X^{u_0,\alpha,N}\geq \eps\},\end{equation*}
				where  the random variable $X^{w_0,\alpha,N}$ is defined in (\ref{1-5}).
				First, we demonstrate three  Lemmas. Then, we give an  estimate of
				$\|\rho_{{n+2}}\|.$
				\begin{lemma}(c.f. \cite[Lemma 5.14]{HM-2011})
					\label{3-3}
					For any positive constants  $\beta,\eps,\alpha\in(0,1],N\in \mN$
					and $\xi\in \tilde H, u_0\in \tilde  H^{\nn+5} $,  the following inequality holds with probability $1$:
					\begin{eqnarray}
						\label{3-1}
						\beta \|P_N(\beta+\cM_{0,1})^{-1}\xi\|
						\leq   \| \xi\|
						\big(\alpha \vee \sqrt{{\beta}/{\eps}}\big) I_{A_\eps} +\|\xi\|I_{A_\eps^c }.
					\end{eqnarray}
				\end{lemma}

				\begin{proof}
					On the event  $A_\eps^c$, the inequality (\ref{3-1})  obviously  holds.
					On the event $A_\eps$, this inequality is proved in \cite[Lemma 5.14]{HM-2011}, so we omit the details.
					
				\end{proof}

				{  	
					Now, we give the following   lemma on $\cR_{k,{k+1}}^\beta.$
					\begin{lemma}
						\label{3-9}
						For any $\kappa, \delta\in (0,1],p\geq 1$ and $N\in \mN$, there exists a $\beta=\beta(\kappa, \delta,p,N)>0$\footnote{ $ \beta(\mathfrak{R}, \delta,p,N) $  denotes   a constant that may    depend  on $\mathfrak{R}, \delta,p,N$  and
							$\nu,d,\Bbbk,(b_j)_{j\in \cZ_0},\mathbb{U}$,$(c_{\mathbbm{i},\mathbbm{j}})_{1\leq \mathbbm{i} \leq d,0\leq \mathbbm{j} \leq \Bbbk}$.
						}
						such  that the following holds for all   $k\in \mN$:
						\begin{eqnarray}
							\label{pp17-1}
							\mE\left[\|P_N \cR_{{k},{k+1}}^\beta\|^p  \big| \cF_k  \right]
							\leq \delta \exp\{\kappa \|u_{k}\|_{\nn+5}^2\}.
						\end{eqnarray}
					\end{lemma}
					\begin{proof}
						We here give a proof for the case $k=0$ and $p\geq 2.$ The other cases can be proved similarly.
						Let  $\mathfrak{R}=\mathfrak{R}_{\delta,\kappa}$ be a positive  constant  such that
						$
						\exp\{\kappa \mathfrak{R}^2  \}\geq \frac{1}{\delta}.
						$
						We divide into the following two cases to prove (\ref{pp17-1}).

						\textbf{Case 1:} $\|u_0\|_{\nn+5}\geq \mathfrak{R}$.
						In this case,
						\begin{eqnarray*}
							\mE\big[\|P_N \cR_{0,1}^{\beta} \|^p \big| \cF_0 \big]\leq
							1\leq \delta e^{\kappa \|u_{0}\|_{\nn+5}^2}.
						\end{eqnarray*}

						\textbf{Case 2:}
						$\|u_0\|_{\nn+5} \leq \mathfrak{R}$.
						For any  positive constants $\eps,\beta$ and $\alpha\in (0,1],$ by Lemma \ref{3-3}, we have
						\begin{eqnarray*}
							\mE\big[\|P_N \cR_{0,1}^\beta \|^p \big| \cF_0 \big] & \leq&
							C_p \big(\alpha \vee \sqrt{\frac{\beta}{ \eps }}\big)^p+
							C_p r(\eps,\alpha, \mathfrak{R},N),
						\end{eqnarray*}
						where $C_p$ is a constant only depending on $p,$
						and  $r(\eps,\alpha, \mathfrak{R},N)$ is defined in (\ref{p25-2}).
						Choose now  $\alpha=\alpha({p})$ small enough such that
						\begin{eqnarray*}
							C_p \alpha^p \leq \frac{\delta}{2}.
						\end{eqnarray*}
						By Proposition \ref{3-8}, $\lim_{\eps \rightarrow 0}  r(\eps,\alpha, \mathfrak{R},N)=0$.
						Pick a small constant  $\eps$ such that
						\begin{eqnarray*}
							C_p r(\eps,\alpha, \mathfrak{R},N)\leq \frac{\delta}{2}.
						\end{eqnarray*}
						Finally, we choose  $\beta$ small enough so that
						\begin{align*}
							C_p( \sqrt{\beta/\eps})^p <\frac{\delta}{2}.
						\end{align*}
						Putting the above steps together, we see that
						$
						\mE\big[\|P_N \cR_{0,1}^{\beta} \|^p \big|\cF_0  \big]  \leq \delta
						e^{\kappa \|w_{0}\|^2} .
						$

					\end{proof}
				}

				\textbf{Now we are in a position to prove Proposition \ref{4-2}. }
				\begin{proof}
					In order to prove Proposition \ref{4-2}, we only need to  prove that   for any $\delta>0$, there exists    a  ${ \beta=\beta(\delta,\mathfrak {R}, \Upsilon)}>0$
					such  that
					if we define the direction $v$  according to  (\ref{p-1}),
					then the followings:
					\begin{eqnarray}
						\label{pp-1} && \mE_{u_0} \Big[ \|\rho_{t}\|^{80}\tilde T_{2n+1}^{u_0} \Big] \leq C_{\beta} \bar{C}^n   \delta^{(n-1)/2},  \quad t\in [2n,2n+1],
						\\  \label{pp-2} &&   \mE_{u_0} \Big[ \|\rho_{t}\|^{80}\tilde T_{2n+2}^{u_0} \Big] \leq C_{\beta}  \hat{C}^n \delta^{(n-1)/2} , \quad t\in [2n+1,2n+2],
						\\   \label{pp-3} &&
						{   \mE_{u_0}  \Big[\| \cD_s^i \rho_{{2n} }\|^{40}   \tilde T_{2n}^{u_0} \Big]
							\leq    C_{\beta}  \tilde  C^{n} \delta^{(n-1)/2}, \quad  s\in (2l,2l+1).
						}
					\end{eqnarray}
					hold for any $n,l\in \mN$
					with $l \leq n-1$ and $u_0\in \tilde  H^{\nn+5}$ with $\|u_0\|_{\nn+5} \leq \mathfrak {R}$.
					In the above,   $\bar{C},\hat C,\tilde C$
					are  positive   constants  depending on $\mathfrak {R}, \Upsilon, \nu,d,\Bbbk,(b_j)_{j\in \cZ_0},\mathbb{U},(c_{\mathbbm{i},\mathbbm{j}})_{1\leq \mathbbm{i}\leq d,0\leq \mathbbm{j} \leq \Bbbk}$. We emphasize that $\bar{C},\hat C,\tilde C$
					are independent of $\delta$. $ C_{\beta}$ is a  constant  depending on $\beta$,$\nu,d$,$\Bbbk$,$(b_j)_{j\in \cZ_0},\mathbb{U}$, $(c_{\mathbbm{i},\mathbbm{j}})_{1\leq \mathbbm{i} \leq d,0\leq \mathbbm{j} \leq \Bbbk}$
					which     eventually depends on  $\delta,\mathfrak {R}, \Upsilon,\nu,d,\Bbbk,(b_j)_{j\in \cZ_0}$,
					$\mathbb{U},(c_{\mathbbm{i},\mathbbm{j}})_{1\leq \mathbbm{i}\leq d,0\leq \mathbbm{j} \leq \Bbbk}.$

					For any   $\beta>0$ and $N\in \mN$,   by (\ref{4-1}), we have
					\begin{eqnarray}
						\nonumber  && \rho_{{n+2}}
						= J_{{n+1},{n+2}}
						\cR_{n,{n+1}}^\beta J_{n,{n+1}}\rho_{n}
						\\  \nonumber && = J_{{n+1},{n+2}}Q_N\cR_{n,{n+1}}^\beta J_{n,{n+1}}\rho_{n}
						+J_{{n+1},{n+2}}P_N\cR_{n,{n+1}}^\beta J_{n,{n+1}}\rho_{n}.
					\end{eqnarray}
					Recall that $\zeta_{n}=\frac{1}{N^{20}}+\|P_N \cR_{{n},{n+1}}^\beta\|^{80}.$
					{  	By Lemma \ref{3-9},
						for any $\kappa,\delta>0$, there exist a $N=N(\kappa,\delta)\in \mN$ and a  $\beta=\beta(\kappa,\delta)>0$
						such  that the following holds for all $n\geq 1$:
						\begin{eqnarray}
							\label{p3-4}
							\mE\Big[\zeta_{2n}
							\Big|\cF_{2n} \Big]
							\leq \delta \exp\{\kappa \|u_{2n}\|_{\nn+5}^2\}.
						\end{eqnarray}
					}
					As in the other space of this paper,
					the letters $C,C_1,C_2,\cdots$ are  always  used to denote unessential constants that  may
					change from line to line  and  implicitly  depend on the data of equation (\ref{1-1}), i.e.,
					$\nu,d,\Bbbk, \{b_i\}_{i\in \cZ_0},\mathbb{U},(c_{\mathbbm{i},\mathbbm{j}})_{1\leq \mathbbm{i}\leq d,0\leq \mathbbm{j} \leq \Bbbk}.$
					The following properties will be used frequently in the proof of Proposition \ref{4-2}:
					\begin{itemize}
						\item[(a)]on the event $\{\omega: \tilde  T_{n}^{u_0}(\omega)\neq 0\}$,
						it holds that  $\int_0^i\|u_r^{u_0}\|_{L^\mmm}^\mmm \dif r\leq \Upsilon  i, ~\forall 1\leq i\leq n;$
						\item[(b)] for any $k\leq n, $ $0\leq \tilde T_n^{u_0}\leq \tilde T_k^{u_0}\leq 1; $
						\item[(c)]
						From the variation of constants formula, and in view of (\ref{0927-1}),
						we get{
							\begin{eqnarray*}
								\cD_s^iJ_{r,t}\xi=
								\left\{
								\begin{split}
									& J_{s,t}^{(2)}(Q\theta_i,J_{r,s}\xi),\quad \text{if }s\geq r,
									\\ & J_{r,t}^{(2)}(J_{s,r}Q\theta_i,\xi),\quad \text{if }s\leq r.
								\end{split}
								\right.
						\end{eqnarray*}}
						(Recall that   $\{\theta_i\}_{i=1}^\mathbb{U}$ is  the standard   basis of $\mR^{\mathbb{U}}$ and $Q:\mR^\mathbb{U}\rightarrow H$ is an linear operator defined in subsection \ref{pp24}). 
					\end{itemize}

					Now we give a proof of  (\ref{pp-1}) and (\ref{pp-3})
					respectively. The proof of (\ref{pp-2}) is similar to  (\ref{pp-1}) and we omit the details.
					
					{  	
						\textbf{Proof of  (\ref{pp-1}) }
						By  (\ref{error-theta}),
						it holds that
						\begin{equation*}
							\|\rho_{2n}\|^{80}\leq C^n  e^{C\int_0^{2 n}\|u_r\|_{L^\mm}^\mm \dif r}
							\prod_{i=0}^{n-1}\zeta_{2i}.
						\end{equation*}

						Therefore,  for any real numbers    $\kappa,  K>0,$  it holds that
						\begin{eqnarray*}
							\nonumber
							&&  \mE\Big[ \|\rho_{2n}\|^{80} \tilde T_{2n+1}^{u_0} \Big]
							\leq  \mE\Big[C^n  e^{C\int_0^{2 n}\|u_r\|_{L^\mm}^\mm \dif r}
							\big(\prod_{i=0}^{n-1}\zeta_{2i}\big) \tilde T_{2n+1}^{u_0}   \Big]
							\\  \label{p06-7} &&\leq   C^n\exp\big\{2n \Upsilon (C+K)+2nK  \big\}
							\mE\Big[ e^{\sum_{i=0}^{2n-1}\frac{\kappa}{2}  \|u_i\|_{\nn+5}^2- K\int_0^{2 n}\|u_r\|_{L^\mm}^\mm \dif r -2nK}
							\\ && \quad\quad\quad\quad\quad\quad\quad\quad\quad\quad
							\quad\quad  \quad\quad \times 				\big(	\prod_{i=0}^{n-1}\zeta_{2i}\big) \tilde T_{2n+1}^{u_0}\cdot \exp\{-\sum_{i=0}^{2n-1}\frac{\kappa}{2} \|u_i\|_{\nn+5}^2 \}  \Big]
							\\ && \leq  C^n\exp\big\{2n \Upsilon (C+K) +2nK \big\}
							\Big( \mE \exp\big\{\sum_{i=0}^{2n-1}\kappa   \|u_i\|_{\nn+5}^2- 2K\int_0^{2 n}\|u_r\|_{L^\mm}^\mm \dif r-4nK \big\}\Big)^{1/2}
							\\ && \quad\quad\quad\quad\quad\quad\quad\quad\quad\quad  \times
							\Big(\mE\Big[ \prod_{i=0}^{n-1}\zeta_{2i}^2  \cdot \exp\{-\sum_{i=0}^{2n-1}\kappa  \|u_i\|_{\nn+5}^2 \}  \Big] \Big)^{1/2}.
						\end{eqnarray*}
						We set a  $\kappa\in (0, \kappa_0]$ and let  $K=K_\kappa$, where  $\kappa_0,K_\kappa$ are decided by  Lemma \ref{p17-4},    then  we  arrive at
						\begin{eqnarray}
							\label{1531-1}
							\begin{split}
								&\mE\Big[ \|\rho_{2n}\|^{80} \tilde T_{2n+1}^{u_0} \Big]
								\leq   C_{\kappa,\mathfrak {R}} C^n\exp\big\{2n \Upsilon (C+K) +2nK \big\}
								\\ &\quad\quad\quad\quad\quad   \quad \times
								\Big(\mE\Big[ \prod_{i=0}^{n-1}\zeta_{2i}^2   \cdot \exp\{-\sum_{i=0}^{2n-1}\kappa  \|u_i\|_{\nn+5}^2 \}  \Big] \Big)^{1/2}.
							\end{split}
						\end{eqnarray}
						With regard to the term
						$\mE\Big[ \prod_{i=0}^{n-1}\zeta_{2i}^2  \cdot \exp\{-\sum_{i=0}^{2n-1}\kappa  \|u_i\|_{\nn+5}^2 \}  \Big]$
						appeared in the rightside of the above,
						by (\ref{p3-4}), we conclude that
						\begin{eqnarray*}
							&& \mE\Big[ \prod_{i=0}^{n-1}\zeta_{2i}^2  \cdot \exp\{-\sum_{i=0}^{2n-1}\kappa  \|u_i\|_{\nn+5}^2 \}  \Big]
							\\  &&  \leq  \mE \Big[\prod_{i=0}^{n-1}\zeta_{2i}^2  \exp\{-\sum_{i=0}^{2n-2}\kappa  \|u_i\|_{\nn+5}^2 \}  \Big]
							\\ && 	=
							\mE \Big[\prod_{i=0}^{n-1}\zeta_{2i}^2  \exp\{-\sum_{i=0}^{2n-2}\kappa  \|u_i\|_{\nn+5}^2 \} \big| \cF_{2n-2} \Big]
							\\ &&\leq \delta
							\mE \Big[\prod_{i=0}^{n-2}\zeta_{2i}^2  \exp\{-\sum_{i=0}^{2n-4}\kappa  \|u_i\|_{\nn+5}^2 \} \Big].
						\end{eqnarray*}
						By iterations,   we arrive at
						\begin{eqnarray*}
							&& 	\mE\Big[ \prod_{i=0}^{n-1}\zeta_{2i}^2  \cdot \exp\{-\sum_{i=0}^{2n-1}\kappa  \|u_i\|_{\nn+5}^2 \}  \Big]
							\\ && 	\leq \delta^{n-1}  \mE \Big[ \zeta_{0}^2  \exp\{-\kappa  \|u_0\|_{\nn+5}^2 \}  \Big]
							\\ && \leq  4 \delta^{n-1}.
						\end{eqnarray*}
						In the above, we have used the fact $\zeta_0\leq 2$ in the third inequality.
						Combining the above inequality with (\ref{1531-1}), we get
						\begin{eqnarray}
							\label{p1104-2}
							\mE\Big[ \|\rho_{2n}\|^{80} \tilde T_{2n+1}^{u_0}  \Big] \leq \mathscr{C}^n \delta^{(n-1)/2},\quad \forall n\geq 1,
						\end{eqnarray}
						where $ \mathscr{C}\geq 1$ is a constant  depending on $\kappa, \Upsilon$ and
						$\nu,d,\Bbbk, \{b_i\}_{i\in \cZ_0},\mathbb{U},(c_{\mathbbm{i},\mathbbm{j}})_{1\leq \mathbbm{i}\leq d,0\leq \mathbbm{j} \leq \Bbbk}.$

						{  From the construction we have}
						\begin{equation*}
							\rho_t=\begin{cases} J_{2n,t}\rho_{2n}-
								\cA_{2n,t}v_{{2n},t}, &
								\text{for }t \in [{2n},{2n+1}],
								\\ J_{2n+1,t}\rho_{2n+1}, &
								\text{for }t\in ({2n+1},{{2n+2}}) \end{cases}	
						\end{equation*} for any  $n\in \mathbb N \cup \{0\}$ and  $t\in [2n,2n+1]$.
						
						Define $\widetilde\cM_{2k,{2k+1}}:=\beta\mI+\cM_{2k,{2k+1}}.$
						Using \eqref{p-1}, Lemma \ref{15-4} and Lemma  \ref{L:2.2}, we get
						\begin{eqnarray} \label{3.8}
							\begin{split}
								& \|v_{2n,2n+1}\|_{L^2([{2n},{2n+1}];
									\R^\mathbb{U})}
								\\ \le &
								\|\cA^*_{2n,{2n+1}}
								\widetilde\cM_{2n,{2n+1}}^{-\frac12}
								\widetilde\cM_{2n,{2n+1}}^{-\frac12}J_{2n,{2n+1}}\rho_{{2n}} \|_{L^2([{2n},{2n+1}];
									\R^\mathbb{U})}  \\
								\leq& \|\cA^*_{2n,{2n+1}}\widetilde\cM_{2n,{2n+1}}^{-\frac12}\| \|\widetilde\cM_{2n,{2n+1}}^{-\frac12}\| \| J_{2n,{2n+1}} \| \| \rho_{{2n}} \|\\
								\leq& \beta^{-\frac12} \| J_{2n,{2n+1}} \| \| \rho_{{2n}} \|
								\\ \leq &   C \beta^{-1/2}\exp\Big\{C \int_{2n}^{2n+1}\|u_r\|_{L^\mm}^\mm \dif r\Big\}\| \rho_{2n}\|.
							\end{split}
						\end{eqnarray}
						Hence,   by \eqref{2.7}, \eqref{3.8} and Lemma \ref{15-4},  for any $t\in [2n,2n+1]$,
						\begin{align}
							\|\rho_t\|&\le  \|J_{2n,t}\rho_{2n}\|+
							\|\aA_{2n,t}v_{2n,t}\|  \nonumber
							\\  \nonumber  &\le  \|J_{2n,t}\rho_{2n}\|+
							\|\aA_{{2n},t}\|_{\cL(L^2([{2n},t];\R^\mathbb U), H)}
							\|v_{{2n},t}\|_{L^2([{2n},{2n+1}];\R^\mathbb U)}
							\\  \nonumber  &\le  \|J_{2n,t}\rho_{2n}\|+
							\|\aA_{{2n},t}\|_{\cL(L^2([{2n},t];\R^\mathbb U), H)}
							\|v_{{2n},{2n+1}}\|_{L^2([{2n},
								{2n+1}];\R^\mathbb U)}
							\\ \nonumber
							&\le \|J_{2n,t}\rho_{2n}\| \\\nonumber
							&\quad+\sup_{s\in [{2n},t]}\|J_{s,t}
							\|_{\cL(H,H)}\cdot C \beta^{-1/2}
							\exp\Big\{ C \int_{2n}^{2n+1}\|u_s\|_{L^\mm}^{m}\dif s \Big\}
							\\ \label{21-1} &\leq C_\beta  \exp\Big\{ C \int_{2n}^{2n+1}\|u_s\|_{L^\mm}^{m}\dif s \Big\} \|\rho_{2n}\|.
						\end{align}
						Therefore, for $t\in [2n,2n+1]$,  by the definition of $\tilde T_{2n+1}^{u_0}$, we have
						\begin{eqnarray}
							\nonumber && \mE \Big[ \|\rho_{t}\|^{80}\tilde T_{2n+1}^{u_0}  \Big]
							\leq  C_\beta   \mE\Big[  \exp\Big\{ C \int_{2n}^{2n+1}\|u_s\|_{L^\mm}^{\mm}\dif s \Big\} \|\rho_{2n}\|^{80} \tilde T_{2n+1}^{u_0} \Big]
							\\ \nonumber &&\leq C_\beta \exp\big\{(2n+1)C\Upsilon \big\} \mE\Big[  \|\rho_{2n}\|^{80} \tilde T_{2n+1}^{u_0} \Big].
						\end{eqnarray}
						Combining the above with  (\ref{p1104-2}), it yields the desired result
						(\ref{pp-1}).
						
					}

					\textbf{Proof of (\ref{pp-3}).}
					For any  non-negative integers    $k,l \in \mN$ with   $k\geq l+1$ and $s\in [2l,2l+1]$, noticing that
					$\rho_{{2k+2} }={J}_{{2k+1},{2k+2}}\beta
					(\cM_{2k,{2k+1}}+\beta \mI)^{-1}{J}_{2k,{2k+1}}
					\rho_{{2k}}$,
					it holds that
					\begin{eqnarray}
						\nonumber  && \cD_s^i \rho_{{2k+2} }
						\\ \nonumber &&  =\beta   \cD_s^i J_{2k+1,{2k+2}}
						\widetilde\cM_{2k,{2k+1}}^{-1}J_{2k,{2k+1}}\rho_{{2k}}+ \beta J_{2k+1,{2k+2}}
						\widetilde\cM_{2k,{2k+1}}^{-1}\cD_s^i J_{2k,{2k+1}}\rho_{{2k}}\\
						&&\nonumber  \quad - \beta J_{2k+1,{2k+2}}
						\widetilde\cM_{2k,{2k+1}}^{-1}\Big[\left(\cD_s^i \cA_{2k,{2k+1}} \right)\cA^*_{2k,{2k+1}}
						\\ \nonumber && \quad\quad\quad\quad  \quad\quad\quad\quad \quad\quad \quad  +\cA_{2k,{2k+1}}\left(\cD_s^i \cA^*_{2k,{2k+1}} \right)\Big]\widetilde\cM_{2k,{2k+1}}^{-1}J_{2k,{2k+1}}\rho_{{2k}}
						\\ \label{p1103-2}  &&  \quad+\beta J_{2k+1,{2k+2}}
						\widetilde\cM_{2k,{2k+1}}^{-1}J_{2k,{2k+1}} \cD_s^i\rho_{2k}.
					\end{eqnarray}
					Observe that
					\begin{equation*}
						\mathcal{D}_s^i{J}_{2k,2k+1} \xi=J^{(2)}_{2k, 2k+1}\left(J_{s, 2k}Q\theta_i, \xi\right).
					\end{equation*}
					Therefore,  by Lemma \ref{L^p J} and Lemma \ref{2.18 heissan}, we have
					\bae\label{ppDj}
					& \|\cD_s^i J_{2k, 2k+1} \xi\|\leq C\left\|J_{s, 2k} Q\theta_i \right\|_{L^8} \|\xi\| e^{C\int_{2k}^{2k+1} \|u_t\|^\mmm_{L^\mmm}\dif r} \\
					& \leq C   \left\| \xi \right\| e^{C\int_s^{2k+1} \|u_t\|^\mmm_{L^\mmm}\dif t}.
					\eae
					Similarly, we also have
					\bae\label{ppDj2}
					& \|\cD_s^i J_{2k+1, 2k+2} \xi\| \leq C   \left\| \xi \right\| e^{C\int_s^{2k+2} \|u_t\|^\mmm_{L^\mmm}\dif t}.
					\eae
					By the chain rule of Malliavin derivative, also in view of  $s\in [2l,2l+1]$ {  and  $l\leq k-1$,}  we get
					\baee
					\mathcal{D}_s^i A_{2k,2k+1} h= & \mathcal{D}_s^i \int_{2k}^{2k+1}J_{r,{2k+1}}Qh(r)\dif r=
					\int_{2k}^{2k+1}J_{r,{2k+1}}^{(2)}(J_{s,r}Q\theta_i,Qh(r))\dif r
					.
					\eaee
					Then,   with the help of   Lemmas \ref{15-4}, \ref{L^p J}, \ref{2.18 heissan},   one arrives at
					\bae\label{ppDa}
					& \left\|\cD_s^i A_{2k, 2k+1} h\right\|
					\\
					\leq & \int_{2k}^{2k+1}\| J_{r,{2k+1}}^{(2)}(J_{s,r}Q\theta_i,Qh(r)\| \dif r
					\\
					\leq & C \int_{2k}^{2k+1}\|J_{s,r}Q\theta_i\|_{L^8} \| Qh(r)\| \dif r
					\exp\Big\{C \int_{s}^{2k+1}\|u_t\|_{L^\mmm}^\mmm \dif t\Big\}
					\\ \leq &  C \exp\Big\{C \int_{s}^{2k+1}\|u_t\|_{L^\mmm}^\mmm \dif t\Big\}
					\|h\|_{L^2([2k,2k+1],\mR^\mathbb{U})}.
					\eae
					In view of \eqref{4-1}, combining the estimates (\ref{ppDj})--(\ref{ppDa})  with (\ref{p1103-2}),  we get
					\begin{eqnarray*}
						&& \| \cD_s^i \rho_{{2k+2} }\|\leq C_\beta  \exp\Big\{C \int_{s}^{2k+1}\|u_t\|_{L^\mmm}^\mmm \dif t\Big\} \|\rho_{2k}\|
						\\ && \quad+C \exp\Big\{C \int_{2k}^{2k+2}\|u_t\|_{L^\mmm}^\mmm \dif t\Big\} \big(\|J_{2k+1,2k+2}Q_N\|+\|P_N \cR_{2k,2k+1}^\beta\|\big) \| \cD_s^i \rho_{2k }\|
						\\ && \leq C_\beta  \exp\Big\{C \int_{s}^{2k+1}\|u_t\|_{L^\mmm}^\mmm \dif t\Big\} \|\rho_{2k}\|
						\\ && \quad+C \exp\Big\{C \int_{2k}^{2k+2}\|u_t\|_{L^\mmm}^\mmm \dif t\Big\} \big(\frac{1}{N^{1/4}}+\|P_N \cR_{2k,2k+1}^\beta\|\big) \| \cD_s^i \rho_{2k }\|,
					\end{eqnarray*}
					where we have used Lemmas \ref{15-4},  \ref{L:2.2} in the first inequality,   and have used  (\ref{16-2}) in the second inequality.
					
					Therefore, there exists a $C \geq 1$ such that for any   $k\in \mN$ with $k\geq l+1$  and  $s\in [2l,2l+1]$, it holds that
					\begin{eqnarray}
						\label{p1103-3}
						\begin{split}
							& \| \cD_s^i \rho_{{2k+2} }\|^{80} \exp\Big\{-C\int_{0}^{2k+2}\|u_r\|_{L^\mmm}^\mmm \dif r-C (2k+2)\Big\}
							\\ & \leq  C_\beta \|\rho_{2k}\|^{80}   + \| \cD_s^i \rho_{{2k} }\|^{80}  \exp\Big\{
							-C\int_{0}^{2k}\|u_r\|_{L^\mmm}^\mmm \dif r-  2k  C \Big\} \zeta_{2k}.
						\end{split}
					\end{eqnarray}
					(Recall that   $\zeta_n=\frac{1}{N^{20}}+\|P_N \cR_{n,n+1}^\beta\|^{80},\forall n\in \mN.$)

					For $k=l$ and  $s\in (2l,2l+1)$,   by direct calculations, we   have
					\begin{eqnarray}
						\nonumber  && \cD_s^i \rho_{{2l+2} }
						\\ \nonumber &&  =\beta   \cD_s^i J_{2l+1,{2l+2}}
						\widetilde\cM_{2l,{2l+1}}^{-1}J_{2l,{2l+1}}\rho_{{2l}}+ \beta J_{2l+1,{2l+2}}
						\widetilde\cM_{2l,{2l+1}}^{-1}\cD_s^i J_{2l,{2l+1}}\rho_{{2l}}\\
						&&\nonumber  \quad - \beta J_{2l+1,{2l+2}}
						\widetilde\cM_{2l,{2l+1}}^{-1}\Big[\left(\cD_s^i \cA_{2l,{2l+1}} \right)\cA^*_{2l,{2l+1}}
						\\ \nonumber && \quad\quad\quad\quad  \quad\quad\quad\quad \quad\quad \quad  +\cA_{2l,{2l+1}}\left(\cD_s^i \cA^*_{2l,{2l+1}} \right)\Big]\widetilde\cM_{2l,{2l+1}}^{-1}J_{2l,{2l+1}}\rho_{{2l}}.
					\end{eqnarray}
					Since
					\begin{equation*}
						\mathcal{D}_s^i{J}_{2l,2l+1} \xi=J^{(2)}_{s, 2l+1}\left(J_{2l, s}\xi, Q \theta_i\right) \quad \text { for } s \in\left(2l, 2l+1\right),
					\end{equation*}
					by Lemma \ref{2.18 heissan}, we have
					\bae\label{pppDj}
					& \left\|\cD_s^i J_{2l, 2l+1} \xi \right\|=\big\|J_{s, 2l+1}^{(2)}\left(J_{2l, s} \xi, Q \theta_i\right)\big\| \\
					& \leq C\left\|Q \theta_i\right\|_{L^8} \left\|J_{2l, s} \xi\right\| e^{C\int_{2l}^{2l+1} \|u_r\|^\mmm_{L^\mmm}\dif r} \\
					& \leq C   \left\| \xi\right\| e^{C\int_{2l}^{2l+1} \|u_r\|^\mmm_{L^\mmm}\dif r}.
					\eae
					Similarly, for $h\in L^2([2l,2l+1],\mR^\mathbb U )$
					\baee
					& \mathcal{D}_s^i A_{2l,2l+1} h= \int_{2l}^s  \mathcal{D}_s^i J_{r, 2l+1}Qh(r)\dif r
					+\int_{s}^{2l+1}  \mathcal{D}_s^i J_{r, 2l+1}Qh(r)\dif r
					\\ &=\int_{2l}^s J^{(2)}_{s, 2l+1}\big(J_{r,s} Q h(r), Q \theta_i)\big)\dif r  +\int_s^{2l+1} J^{(2)}_{r, 2l+1}\big(J_{s,r}Q \theta_i, Qh(r)\big)\dif r.
					\eaee
					Then,  using  Lemma \ref{L^p J}, \ref{15-4} and \ref{2.18 heissan}, we get
					\bae\label{pppDa}
					& \left\|\cD_s^i A_{2l, 2l+1} h\right\| \\
					\leq & \int_{2l}^s \big\|J_{s, 2l+1}^{(2)}\left(J_{r, s} Q h(r), Q \theta_i\right)\big\|\dif s +\int_s^{2l+1}\|J_{r, 2l+1}^{(2)}\left(J_{s, r} Q \theta_i,Q h(r)\right)\|\dif s\\
					\leq & C \left\|Q \theta_i\right\|_{L^8} e^{C \int_{2l}^{2l+1} \|u_r\|^\mmm_{L^\mmm}\dif r} \int_{2l}^s \left\|J_{r, s} Q h(r)\right\|\dif r \\
					& +C e^{C \int_{2l}^{2l+1} \|u_r\|^\mmm_{L^\mmm} \dif r} \int_s^{2l+1}\left\|J_{s, r} Q \theta_i\right\|_{L^8}\|Q h (r)\|\dif r\\
					\leq& C e^{C \int_{2l}^{2l+1}\|u_r\|^\mmm_{L^\mmm} \dif s} \|h\|_{L^2{([2l,2l+1],\mR^\mathbb{U})}},
					\eae
					Noticing the above estimates (\ref{pppDj})(\ref{pppDa}), similar to (\ref{p1103-3}), we have,
					\begin{eqnarray}
						\label{p1103-4}
						\begin{split}
							& \| \cD_s^i \rho_{{2l+2} }\|^{80} \exp\Big\{-C \int_{0}^{2l+2}\|u_r\|_{L^\mmm}^\mmm \dif r-C(2l+2)\Big\}
							\leq C_\beta  \|\rho_{2l}\|^{80}.
						\end{split}
					\end{eqnarray}

					By iteration and the above inequalities  (\ref{p1103-3})(\ref{p1103-4}),
					for some $\cC\geq 1$
					that only depends  on $ \nu,d,\Bbbk,(b_j)_{j\in \cZ_0}$,
					$\mathbb{U},(c_{\mathbbm{i},\mathbbm{j}})_{1\leq \mathbbm{i}\leq d,0\leq \mathbbm{j} \leq \Bbbk}.$
					,  we arrive at
					\begin{eqnarray*}
						\begin{split}
							& \| \cD_s^i \rho_{{2n} }\|^{80} \exp\Big\{-\cC \int_{0}^{2n}\|u_r\|_{L^\mmm}^\mmm \dif r-2n\cC\Big\}
							\\  &\leq C_\beta \sum_{k=1}^{n-l} \big(\prod_{j=1}^{k-1}\zeta_{2n-2j} \big) \|\rho_{2n-2k}\|^{80},
							\forall  n,l\in \mN \text{ with  }n-l \geq 1  \text{ and }s\in (2l,2l+1),
						\end{split}
					\end{eqnarray*}
					here  and below  we adopt the notation: $\prod_{j=i_1}^{i_2}a_j=1$ and $\sum_{j=i_1}^{i_2}a_j=0$ for $i_1>i_2$.
					{
						By the above inequality, for any $\kappa>0,$ one has
						\begin{eqnarray}
							\nonumber  && \mE \Big[\| \cD_s^i \rho_{{2n} }\|^{80}  \exp\Big\{-\cC \int_{0}^{2n}\|u_r\|_{L^\mmm}^\mmm \dif r-2n\cC
							-\kappa \sum_{i=0}^{2n}\|u_i\|_{\nn+5}^2\Big\} \tilde T_{2n}^{u_0} \Big]
							\\ \nonumber &&\leq
							C_\beta  \sum_{k=1}^{n-l}\mE\Big[
							\Big( \prod_{j=1}^{k-1} \zeta_{2n-2j}
							\exp
							\big\{-\kappa \|u_{2n-2j}\|_{\nn+5}^2\big\}
							\Big) \cdot \|\rho_{2n-2k}\|^{80}\tilde T_{2n}^{u_0}
							\Big]
							\\ \label{p1104-3} &&:=C_\beta \sum_{k=1}^{n-l} \mathcal I_k.
						\end{eqnarray}
						For any $2\leq k\leq n-l$, by (\ref{p3-4}) and (\ref{p1104-2}), it holds that
						\begin{align*}
							\mathcal I_k \leq  &\mE\Big[  \prod_{j=1}^{k-1}\Big( \zeta_{2n-2j}
							\exp\Big\{-\kappa \|u_{2n-2j}\|_{\nn+5}^2\Big\}
							\Big) \cdot \|\rho_{2n-2k}\|^{80}
							\tilde T_{2n-2k+1}^{u_0}
							\Big]
							\\  \leq &  \mE\Bigg[  \mE \Big[ \prod_{j=1}^{k-1}\Big( \zeta_{2n-2j}
							\exp\Big\{-\kappa \|u_{2n-2j}\|_{\nn+5}^2 \Big\}
							\Big) \cdot \|\rho_{2n-2k}\|^{80}\tilde T_{2n-2k+1}^{u_0}   \Big|\cF_{2n-2}
							\Big]
							\Bigg]
							\\ \leq  & \delta  \mE\Big[  \prod_{j=2}^{k-1}\Big( \zeta_{2n-2j}
							\exp\Big\{-\kappa \|u_{2n-2j}\|_{\nn+5}^2\Big\}
							\Big) \cdot \|\rho_{2n-2k}\|^{80}
							\tilde T_{2n-2k+1}^{u_0}
							\Big]
							\\ \leq & \cdots
							\\ \leq & \delta^{k-1} \mE\Big[    \|\rho_{2n-2k}\|^{80}\tilde T_{2n-2k+1}^{u_0}
							\Big]
							\\ \leq & \delta^{k-1} \mathscr{C}^{n-k}    \delta^{n-k}= \mathscr{C}^{n-k}  \delta^{n-1}.
						\end{align*}
						With similar arguments, the above inequality also applies to the case $\mathcal I_1$.Combining the above estimates of $ \mathcal I_k, 1\leq k\leq n-l$  with (\ref{p1104-3}), also in view of the fact: $\int_{0}^{2n}\|u_r\|_{L^\mmm}^\mmm \dif r(\omega)\leq  2 \Upsilon n $ holds on the event $\{\omega: \tilde T_{2n}^{u_0}\neq 0\}$, we arrive at
						\begin{eqnarray*}
							&& \mE \Big[\| \cD_s^i \rho_{{2n} }\|^{40}   \tilde T_{2n}^{u_0}  \Big]
							\\ && \leq 	e^{2 \Upsilon n\cC+2n\cC } \mE \Bigg[\| \cD_s^i \rho_{{2n} }\|^{40}  \exp\Big\{-\cC \int_{0}^{2n}\|u_r\|_{L^\mmm}^\mmm \dif r-2n\cC-\kappa \sum_{i=0}^{2n}\|u_i\|_{\nn+5}^2 \Big\} \tilde T_{2n}^{u_0} \\ && \quad\quad\quad\quad \quad\quad\quad\quad \times \exp\big\{ \kappa \sum_{i=0}^{2n}\|u_i\|_{\nn+5}^2   \big\}\Bigg]
							\\
							&& \leq 	e^{2 \Upsilon n\cC+2n\cC }  \Bigg[ \mE \| \cD_s^i \rho_{{2n} }\|^{80}  \exp\Big\{-2\cC \int_{0}^{2n}\|u_r\|_{L^\mmm}^\mmm \dif r-4n\cC-2\kappa \sum_{i=0}^{2n}\|u_i\|_{\nn+5}^2 \Big\} \tilde T_{2n}^{u_0}\Bigg]^{1/2} \\ && \quad\quad\quad\quad \Big[ \mE \exp\big\{2  \kappa \sum_{i=0}^{2n}\|u_i\|_{\nn+5}^2   \big\} \tilde T_{2n}^{u_0} \Big]^{1/2}
							\\ && \leq C_{\kappa,\beta,\mathfrak {R}} (n-l){ \sqrt{ \mathscr{C}^{n-1}  \delta^{n-1}} } \exp\big\{2 C \Upsilon n +2C n \big\}C_\kappa^n,~\forall l\leq n-1 \text{ and } s\in (2l,2l+1).
						\end{eqnarray*}
						{   	In the last inequality of the above, we have used Lemma \ref{p17-4}.}
						Recall that $\mathscr{C}$ is a constant given in (\ref{p1104-2}) which only depends on $\Upsilon$, $\nu,d,\Bbbk, \{b_i\}_{i\in \cZ_0},\mathbb{U}$,
						and $(c_{\mathbbm{i},\mathbbm{j}})_{1\leq \mathbbm{i}\leq d,0\leq \mathbbm{j} \leq \Bbbk}$.
						The above inequality implies  the desired result  (\ref{pp-3}).  
					}
				\end{proof}

				\subsection{The estimation of non-adapted integral   $\int_0^{t}v(s)\dif W(s)$.}\label{skorohod}
				Recall that $\mmm=40 \Bbbk d(d+2\Bbbk)$ and   $\mathcal{E}_\mmm$ is a constant decided by     Lemma \ref{L^p 2.2}. For any  $\Upsilon>0, n\in \mN$ and $1\leq k\leq n$,  recall that
				\begin{eqnarray}
					\nonumber &&T_{n}^{u_0}  = \Pi_{i=1}^n   \chi\Big(\int_0^i\|u_r^{u_0}\|_{L^\mmm}^\mmm\dif r-\Upsilon   i\Big),
					\\  \nonumber &&    T_{n,k}^{u_0}  = \Pi_{i=1,i\neq k}^n   \chi\Big(\int_0^i\|u_r^{u_0}\|_{L^\mmm}^\mmm\dif r-\Upsilon i\Big)
					\cdot  \chi'\Big(\int_0^k\|u_r^{u_0}\|_{L^\mmm}^\mmm\dif r-\Upsilon k\Big),
				\end{eqnarray}
				where $(u_s^{u_0})_{s\geq 0}$ is the solution to (\ref{1-1}) with initial value $u_0\in \tilde H^\nn$,  $\chi:\mR\rightarrow [0,1]$ is  a function defined in (\ref{p1110-3}).
				{  The primary objective of this subsection is to establish a bound for the Skorokhod integral $\int_0^{t}v(s)\dif W(s)$, specifically to prove the following proposition.}
				\begin{proposition}
					\label{4-3}
					For any $\mathfrak {R}>0,\Upsilon\geq 4 \mathcal{E}_\mmm,u_0\in \tilde  H^{\nn+5}$ with $\|u_0\|_{\nn+5} \leq \mathfrak {R}$, there exists a
					sufficiently large number  $\gamma_0=\gamma_0( \mathfrak {R},\Upsilon)$  such that,
					if we let $\beta=\beta(\gamma_0,\mathfrak {R},\Upsilon)$ be a constant decided by Proposition \textup{\ref{4-2}}
					and set the direction of $v$ according to \eqref{p-1}, then it holds that
					\begin{eqnarray*}
						&& \mE_{u_0}   \Big| \int_{0}^{2n} v(s)\dif W(s) T_{2n}^{u_0}\Big|^2   \leq C_{\gamma_0,\mathfrak {R},\Upsilon}<\infty ,~\forall n\in \mN.
					\end{eqnarray*}
					In the above, $\gamma_0( \mathfrak {R},\Upsilon)$ denotes a constant depending on $\mathfrak {R},\Upsilon,\nu,d,\Bbbk,(b_j)_{j\in \cZ_0},\mathbb{U}$,
					$(c_{\mathbbm{i},\mathbbm{j}})_{1\leq \mathbbm{i}\leq d,0\leq \mathbbm{j} \leq \Bbbk}$, and $C_{\gamma_0,\mathfrak {R},\Upsilon}$ is a constant depending on $\gamma_0,\mathfrak {R},\Upsilon, \nu,d,\Bbbk,(b_j)_{j\in \cZ_0},\mathbb{U},$
					$(c_{\mathbbm{i},\mathbbm{j}})_{1\leq \mathbbm{i}\leq d,0\leq \mathbbm{j} \leq \Bbbk}$.
				\end{proposition}
				Before we state a proof of Proposition \ref{4-3}, we give a lemma first.
				\begin{lemma}
					\label{p11-6}
					For any $\beta>0$,  set the direction of $v$ according to \eqref{p-1}.  Then for any $k,l\in \mN$ with $k\geq l$ and $ s\in (2l,2l+1)$, it holds that
					\begin{eqnarray}
						\label{p11-5}
						\begin{split}
							\|\mathcal{D}_s^i v_{2k,2k+1}\| &\leq   C_\beta  \exp\Big\{C \int_{2l}^{2k+1}\|u_t\|_{L^\mmm}^\mmm \dif t\Big\}\|\rho_{2k}\|
							\\ &\quad +C_\beta \exp\Big\{C \int_{2k}^{2k+1}\|u_t\|_{L^\mmm}^\mmm \dif t\Big\}
							\|\cD_s^i\rho_{2k}\|,
						\end{split}
					\end{eqnarray}
					where $\|\mathcal{D}_s^i v_{2k,2k+1}\| :=\big(\int_{2k}^{2k+1}|\mathcal{D}_s^i v_{2k,2k+1}(r)|_{\mR^\mathbb{U}}^2\dif r\big)^{1/2}, C_\beta$ is a constant depending on $\beta,\nu,d,\Bbbk,(b_j)_{j\in \cZ_0},\mathbb{U},(c_{\mathbbm{i},\mathbbm{j}})_{1\leq \mathbbm{i}\leq d,0\leq \mathbbm{j} \leq \Bbbk}$ and $C$ is a constant depending on $\nu,d,\Bbbk,(b_j)_{j\in \cZ_0}$, $\mathbb{U},(c_{\mathbbm{i},\mathbbm{j}})_{1\leq \mathbbm{i}\leq d,0\leq \mathbbm{j} \leq \Bbbk}.$
				\end{lemma}
				\begin{proof}
					By the chain rule of the Malliavin derivative, we get
					\baee
					&\cD_s^i v_{2k,2k+1}\\
					=& \cD_s^i \cA^*_{2k,{2k+1}}
					\widetilde\cM_{2k,{2k+1}}^{-1}J_{2k,{2k+1}}\rho_{{2k}}+ \cA^*_{2k,{2k+1}}
					\widetilde\cM_{2k,{2k+1}}^{-1}\cD_s^i J_{2k,{2k+1}}\rho_{{2k}}\\
					&\quad - \cA^*_{2k,{2k+1}}
					\widetilde\cM_{2k,{2k+1}}^{-1}\Big[\left(\cD_s^i \cA_{2k,{2k+1}} \right)\cA^*_{2k,{2k+1}}
					\\ &\quad\quad\quad\quad  \quad\quad\quad\quad \quad\quad \quad  +\cA_{2k,{2k+1}}\left(\cD_s^i \cA^*_{2k,{2k+1}} \right)\Big]\widetilde\cM_{2k,{2k+1}}^{-1}J_{2k,{2k+1}}\rho_{{2k}}
					\\ & \quad+\cA^*_{2k,{2k+1}}
					\widetilde\cM_{2k,{2k+1}}^{-1}J_{2k,{2k+1}} \cD_s^i\rho_{2k}
					\eaee
					where $\widetilde\cM_{2k,{2k+1}}= \cM_{2k,{2k+1}}+\beta \mI.$
					Using Lemma \ref{L:2.2}, and   by the similar argument as that   in \cite[Section 4.8]{HM-2006}, we see
					\bae\label{pre Dv}
					& \left\|\mathcal{D}_s^i v_{2k,2k+1}\right\| \leq C_\beta \left\|\mathcal{D}_s^i A_{2k,2k+1}\right\|\left\|{J}_{2k,2k+1}\right\|\left\|\rho_{2k}\right\|
					\\ & \quad\quad\quad\quad  +C_\beta \left\|\mathcal{D}_s^i {J}_{2k,2k+1}\right\|\left\|\rho_{2k}\right\|
					+C_\beta \|J_{2k,2k+1}\|\|\cD_s^i \rho_{2k}\|.
					\eae
					We divide  the following two cases to prove this lemma.
					
					\textbf{Case 1: $l<k$. }
					For any $\xi\in \tilde  H$ and $s\in (2l,2l+1)$ with $l<k$,
					combining (\ref{ppDj})--(\ref{ppDa}) with (\ref{pre Dv}), it yields the desired result.

					\textbf{Case 2: $l=k$. }
					For any $\xi\in \tilde  H$ and $s\in (2l,2l+1)$,
					combining  \eqref{pppDj} and \eqref{pppDa} with  \eqref{pre Dv}, also in view of $\cD_s^i \rho_{2k}=0$,  we conclude  the desired result (\ref{p11-5}). The proof  is complete.
				\end{proof}
				Now we are in a position to finish  the proof of   Proposition \ref{4-3}.
				\begin{proof}
					By \cite[Proposition 1.3.3]{nualart2006}, we obtain
					\begin{equation*}
						\begin{aligned}
							&T_{2n}^{u_0} \int_{0}^{2n} v(s)\dif W(s)=\int_0^{2n} T_{2n}^{u_0} v(s)\dif W(s) +\left\langle D(T_{2n}^{u_0}),v \right\rangle_{L^2([0,2n],\mathbb R^\mathbb{U})}.
						\end{aligned}
					\end{equation*}
					Thus, by \cite[(1.54)]{nualart2006} we have
					{
						\begin{eqnarray}
							\nonumber && \mE \left| T_{2n}^{u_0} \int_{0}^{2n} v(s)\dif W(s)\right|^2
							\\ \nonumber &&  \leq    2 \mE  \Big| \int_{0}^{2n}
							T_{2n}^{u_0} v(s)  \dd  W(s) \Big|^2
							+2 \mE  \Big|
							\int_{0}^{2n}  \left\langle v(r), \cD_r (  T_{2n}^{u_0} ) \right\rangle_{\mR^\mathbb{U}} \dif r  \Big|^2
							\\  \nonumber && =2\mE \int_0^{2n}|T_{2n}^{u_0} v(s)|^2\dif s
							+2 \mE \int_0^{2n}\int_0^{2n} \operatorname{tr}\Big(\cD_r\big(T_{2n}^{u_0} v(s)\big)\circ \cD_s\big(T_{2n}^{u_0} v(r)\big) \Big)\dif s\dif r
							\\  \nonumber && \quad\quad+2 \mE  \Big|
							\int_{0}^{2n}  \left\langle v(r), \cD_r (  T_{2n}^{u_0} ) \right\rangle_{\mR^\mathbb{U}} \dif r  \Big|^2
							\\ \nonumber &&  \leq  2 \mE \int_0^{2n}T_{2n}^{u_0}| v(s)|^2\dif s
							\\ \nonumber
							\nonumber && \quad\quad +2 \mE \int_0^{2n}\int_0^{2n} \big(T_{2n}^{u_0}\big)^2   \operatorname{tr}\Big(\cD_r v(s)\circ \cD_sv(r) \Big)\dif s\dif r
							\\ \nonumber
							&&\quad\quad  +2 \mE \int_0^{2n}\int_0^{2n} \operatorname{tr}\Big(\big(\cD_r T_{2n}^{u_0}\otimes v(s)\big)\circ \big( \cD_s T_{2n}^{u_0}  \otimes v(r)\big) \Big)\dif s\dif r\\
							\nonumber && \quad\quad +2 \mE \int_0^{2n}\int_0^{2n} T_{2n}^{u_0} \operatorname{tr}\Big(\big(\cD_r T_{2n}^{u_0}\otimes v(s)\big)\circ \cD_sv(r) \Big)\dif s\dif r
							\\ \nonumber
							&& \quad\quad +2 \mE \int_0^{2n}\int_0^{2n} T_{2n}^{u_0} \operatorname{tr}\Big(\cD_rv(s)\circ \big(\cD_sT_{2n}^{u_0}\otimes v(r)\big) \Big)\dif s\dif r
							\\  \nonumber && \quad\quad +2 \mE  \Big|
							\int_{0}^{2n}  \left\langle v(r), \cD_r (  T_{2n}^{u_0} ) \right\rangle_{\mR^\mathbb{U}} \dif r  \Big|^2
							\\ \label{p1109-1} &&:=2\sum_{i=1}^6 L_i,
						\end{eqnarray}
					}
					in the  above, $\circ $ denotes  the normal product of two matrices,
					$tr(A)$ is the trace of  matrix $A$ and for any  vector $a,b\in \mR^{\mathbb U}$, $a\otimes b:=a b^T$ is a $\mathbb{U}\times \mathbb{U}$ matrix.
					
					We will estimate $L_i,i=1,\cdots,6$, respectively.
					Recall that
					\begin{equation*}\tilde T_n^{u_0}=\Pi_{i=1}^n   \chi\Big(\int_0^i\|u_r^{u_0}\|_{L^\mmm}^\mmm \dif r-\Upsilon  i-2 \Big),\quad \forall n\in \mN.
					\end{equation*}
					With regard to $\tilde T_n,T_n,T_{n,k}$, the following properties will be frequently  utilized in the proof:
					\begin{itemize}
						\item[(a)]on the event $\{\omega: \tilde  T_{n}^{u_0} (\omega)\neq 0\}$,
						it holds that  $\int_0^i\|u_r^{u_0}\|_{L^\mmm}^\mmm \dif r\leq \Upsilon  i, ~\forall 1\leq i\leq n;$
						\item[(b)] for any $m\leq n, $ $0\leq \tilde T_n^{u_0} \leq \tilde T_m^{u_0} \leq 1; $
						\item[(c)]  $\max\{T_n^{u_0},|T_{n,k}^{u_0}|\} \leq \tilde  T_{n}^{u_0}\leq 1,  ~ k=1,\cdots, n.$
					\end{itemize}
					
					
					\textbf{(1) Estimate of $L_1.$}
					For the term $L_1$, by (\ref{3.8}) and Proposition \ref{4-2}, it yields that
					\baee
					L_1  & =\sum_{k=1}^{n}\int_{2k-2}^{2k-1}T_{2n}^{u_0}| v(s)|^2\dif s
					\\  &  \leq  C \beta^{-1} \sum_{k=1}^n
					\mE\Big[ T_{2n}^{u_0}e^{C \int_{2k-2}^{2k-1}  \|u_r\|^\mmm_{L^\mmm} \dif r }\left\| \rho_{{2k-2}} \right\|^2\Big]
					\\  &  \leq  C \beta^{-1} \sum_{k=1}^n
					\mE\Big[ \tilde  T_{2k-2}^{u_0} \exp\big\{(2k-1)C \Upsilon  \big\}\left\| \rho_{{2k-2}} \right\|^2\Big]
					\\ &\leq   C_{\gamma_0,\mathfrak {R},\Upsilon} \sum_{k=1}^n \exp\big\{(2k-1)C \Upsilon \big\} \exp\{-\gamma_0 k/40 \}.
					\eaee
					When $\gamma_0$ is sufficiently large, i.e., $\gamma_0\geq \gamma_0(\Upsilon)$\footnote{Throughout  this paper, $\gamma_0( \Upsilon)$ denotes a constant that may  depend  on $\Upsilon$ and $\nu,d,\Bbbk,(b_j)_{j\in \cZ_0},\mathbb{U}, (c_{\mathbbm{i},\mathbbm{j}})_{1\leq \mathbbm{i}\leq d,0\leq \mathbbm{j} \leq \Bbbk} .$}, we conclude that
					\baee
					L_1\leq& C_{\gamma_0,\mathfrak {R},\Upsilon}  <\infty, \quad  \forall n\in \mN.
					\eaee
					\textbf{(2) Estimate of $L_2.$}
					Next we calculate $L_2$.By the definition of $v$, it holds that
					\bae
					\label{L_1^2}
					L_2\leq&C \sum_{k=0}^{n-1}\mE \big(T_{2n}^{u_0}\big)^2 \int_{2k}^{2k+1}\int_{2k}^{2k+1}    \big|\cD_s v_{2k,2k+1}(r)\big|_{\mR^{\mathbb{U}}\times\mR^{\mathbb{U}} }^2\dif s\dif r\\
					=&C \mE \big(T_{2n}^{u_0}\big)^2 \sum_{i=1}^{\mathbb{U} }\sum_{k=0}^{n-1} \int_{2k}^{2k+1}    \left\|\cD_s^i v_{2k,2k+1}\right\|^2\dif s,\\
					\eae
					where $\|\cD_s^i v_{2k,2k+1}\|^2 =\int_{2k}^{2k+1}|\cD_s^i v_{2k,2k+1}(r)|_{\mR^{\mathbb{U}}}^2\dif r$. By \eqref{L_1^2}, Lemma \ref{p11-6},  Proposition \ref{4-2} and $\rho_{2k}\in \cF_{2k}$, we have
					\bae\label{L12}
					L_2\leq& C_\beta  \mE\Big(  \sum_{k=0}^{n-1} e^{C \int_{2k}^{2k+1} \|u_r\|^\mmm_{L^\mmm}  \dif r} \tilde T_{2k}^{u_0} \left\|\rho_{2k}\right\|^2\Big)\\
					\leq&C_\beta  \sum_{k=0}^{n-1}  e^{(2k+1)C\Upsilon } e^{-\gamma_0 k/40 }\leq C_{\gamma_0,\mathfrak {R},\Upsilon}<\infty,\quad  \forall n\in \mN.
					\eae
					In the above, we have assumed that $\gamma_0$ is sufficiently large, i.e., $\gamma_0\geq \gamma_0(\Upsilon).$

					\textbf{(3) Estimate of $L_3.$}
					Now we consider the term  $L_3.$
					{ Notice that
						\bae\label{D^i_r}
						&|\cD_r^iT_{2n}^{u_0}|=\bigg|\sum_{k=\lfloor r \rfloor+1}^{2n}  T^{u_0}_{2n,k}
						\cD_r^i\Big(\int_0^{k}\big(\|u_t\|^\mmm_{L^\mmm}+1\big)\dif t\Big)\bigg|\\
						\leq &m\sum_{k=\lfloor r\rfloor+1}^{2n} |T^{u_0}_{2n,k}|
						\int_r^k \left\langle u_t^{\mmm-1} ,J_{r,t}Q \theta_i  \right\rangle \dif s \\
						\leq &C \sum_{k=\lfloor r \rfloor+1}^{2n} |T^{u_0}_{2n,k}|
						\int_r^k  \|u_t\|_{L^{2\mmm-2}}^{\mmm-1} \| J_{r,t}Q \theta_i\|   \dif t,
						\eae
					}
					Then we conclude that
					\bae\label{L13}
					L_3 =&
					\mE \int_0^{2n}\int_0^{2n} \left\langle\cD_s T_{2n}^{u_0} , v(s)\right\rangle_{\mR^\mathbb{U}} \left\langle\cD_r T_{2n}^{u_0}, v(r)\right\rangle_{\mR^\mathbb{U} }\dif s\dif r \\
					=&\mE \Big(\sum_{l=0}^{n-1}\int_{2l}^{2l+1}\left\langle\cD_r T_{2n}^{u_0}, v(r)\right\rangle_{\mR^\mathbb{U} }\dif r\Big)^2 =\mE \Big(
					\sum_{i=1}^{\mathbb{U}} \sum_{l=0}^{n-1}\int_{2l}^{2l+1}v^i(r) \cD_r^i T_{2n}^{u_0}  \dif r\Big)^2
					\\ \leq &C \mE \Bigg(\sum_{i=1}^{\mathbb{U}}\sum_{l=0}^{n-1}\int_{2l}^{2l+1}|v^i(r)|\cdot
					\Big(\sum_{k=2l+1}^{2n} |T^{u_0}_{2n,k}|
					\int_{r}^k \|u_t\|_{L^{2\mmm-2}}^{\mmm-1}\|J_{r,t}Q \theta_i \| \dif t\Big)
					\dif r\Bigg)^2
					\\ =& C \mE \Big(\sum_{i=1}^{\mathbb{U}} \sum_{l=0}^{n-1} \sum_{k=2l+1}^{2n} |T^{u_0}_{2n,k}|  \int_{2l}^{2l+1} \int_{r}^k |v^i(r)|\cdot
					\|u_t\|_{L^{2\mmm-2}}^{\mmm-1}\|J_{r,t}Q \theta_i \| \dif t
					\dif r\Big)^2
					\\  \leq & C \Big(\sum_{i=1}^{\mathbb{U}} \sum_{l=0}^{n-1}\sum_{k=2l+1}^{2n} \big(\mE X_{i,l,k}^2\big)^{1/2}  \Big)^{2}.
					\eae
					In the above,
					$X_{i,l,k}=|T^{u_0}_{2n,k}|  \int_{2l}^{2l+1} \int_{r}^k |v^i(r)|\cdot
					\|u_t\|_{L^{2m-2}}^{m-1}\|J_{r,t}Q \theta_i \| \dif t
					\dif r$, and  in the last inequality  of the above   we have used the following  fact:
					\begin{eqnarray}
						\label{p11-8}
						\begin{split}
							& \mE ( \sum_{\ell } X_{\ell})^2=\mE ( \sum_{\ell,\ell^{'}} X_{\ell} X_{\ell'})=
							\sum_{\ell,\ell^{'}} \mE X_\ell X_{\ell^{'}}
							\\ &  \leq   \sum_{\ell,\ell^{'}}\big( \mE X_\ell^2\big)^{1/2} \big(\mE X_{\ell^{'}}^2\big)^{1/2}=\Big(\sum_{\ell}(\mE X_{\ell}^2)^{1/2} \Big)^2.
						\end{split}
					\end{eqnarray}
					(In this paper, \textbf{the double   integral $\int\int f(t,r)\dif t\dif r$
						is interpreted   as  $\int \big(\int f(t,r)\dif t\big)\dif r$.
						The  triple  integral $\int \int \int f(t,s,r)\dif t\dif s\dif r$ is interpreted   as  $\int \int \big(\int f(t,s,r)\dif t\big) \dif s \dif r.)$ }
					
					Now,   we will give an estimate of $\mE X_{i,l,k}^2$
					for any $1\leq i\leq \mathbb{U}, 0\leq l\leq n-1$ and $2l+1\leq k\leq 2n.$
					In view of  $(\int_{2l}^{2l+1} |v^i(r)| \dif r)^2\leq  \int_{2l}^{2l+1} |v^i(r)|^2 \dif r$   and
					\begin{eqnarray*}
						&&  \Big(\int_{2l}^k
						\|u_t\|_{L^{2\mmm-2}}^{\mmm-1} \dif t
						\cdot  \int_{2l}^k(\|u_t\|_{L^\mmm}^\mmm+1)\dif t  \Big)^2
						\leq C \Big(\int_{2l}^k(\|u_t\|_{L^{2\mmm-2}}^\mmm+1)\dif t  \Big)^4
						\\ && \leq C(k-2l)^{3/4}\int_{2l}^k(\|u_t\|_{L^{2\mmm-2}}^{4\mmm}+1)\dif t
						\leq C(k-2l)^{3/4}\int_{2l}^k(\|u_t\|_{L^{4\mmm}}^{4\mmm}+1)\dif t,
					\end{eqnarray*}
					we arrive at
					\begin{eqnarray}
						\nonumber   && \mE X_{i,l,k}^2= \mE\Big( |T^{u_0}_{2n,k}|   \int_{2l}^{2l+1} \int_{r}^k |v^i(r)|\cdot
						\|u_t\|_{L^{2\mmm-2}}^{\mmm-1}\|J_{r,t}Q \theta_i \| \dif t
						\dif r\Big)^2
						\\   \nonumber  && \leq    C \mE\Big( |T^{u_0}_{2n,k}|    \int_{2l}^{2l+1} \int_{2l}^k |v^i(r)|\cdot
						\|u_t\|_{L^{2\mmm-2}}^{\mmm-1} \dif t
						\dif r\cdot  \int_{2l}^k(\|u_t\|_{L^\mmm}^\mmm+1)\dif t  \Big)^2
						\\   \nonumber  && =C  \mE\Big( |T^{u_0}_{2n,k}|  \int_{2l}^{2l+1} |v^i(r)| \dif r
						\cdot  \int_{2l}^k
						\|u_t\|_{L^{2\mmm-2}}^{\mmm-1} \dif t
						\int_{2l}^k(\|u_t\|_{L^\mmm}^\mmm+1)\dif t  \Big)^2
						\\   \nonumber  && \leq  C(k-2l)^{3/4} \mE\Big[ \big(|T^{u_0}_{2n,k}| \big)^{1/4} \cdot \big(|T^{u_0}_{2n,k}|  \big)^{1/2} \big(  \int_{2l}^{2l+1} |v^i(r)|^2  \dif r\big)^{1/2}
						\cdot  \int_{2l}^k(\|u_t\|_{L^{4\mmm}}^{4\mmm}+1) \dif t   \Big]
						\\   \nonumber  && \leq  C(k-2l)^{3/4} \Big[ \mE |T^{u_0}_{2n,k}| \Big]^{1/4}\cdot \Big[\mE |T^{u_0}_{2n,k}|\cdot   \int_{2l}^{2l+1} |v^i(r)|^2  \dif r \Big]^{1/2}
						\\ \nonumber && \quad\quad\quad\quad  \quad\quad\quad\quad  \quad\quad\quad\quad \quad\quad\quad\quad    \cdot  \Big[\mE \Big(\int_{2l}^k(\|u_t\|_{L^{4\mmm}}^{4\mmm}+1) \dif t\Big)^4   \Big]^{1/4}
						\\ \label{p11-1} &&:= C(k-2l)^{3/4}\cdot  L_{31}^{1/4}\cdot L_{32}^{1/2}\cdot L_{33}^{1/4}.
					\end{eqnarray}
					In the first inequality of the above, we have used the following  fact:
					\begin{equation*}
						\|J_{r,t}Q\theta_i\|\leq C\int_{2l}^k(\|u_s\|_{L^\mmm}^\mmm+1)\dif s ,\quad  2l\leq r\leq 2l+1\leq k,~r\leq t\leq k,
					\end{equation*}
					which is demonstrated in Lemma \ref{15-4}.

					In the followings, we will estimate $L_{31},L_{32},L_{33}$, respectively.
					For any  $k\geq \max\{\frac{16}{\Upsilon},1\}$,
					by Lemma \ref{L^p 2.2} and the fact $\Upsilon\geq  4\mathcal E_\mmm$, it holds that
					\begin{align}
						\nonumber  L_{31}= &  \mE |T^{u_0}_{2n,k}|
						\\ \nonumber  \leq &
						\mP \Big(\int_0^k\|u_r^{u_0}\|_{L^\mmm}^\mmm\dif r-\Upsilon k\geq -4 \Big)
						\\ \nonumber  = & \mP \Big(\int_0^k\|u_r^{u_0}\|_{L^\mmm}^\mmm\dif r-\mathcal E_\mmm  k\geq
						\Upsilon k -\mathcal E_\mmm  k-4 \Big)
						\\ \nonumber \leq &  \mP \Big(\int_0^k\|u_r^{u_0}\|_{L^\mmm}^\mmm\dif r-\mathcal E_\mmm  k\geq
						\Upsilon k -\frac{\Upsilon}{4} k-4  \Big)
						\\ \nonumber \leq & \mP \Big(\int_0^k\|u_r^{u_0}\|_{L^\mmm}^\mmm\dif r-\mathcal E_\mmm  k\geq
						\Upsilon k/2 \Big)\
						\\  \label{pp10-1} \leq & \frac{ C_{\mathfrak {R}}k^{50}}{(\Upsilon k)^{100}}\leq
						\frac{ C_{\mathfrak {R}}k^{50}}{(\mathcal E_\mmm k)^{100}}\leq  \frac{ C_{\mathfrak {R}}}{ k^{50}}
						,
					\end{align}
					where the constant  $\mathcal C_{\mathfrak {R}}\in (1,\infty)$ is a constant depending on $\mathfrak {R},m $ and  $\nu,d,\Bbbk,(b_j)_{j\in \cZ_0},\mathbb{U}$,
					$(c_{\mathbbm{i},\mathbbm{j}})_{1\leq \mathbbm{i}\leq d,0\leq \mathbbm{j} \leq \Bbbk}.$
					Eventually,  $\mathcal C_{\mathfrak {R}}$ depends on  $\mathfrak {R}$ and $\nu,d,\Bbbk,(b_j)_{j\in \cZ_0},\mathbb{U}$,
					$(c_{\mathbbm{i},\mathbbm{j}})_{1\leq \mathbbm{i}\leq d,0\leq \mathbbm{j} \leq \Bbbk}.$
					Obviously, for any $k\geq 1$, $ L_{31}\leq 1.$
					Thus, for any $k\geq 1$ and $\Upsilon\geq  4\mathcal E_m$, by (\ref{pp10-1}),  we arrive at
					\begin{eqnarray}
						\label{p11-3}
						L_{31}=\mE |T^{u_0}_{2n,k}| \leq \frac{ C_{\mathfrak {R}, \Upsilon}}{ k^{50}},
					\end{eqnarray}
					where    $ C_{\mathfrak {R}, \Upsilon}   \in (1,\infty)$ is a constant depending on $\mathfrak {R},\Upsilon, \mmm $ and  $\nu,d,\Bbbk,(b_j)_{j\in \cZ_0},\mathbb{U}$,
					$(c_{\mathbbm{i},\mathbbm{j}})_{1\leq \mathbbm{i}\leq d,0\leq \mathbbm{j} \leq \Bbbk}.$

					For the term $L_{32}$, by (\ref{3.8}) and Proposition \ref{4-2}, one gets
					\begin{align*}
						L_{32} \leq &  C_\beta \mE \Big[ |T^{u_0}_{2n,k}|   \exp\Big\{C\int_{2l}^{2l+1}\|u_r\|_{L^\mmm}^\mmm \dif r\Big\}\|\rho_{2l}\|^2\Big]
						\\ \leq &  C_\beta \mE \Big[  \tilde  T_{2l+1}^{u_0}  \exp\Big\{C\int_{2l}^{2l+1}\|u_r\|_{L^\mmm}^\mmm \dif r\Big\}\|\rho_{2l}\|^2\Big]
						\\ \leq & C_\beta \exp\Big\{(2l+1)C\Upsilon \Big\} e^{-\gamma_0 l/40  }.
					\end{align*}
					For the term $L_{33}$, by Lemma \ref{L^p 2.2} and H\"older's inequality, we arrive that{
						\begin{eqnarray*}
							L_{33}\leq Ck^3(\|u_0\|_{L^{16\mmm}}^{16\mmm}+k), \quad \forall k\geq 1.
					\end{eqnarray*}}
					Combining the above estimates of $L_{31},L_{32},L_{33}$  with (\ref{p11-1}),
					for any $0\leq l\leq n-1$ and $2l+1\leq k\leq 2n$,
					we conclude that
					\begin{eqnarray}
						\begin{split}
							\label{p0204-1}
							& \mE X_{i,l,k}^2\leq
							C_{\beta, \mathfrak {R} ,   \Upsilon   }(k-2l)^{3/4}
							\big(\frac{1}{k^{50}}\big)^{1/4}
							\\ &\quad\quad\quad\times \exp\big\{(2l+1)C \Upsilon \} e^{-\gamma_0 l/80   } {  k^{3/4}(\|u_0\|_{L^{16\mmm}}^{16\mmm}+k)^{1/4},}
						\end{split}
					\end{eqnarray}
					where   $\mathcal C_{\beta, \mathfrak {R} ,   \Upsilon  }\in (1,\infty)$ is a constant depending on $\beta,\mathfrak {R},\Upsilon  $ and  $\nu,d,\Bbbk,(b_j)_{j\in \cZ_0},\mathbb{U}.$
					Substituting   the  above into (\ref{L13}), after some simple calculations and in view of $\beta=\beta(\gamma_0,\Upsilon,\mathfrak {R})$, we arrive at
					\begin{eqnarray*}
						L_3 \leq C_{\gamma_0,\mathfrak {R} ,   \Upsilon}<\infty,\quad  \forall n\in \mN,
					\end{eqnarray*}
					provided  that  $\gamma_0$ is sufficiently large, i.e., $\gamma_0\geq \gamma_0(\Upsilon)$.

					\textbf{(4) Estimate of $L_4.$}
					For the term   $L_4$,   it holds that
					\bae\label{L_1^{(4)}}
					&L_{4}\\
					\leq& C \mE \int_0^{2n}\int_0^{2n} \left|\cD_rT_{2n}^{u_0}\right| |v(s)|\left| \cD_s v(r) \right|_{\mR^\mathbb{U}\times\mR^\mathbb{U}}\dif s\dif r\\
					\leq& C \mE \sum_{k=0}^n\sum_{l=0}^k \int_{2k}^{2k+1}\int_{2l}^{2l+1} \left|\cD_rT_{2n}^{u_0}\right| |v(s)|\left| \cD_sv_{2k,2k+1}(r)\right|_{\mR^\mathbb{U}\times\mR^\mathbb{U}}\dif s\dif r\\
					\leq &
					C \sum_{k=0}^n\sum_{l=0}^k\sum_{j=2k+1}^{2n}  \mE  \int_{2k}^{2k+1}\int_{2l}^{2l+1}
					\int_r^j
					|T^{u_0}_{2n,j}|\\
					&\quad\quad \cdot   \|u_t\|_{L^{2\mmm-2}}^{\mmm-1} \| J_{r,t}Q \theta_i\|  |v(s)|\left| \cD_sv_{2k,2k+1}(r)\right|_{\mR^\mathbb{U}\times\mR^\mathbb{U}}\dif t\dif s\dif r
					\\ :&= C \sum_{k=0}^n\sum_{l=0}^k\sum_{j=2k+1}^{2n}
					\mathcal Q_{k,l,j},
					\eae
					where  in the last  inequality we have used (\ref{D^i_r}).
					Now we give an estimate of   $\mathcal Q_{k,l,j}$ first.

					For any  $0\leq k\leq n, 0\leq l\leq k, 2k+1\leq j\leq 2n$,
					by (\ref{3.8}), Lemma \ref{15-4} and Lemma \ref{p11-6},
					we conclude that
					\begin{align}
						\nonumber  \mathcal  Q_{k,l,j}
						\leq& C \sum_{i=1}^{\mathbb{U}}
						\mE \Big[ \int_{2k}^{2k+1}\int_{2l}^{2l+1}
						| T^{u_0}_{2n,j}|\cdot    |v(s)|\left| \cD_s^iv_{2k,2k+1}(r)\right|_{\mR^\mathbb{U}}\dif s\dif r
						\\  \nonumber & \quad\quad\quad\quad\quad \quad\quad \times  \int_{2k}^j \|u_t\|_{L^{2m-2}}^{m-1}  \dif t
						\int_{2k}^j (\|u_t\|_{L^{\mmm}}^{\mmm}+1)  \dif t\Big]
						\\  \nonumber  \leq& C \sum_{i=1}^{\mathbb{U}}
						\mE \Big[\int_{2l}^{2l+1}
						| T^{u_0}_{2n,j}|\cdot  |v(s)|\cdot \| \cD_s^iv_{2k,2k+1}\| \dif s
						\Big(
						\int_{2k}^j (\|u_t\|_{L^{2\mmm-2}}^{\mmm}+1)  \dif t\Big)^2 \Big]
						\\ \nonumber
						\leq& C \sum_{i=1}^{\mathbb{U}}
						\mE \Bigg[  | T^{u_0}_{2n,j}|  \Big(\int_{2l}^{2l+1}
						|v(s)|^2\dif s\Big)^{1/2}
						\Big(\int_{2l}^{2l+1}\| \cD_s^iv_{2k,2k+1}\|^2  \dif s
						\Big)^{1/2}
						\\ & \nonumber  \quad\quad \quad\quad\quad\quad\quad\quad \times \Big(
						\int_{2k}^j (\|u_t\|_{L^{2\mmm-2}}^{\mmm}+1)  \dif t\Big)^2 \Bigg]
						\\   \nonumber \leq& C_\beta  \sum_{i=1}^{\mathbb{U}}
						\mE \Bigg[ | T^{u_0}_{2n,j}|\cdot  \|\rho_{2l}\|
						\Big(\|\rho_{2k}\|^2 +\int_{2l}^{2l+1}\| \cD_s^i\rho_{2k}\|^2 \dif s
						\Big)^{1/2}\exp\Big\{C \int_{2l}^{2k+1}\|u_t\|_{L^\mmm}^\mmm \dif t\Big\}
						\\ & \nonumber  \quad\quad \quad\quad\quad\quad\quad\quad \times \Big(
						\int_{2k}^j (\|u_t\|_{L^{2\mmm-2}}^{\mmm}+1)  \dif t\Big)^2 \Bigg]
						\\ \nonumber  \leq& C_\beta  \exp\big\{(2k+1)C\Upsilon \big\}
						\sum_{i=1}^{\mathbb{U}} \mE ( L_{41}L_{42}L_{43}L_{44})
						\\ \label{p11-4} \leq& C_\beta  \exp\big\{(2k+1)C\Upsilon \big\}
						\sum_{i=1}^{\mathbb{U}}  \big(\mE  L_{41}^4\big)^{1/4} \big(\mE  L_{42}^4\big)^{1/4} \cdot \big(\mE  L_{43}^4\big)^{1/4}  \big(\mE  L_{44}^4\big)^{1/4},
					\end{align}
					where  $ \| \cD_s^iv_{2k,2k+1}\|=\big(\int_{2k}^{2k+1}|\cD_s^iv_{2k,2k+1}(r)|_{\mR^\mathbb{U}}^2\dif r\big)^{1/2}$ and
					\begin{eqnarray*}
						&&  L_{41}:=\big|T^{u_0}_{2n,j}\big|^{1/4}, \quad L_{42}:=\big|T^{u_0}_{2n,j}\big|^{1/4} \|\rho_{2l}\|,
						\\ && L_{43}:=\big|T^{u_0}_{2n,j}\big|^{1/4} \big(\|\rho_{2k}\|^2+\int_{2l}^{2l+1}\| \cD_s^i\rho_{2k}\|^2 \dif s
						\big)^{1/2},
						\\ && L_{44}:=\Big(
						\int_{2k}^j (\|u_t\|_{L^{2\mmm-2}}^{\mmm}+1)  \dif t\Big)^2 .
					\end{eqnarray*}
					By  (\ref{p11-3}), we get
					\begin{eqnarray*}
						\mE L_{41}^4 \leq \frac{C_{\mathfrak {R}, \Upsilon}}{j^{50}} , \quad \forall j\geq 1.
					\end{eqnarray*}
					For the term $L_{42}$, by  Proposition \ref{4-2}, we conclude that
					\begin{eqnarray*}
						\mE  L_{42}^4 \leq \mE  \tilde T^{u_0}_{2l}\|\rho_{2l}\|^4\leq C_{\gamma_0,\mathfrak {R},\Upsilon}e^{-\gamma_0 l/20}.
					\end{eqnarray*}
					For the term $L_{43}$, also   by Proposition \ref{4-2},  we have
					\begin{align*}
						\mE L_{43}^4  \leq &  C_{\gamma_0,\mathfrak {R},\Upsilon} \mE \big[ \tilde T^{u_0}_{2k+1} \big(\|\rho_{2k}\|^4+\int_{2l}^{2l+1}\| \cD_s^i\rho_{2k}\|^4 \dif s
						\big) \big]
						\\ \leq& {    C_{\gamma_0,\mathfrak {R},\Upsilon} \exp\{-\gamma_0 k/10 \}.}
					\end{align*}
					By Lemma \ref{L^p 2.2}, we see that
					\begin{eqnarray*}
						{   \mE L_{44}^4 \leq C j^7 (\|u_0\|_{L^{16\mmm}}^{16\mmm}+j ),\quad \forall j\geq 1.}
					\end{eqnarray*}
					Combining the estimates of $\mE L_{4i}^4,i=1,\cdots,4$ with (\ref{p11-4})(\ref{L_1^{(4)}}),  we arrive at
					\begin{eqnarray*}
						L_4\leq C_{\gamma_0,\mathfrak {R},\Upsilon} <\infty, \quad  \forall n\in \mN,
					\end{eqnarray*}
					if   $\gamma_0$ is sufficiently large, i.e., $\gamma_0\geq \gamma_0(\Upsilon).$
					
					\textbf{(5) Estimate of $L_5.$}
					Following the same  arguments  in the estimate of $L_4$, we arrive at
					\begin{eqnarray*}
						L_5\leq C_{\gamma_0,\mathfrak {R},\Upsilon} <\infty, \quad  \forall n\in \mN,
					\end{eqnarray*}
					provided that  $\gamma_0$ is sufficiently large, i.e., $\gamma_0\geq \gamma_0(\Upsilon).$

					\textbf{(6) Estimate of $L_6.$}
					In the end, we give give an estimate of $L_6$.
					By direct calculations,  by (\ref{D^i_r}) and (\ref{p11-8}), we get
					\begin{eqnarray}
						\nonumber && L_6
						\leq \mE  \Big|\sum_{k=0}^{n-1}
						\int_{2k}^{2k+1}  \langle v(r), \cD_r (  T_{2n}^{u_0} ) \rangle_{\mR^\mathbb{U}} \dif r  \Big|^2
						\\  \nonumber && \leq  \mE  \Big|\sum_{k=0}^{n-1}
						\sum_{j=2k+1}^{2n} \int_{2k}^{2k+1}\int_{r}^{j}  |v(r)| \cdot
						T^{u_0}_{2n,j}  \|u_t\|_{L^{2\mmm-2}}^{\mmm-1}
						\| J_{r,t}Q \theta_i\|   \dif t
						\dif r  \Big|^2
						\\  \nonumber  & & \leq    \Bigg(\sum_{k=0}^{n-1}
						\sum_{j=2k+1}^{2n} \Big(\mE \Big| \int_{2k}^{2k+1}\int_{r}^{j}  |v(r)| \cdot
						T^{u_0}_{2n,j}  \|u_t\|_{L^{2\mmm-2}}^{\mmm-1}
						\| J_{r,t}Q \theta_i\|   \dif t
						\dif r  \Big|^2 \Big)^{1/2}\Bigg)^2.
						\\ && \label{p112-1}
					\end{eqnarray}
					For any $0\leq k\leq n-1$ and $j\in [2k+1,2n]$,
					by  (\ref{3.8}), (\ref{p11-3}), { Lemma \ref{L^p 2.2}},  Lemma \ref{15-4} and Proposition \ref{4-2}, it holds that
					\begin{eqnarray*}
						&& \mE \Big| \int_{2k}^{2k+1}\int_{r}^{j}  |v(r)| \cdot
						|T^{u_0}_{2n,j}|\cdot    \|u_t\|_{L^{2\mmm-2}}^{\mmm-1}
						\| J_{r,t}Q \theta_i\|   \dif t
						\dif r  \Big|^2
						\\ && \leq \mE  \Big[ |T^{u_0}_{2n,j}|\cdot \Big|   \int_{2k}^{2k+1}  |v(r)| \dif r
						\cdot \int_{2k}^j \|u_t\|_{L^{2\mmm-2}}^{\mmm-1} \dif t \cdot
						\int_{2k}^j  (\|u_t\|_{L^{\mmm}}^{\mmm}+1)\dif t \Big|^2\Big]
						\\ &&\leq C \mE  \Big[ |T^{u_0}_{2n,j}|\cdot |v(\cdot )|_{L^2([2k,2k+1],\mR^\mathbb{U})}^2
						\cdot
						\big(\int_{2k}^j  (\|u_t\|_{L^{2\mmm-2}}^{\mmm}+1)\dif t\big)^4
						\Big]
						\\ &&\leq C_\beta e^{(2k+1)C\Upsilon   } \mE  \Big[ |T^{u_0}_{2n,j} |^{1/4}\cdot \big(\tilde T^{u_0}_{2k}\big)^{1/4}\|\rho_{2k}\|
						\cdot
						\big(\int_{2k}^j  (\|u_t\|_{L^{2\mmm-2}}^{\mmm}+1)\dif t\big)^4
						\Big]
						\\ && \leq C_\beta e^{(2k+1)C \Upsilon  }  \Big(\mE |T^{u_0}_{2n,j}|  \Big)^{1/4 }\Big(\mE  \tilde T^{u_0}_{2k} \|\rho_{2k}\|^4\Big)^{1/4}
						\Big( \mE  \big(\int_{2k}^j  (\|u_t\|_{L^{2\mmm-2}}^{\mmm}+1)\dif t\big)^{8}\Big)^{1/2}
						\\ &&\leq C_{\beta,\mathfrak {R},\Upsilon} e^{(2k+1)C \Upsilon}  \big(\frac{1}{j^{50}}\big)^{1/4}  \exp\{-\gamma_0k /80\}\cdot
						{ j^{7/2} (\|u_0\|_{L^{16\mmm}}^{16\mmm}+j )^{1/2}
						}.
					\end{eqnarray*}
					Combining the above estimate with (\ref{p112-1}), we get
					\begin{eqnarray*}
						L_6\leq C_{\gamma_0,\mathfrak {R},\Upsilon}<\infty, \quad  \forall n\in \mN,
					\end{eqnarray*}
					provided that  $\gamma_0$ is sufficiently large, i.e., $\gamma_0\geq \gamma_0(\Upsilon).$

					Combining the estimates of $L_i,i=1,\cdots,6$ we complete the proof.
				\end{proof}

				\subsection{A  proof of Proposition \ref{3-11}}
				\label{sec 3.1}

				\begin{proof}
					Assume that $\mathfrak {R}>0$,  $u_0,u_0'\in B_{H^{\nn+5}}(\mathfrak {R})=\{u\in \tilde H^{\nn+5}: \|u\|_{\nn+5}< \mathfrak {R}\}$ and $f\in C_b^1(\tilde H).$

					For any $\eps>0$, we first choose $\Upsilon\geq 4 \mathcal E_\mmm +5$  sufficiently large so that (\ref{I1 grad}) and (\ref{I3 grad}) hold.
					Next, let $\gamma_0= \gamma_0( \mathfrak {R},\Upsilon)$ and   $\beta=\beta(\gamma_0,\mathfrak {R},\Upsilon)$ be positive constants determined by Proposition \ref{4-3}.  Finally, we define $v$ according to (\ref{p-1}).  With these choices, the conclusions of Proposition \ref{4-2} and Proposition \ref{4-3} are satisfied.
					
					Following the analysis at the beginning of this section, it remains to estimate $I_2$  using gradient estimates of  $K_nf(u_0)$.For any $t \geq 0$ {  and $\xi\in \tilde H $ with $\|\xi\|=1$.} Recalling \eqref{Kn}--\eqref{div J2},
					it holds that
					\begin{eqnarray}
						\label{I_2 grad}
						D_{\xi}K_nf(u_0)=J_{12}+(J_{11}+J_{21})+J_{22}.
					\end{eqnarray}

					First, by Proposition \ref{4-2}, we have
					\bae\label{J12 grad}
					J_{12} = \mE \left[ (D  f)(u_n^{u_0})\rho_{n} \cdot T_n^{u_0}\right] \leq  \|D f\|_{L^\infty}  \mathbb{E}\|\rho_{n}T_n^{u_0}\|\leq
					C_{\gamma_0,\mathfrak {R}, \Upsilon}   \|D  f\|_{L^\infty} e^{-\gamma_0 n/80}.
					\eae
					From the integration by part formula  in the Malliavin calculus,  we have
					\begin{equation*}
						J_{11} + J_{21} = \mathbb E\big[\mathcal{D}^{v}\big(f(u_{n}^{u_{0}}) T_n^{u_0}\big)\big] = \mathbb{E}\bigg[f(u_{n}^{u_{0}})T_n^{u_0}\int_{0}^{n}v(s) dW(s)\bigg].
					\end{equation*}
					By Proposition \ref{4-3}, we have
					\begin{equation}\label{24-4-12-2}
						| J_{11} + J_{21}  | \leq \|f\|_{L^\infty}  \mathbb{E} \big| \int_{0}^{n}v(s)dW(s) \tilde T_n^{u_0}  \big|\leq C_{\gamma_0,\mathfrak {R},\Upsilon}\|f\|_{L^\infty} <\infty,\quad \forall n\in \mN.
					\end{equation}

					By direct calculation
					\begin{eqnarray}
						\nonumber 		&&  J_{22}
						\\ \nonumber && = m\sum_{k=1}^n  \mE \bigg( f(u_n^{u_0})   T^{u_0}_{n,k}
						\int_0^k  \big\langle ( u^{u_0}_r)^{\mmm-1},  \rho_r  \big\rangle\dif r \bigg)
						\\ \nonumber  && \leq C  \|f\|_{L^\infty} \sum_{k=1}^n  \mE \bigg(
						|T^{u_0}_{n,k}|     \int_0^k  \|u_r\|_{L^{2\mmm-2}}^{\mmm-1}\|\rho_r\|  \dif r  \bigg)
						\\ \nonumber
						&& 	\leq  C \|f\|_{L^\infty} \sum_{k=1}^n  \mE \bigg(
						\big|T^{u_0}_{n,k}  \big|^{1/4}\cdot    \big(\int_0^k  \|u_r\|_{L^{2\mmm-2}}^{2\mmm-2}\dif r\big)^{1/2}\cdot  \big|T^{u_0}_{n,k} \big|^{1/4}  \big(\int_0^k \|\rho_r\|^2  \dif r \big)^{1/2} \bigg)
						\\ \nonumber  &&
						\leq C\|f\|_{L^\infty}   \sum_{k=1}^n\Big(\mE  |T^{u_0}_{n,k}|  \Big)^{1/4}  \Big( \mE
						\int_0^k \left\|u_r^{u_0}\right\|_{L^{2\mmm-2}}^{2\mmm-2} \dif r\Big)^{\frac12} \Big[ \mE
						\big(\tilde T^{u_0}_{n}\int_0^k \left\|\rho_r\right\|^2 \dif r\big)^{2}\Big]^{\frac{1}{4}} \\
						\nonumber  && \leq C_\mathcal{R}\|f\|_{L^\infty}   \sum_{k=1}^n \Big(\frac{1}{k^{50}}\Big)^{1/4}  \Big(\|u_0\|_{L^{2\mmm-2}}^{2\mmm-2}+k\Big)^{1/2} \Big(
						\sqrt{k} \int_0^k \mE( \tilde T^{u_0}_{n} \|\rho_r\|^4)\dif r \Big)^{\frac14}
						\\ \nonumber
						&& \leq  C_{\gamma_0,\mathfrak {R},\Upsilon} \|f\|_{L^\infty}  \sum_{k=1}^n  \Big(\frac{1}{k^{50}}\Big)^{1/4}    \Big(\|u_0\|_{L^{2\mmm-2}}^{2\mmm-2}+k\Big)^{1/2} \Big(
						\sqrt{k} \int_0^k \exp\{-\gamma_0 r/20\}\dif r \Big)^{\frac14}
						\\  \label{J_22 grad}&& \leq    C_{\gamma_0,\mathfrak {R},\Upsilon}\|f\|_{L^\infty}  <\infty,
					\end{eqnarray}
					where in the above, we used (\ref{p11-3}),  Lemma \ref{L^p 2.2} and  Proposition \ref{4-2}. 

					Substituting \eqref{J12 grad}--\eqref{J_22 grad} into \eqref{I_2 grad},
					for any  $n\in \mN,u_0\in B_{H^{\nn+5}}(\mathfrak {R}) $
					{   and $\xi\in \tilde H $ with
						$\|\xi\|=1$,
						we have
						\baee
						D_{\xi} K_nf(u_0) \leq C_{\gamma_0,\mathfrak {R},\Upsilon} (\|f\|_{L^\infty} +\|D  f\|_{L^\infty}).
						\eaee
					}
					Let $\gamma(s)=su_0+(1-s)u_0'$. Then, by the above inequality, we arrive at
					\begin{align}
						\nonumber |I_2|  = &
						\big| \mE \big[ f(u_n^{u_0}) T_n^{u_0}\big] -\mE \big[ f\big(u_n^{u_0^\prime}\big) T_n^{u_0'}\big]\big| =|K_nf(u_0)-K_nf(u_0')|
						\\  \nonumber =&  \int_0^1 \langle D  K_nf(\gamma(s)),u_0-u_0'\rangle \dif s
						\\  \nonumber \leq &  C_{\gamma_0,\mathfrak {R},\Upsilon}  (\|f\|_{L^\infty} +\|D f\|_{L^\infty})  \|u_0-u_0'\|.
					\end{align}

					For any  bounded and Lipschitz continuous function $f$ on $H$,
					by the arguments  in \cite[Page 1431]{KPS10},
					there exists a sequence $(f_k)$ satisfies $(f_k)\subseteq C_b^1(\tilde H)$
					and $\lim_{k\rightarrow  \infty}f_k(x)=f(x)$ pointwise.
					In addition, $\|f_k\|_{L^\infty}\leq \|f\|_{L^\infty}$ and $\|D  f_k\|_{L^\infty}\leq \|f\|_{Lip}$, where  $\|f\|_{Lip}=\sup_{x\neq y}\frac{|f(x)-f(y)|}{\|x-y\|}.$
					Therefore, for any $n\in \mN$, one has
					\begin{align}
						\nonumber  |I_2|=&| K_n f(u_0)-K_n f(u_0')|=\lim_{k\rightarrow \infty}| K_n f_k(u_0)-K_n f_k(u_0')|
						\\ \nonumber \leq&  \lim_{n\rightarrow \infty} \Big[C_{\gamma_0,\Upsilon,\mathfrak {R}} \|u_0-u_0'\|\big(\|f_k\|_{L^\infty}+\|D f_k \|_{L^\infty} \big)
						\Big]
						\\ \nonumber  \leq&  C_{\gamma_0,\mathfrak {R},\Upsilon }  \|u_0-u_0'\|(\|f\|_{L^\infty}+\|f\|_{Lip}).
					\end{align}

					Combining the above estimate with
					(\ref{I1 grad})(\ref{I3 grad}),
					for any  $\mathfrak {R}>0$,  bounded and Lipschitz continuous function $f$ on $\tilde H$  and $\eps>0$, there exists a $\delta=\delta(\mathfrak {R},\|f\|_{L^\infty}, \|f\|_{Lip},\eps)>0$ such that
					the following:
					\begin{eqnarray*}
						\big|\mE f(u_n^{u_0})  -\mE f(u_n^{u_0^\prime})\big|<\eps
					\end{eqnarray*}
					holds for any  $u_0,u_0'\in B_{ H^{\nn+5}}(\mathfrak {R})$ with $\|u_0-u_0'\|<\delta.$
					The proof is complete.

				\end{proof}

				\section{Proof of  Irreducibility }
				\label{sec irr}

				First, we present a proposition that is central to establishing irreducibility. Subsequently, we provide a proof of Proposition \ref{16-6}. Recall that  $\nn=\lfloor d/2+1\rfloor.$
				\begin{proposition}
					\label{p24-5}
					For any $\mathcal{C},\gamma>0$, there exist positive  constants  $T=T(\mathcal{C},\gamma)>0,p_0=p_0( \mathcal{C},\gamma)$ such that
					\begin{eqnarray*}
						P_{T}(u_0, \cB_\gamma)\geq p_0( \mathcal{C},\gamma),\quad \forall u_0\in H^{\nn+1} \text{ with }\|u_0\|_\nn\leq \mathcal{C},
					\end{eqnarray*}
					where $\cB_\gamma:=\{u\in H^\nn,\|u\|_\nn \leq \gamma\}.$
					In the above,  $T(\mathcal{C},\gamma)$ and $ p_0( \mathcal{C},\gamma)$ denote  two positive constants depending on  $\mathcal{C},\gamma$ and  $\nu,d,\Bbbk,(b_j)_{j\in \cZ_0},\mathbb{U}$,
					$(c_{\mathbbm{i},\mathbbm{j}})_{1\leq \mathbbm{i}\leq d,0\leq \mathbbm{j} \leq \Bbbk}.$
				\end{proposition}

				\begin{proof}
					Define $v_t:=u_t-  \eta_t$,  where $ \eta_t=\sum_{i\in\cZ_0}b_iW^i_{t} e_i
					$. Then, $v_t, t\geq 0 $ satisfies
					\begin{eqnarray*}
						\left\{
						\begin{split}
							& \frac{\partial v_t}{\partial t}= \nu\Delta (v_t+ \eta_t )
							-\text{div} A(v_t+ \eta_t )
							\\ & v_t |_{t=0}=u_0.
						\end{split}
						\right.
					\end{eqnarray*}
					Using chain rule  to $
					\partial_t \langle v_t, (-\Delta )^\nn v_t   \rangle
					$, it yields that
					\begin{eqnarray}
						\nonumber && \frac{1}{2} \frac{\dif  }{\dif t } \|v_t\|_\nn^2 =\langle \partial_t  v_t, (-\Delta )^\nn v_t   \rangle
						\\  \nonumber  &&=  \nu \langle  \Delta v_t, (-\Delta )^\nn v_t   \rangle
						+\nu \langle  \Delta  \eta_t , (-\Delta )^\nn v_t   \rangle
						+\langle -\text{div} A(u_t ), (-\Delta )^\nn v_t   \rangle
						\\ \label{p14-10} &&:=I_1+I_2+I_3.
					\end{eqnarray}
					
					For the term $I_1$,  it holds that
					\begin{eqnarray}
						\label{p13-1}  I_1=-\nu \|v_t\|_{\nn+1}^2.
					\end{eqnarray}
					Now we consider the term $I_2.$
					By Young's inequality, we get
					\begin{eqnarray}
						\label{p13-2}
						|I_2|\leq \frac{\nu }{8} \|v_t\|_{\nn+1}^2+C \|\eta  \|_{\nn+1}^2.
					\end{eqnarray}
					In the end, we consider   the term $I_3$.By Lemma
					\ref{div reg}, there exist positive constants $\kappa=\kappa_{\nn,d,\Bbbk}>0,m=m_{\nn,d,\Bbbk}\in 2\mN$ with
					$m\geq 4d \Bbbk>\kappa $  such that
					\begin{align*}
						|I_3| \leq&  \|v_t\|_{\nn+1}\| \operatorname{div}A(u_t) \|_{\nn-1}
						\leq \frac{\nu}{8} \|v_t\|_{\nn+1}^2
						+C \| \operatorname{div}A(u_t) \|_{\nn-1}^2
						\\ \leq &  \frac{\nu}{4} \|v_t\|_{\nn+1}^2 +C \|u_t\|_{L^{m }}^{\kappa }+C \|u_t\|_{L^{m}}^{m}
					\end{align*}
					Combining the above with  (\ref{p14-10})--(\ref{p13-2})  , one arrives at
					\begin{eqnarray*}
						\frac{\dif  }{\dif t } \|v_t\|_\nn^2
						\leq -\nu \|v_t\|_{\nn}^2
						+C  (\|v_t\|_{L^m}^{\kappa}+\|\eta_t\|_{L^m}^{\kappa}
						+\|v_t\|_{L^m}^{m}
						+\|\eta_t\|_{L^m}^{m}+\|\eta_t  \|_{\nn+1}^2 ).
					\end{eqnarray*}
					Thus, for any $t\geq 0$, it holds that
					\begin{eqnarray}
						\label{p14-2}
						\begin{split}
							& \|v_t\|_\nn^2  \leq e^{-\nu t} \|u_0\|_\nn^2
							\\ \quad\quad & +C \int_0^t e^{-\nu (t-s)} (\|v_s\|_{L^\mmm}^{\kappa}+\|\eta_s\|_{L^m}^{\kappa}
							+\|v_s\|_{L^m}^{m}
							+\|\eta_s\|_{L^m}^{m} +\|\eta_s  \|_{\nn+1}^2) \dif s.
						\end{split}
					\end{eqnarray}
					Thus, in order to bound  $\|v_t\|_\nn^2$, it is necessary to estimate $\|v_t\|_{L^m}^m.$
					In the following, we aim to establish an estimate for  $\|v_t\|_{L^m}^m.$

					Using chain rule  to $\partial_t \int_{\mT^d} v_t(x)^m\dif x $ yields that
					\begin{eqnarray}
						\nonumber && \frac{\dif  }{\dif t }\int_{\mT^d} v_t(x)^m\dif x
						=\int_{\mT^d} m v_t^{m-1} \partial_t v_t \dif x
						\\   \nonumber &&= \int_{\mT^d} m v_t(x)^{m-1} \big(\nu \Delta (v_t+ \eta_t )
						-\text{div} A(v_t+ \eta_t ) \big) \dif x
						\\   \nonumber &&= \nu m \int_{\mT^d}  v_t(x)^{m-1} \ \Delta v_t \dif x
						+\nu m \int_{\mT^d}  v_t(x)^{m-1}\Delta \eta_t  (x)
						\dif x
						\\   \nonumber &&\quad  -   \int_{\mT^d} m v_t(x)^{m-1}
						\text{div} A(v_t+ \eta_t ) \dif x
						\\ \label{p1126-4} &&:=J_1+J_2+J_3.
					\end{eqnarray}
					
					For the term $J_1$, by (\ref{28-1}), it holds that
					\begin{align*}
						J_1=& -\nu m \int_{\mT^d }v_t(x)^{m-1}\big(-\Delta v_t(x)\big)\dif x
						\\ \leq &
						-\nu m\Big( 		C_{m}^{-1}\|v_t\|_{L^m}^m+\frac{1}{m}\|(-\Delta)^{1/2}v_t^{m/2}\|
						^2\Big)
						\\ \leq & 		-\nu m C_{m}^{-1}\|v_t\|_{L^m}^m-\nu \|(-\Delta)^{1/2}v_t^{m/2}\|^2
					\end{align*}
					in the above,  $C_{m}\in (1,\infty)$ is   a positive constant depending on $m$ and $d.$
					Now we consider the term $J_2.$
					By H\"older's inequality and  Young's inequality, for any $\eps>0$, we get
					{
						\bae
						\label{p1126-1}
						|J_2| \leq& C \int_{\mT^d}  |v_t(x)|^{m-2}\big|\nabla v_t(x)\big| \big|\nabla \eta_t  (x)\big|
						\dif x \\
						\leq& C \int_{\mT^d}  |v_t(x)|^{(m-2)/2}\big|(-\Delta)^{1/2}v_t(x)^{m/2}\big||\nabla \eta_t  (x)|
						\dif x
						\\ \leq & C  \|v_t\|_{L^m}^{\frac{m-2}{2}}
						\|  (-\Delta)^{1/2}v_t^{m/2}  \|
						\|\nabla \eta_t\|_{L^m}
						\\
						\leq&\eps \|v_t\|_{L^m}^m+\frac12 \nu\|(-\Delta)^{1/2}v_t^{m/2}\|^2 +C_\eps \| \nabla \eta_t  \|_{L^m}^m.\\
						\eae}
					For the term $J_3$, by direct calculations, we conclude that
					\begin{eqnarray}
						\label{p1126-2}    && |J_3| \leq  C\zeta_t\int_{\mT^d } \big(|v_t|^{m-1 }(x)+
						|v_t|^{m+\Bbbk-2 }(x)\big)
						\dif x,
					\end{eqnarray}
					{  where
						\begin{align*}
							\zeta_t:=&\|\eta_t \|_{L^\infty}+
							\|\eta_t \|_{L^\infty}^{\Bbbk}+
							\|\nabla\eta_t \|_{L^\infty}+
							\|\nabla\eta_t \|_{L^\infty}^{\Bbbk} +
							\|\eta_t   \|_{L^m}^{\kappa}+
							\|\eta_t \|_{L^m}^{m}+ \|\eta_t \|_{\nn+1}^2.
					\end{align*}}
					In the next, we consider the  second  term in (\ref{p1126-2}).
					Set
					\begin{eqnarray*}
						q=
						\left\{
						\begin{split}
							& \frac{2d}{d-2}, \quad d\geq 3,
							\\
							& 4, \quad d=2,
							\\
							& \infty, \quad d=1.
						\end{split}
						\right.
					\end{eqnarray*}
					Then,  by  H\"older's inequality, it holds that
					\begin{eqnarray}
						\label{p0204-12}
						\begin{split}
							&\int_{\mT^d } |v_t|^{m/2} |v_t|^{m/2} |v_t|^{\Bbbk-2 }(x) \dif x
							\\ &\leq C  \| v_t^{m/2}  \|_{L^q}    \|v_t\|_{L^m}^{m/2} ( \|v_t\|_{L^{(\Bbbk-2)d}}^{\Bbbk-2}+\|v_t\|_{L^{(\Bbbk-2)4}}^{\Bbbk-2}).
						\end{split}
					\end{eqnarray}
					In the above, for the case  $d\geq 3$,   we have used the fact  $\frac{1}{q}+\frac{1}{2}+\frac{1}{d}=1$  and for the cases $d=1,2$,
					the above inequality (\ref{p0204-12})
					can also  be verified by  the value of $q$ and   H\"older's inequality directly.
					By Sobolev embedding theorem, we also have
					\begin{eqnarray}
						\label{p38-1}
						\| v_t^{m/2}  \|_{L^q}   \leq C
						\| v_t^{m/2} \|_{L^2}+
						\|\nabla   v_t^{m/2} \|_{L^2}.
					\end{eqnarray}
					{ Thus,  by \eqref{p0204-12}--(\ref{p38-1}), one arrives at}\footnote{For $\Bbbk=2$, we use the notation that $\|f\|_{L^0}:=1$ for any function $f$ on $\mT^d$.}
					\begin{eqnarray*}
						\nonumber  &&  \zeta_t \int_{\mT^d } |v_t|^{m+\Bbbk-2 }(x) \dif x                          = \zeta_t  \int_{\mT^d } |v_t|^{m/2} |v_t|^{m/2} |v_t|^{\Bbbk-2 }(x) \dif x
						\\ \nonumber &&\leq C  \zeta_t   \| v_t^{m/2}  \|_{L^q}    \|v_t\|_{L^m}^{m/2} ( \|v_t\|_{L^{(\Bbbk-2)d}}^{\Bbbk-2}+\|v_t\|_{L^{(\Bbbk-2)4}}^{\Bbbk-2})
						\\   \nonumber &&\leq     C  \zeta_t   (\|v_t\|_{L^m}^{m/2}+\| \nabla  v_t^{m/2}\|_{L^2})
						\|v_t\|_{L^m}^{m/2} ( \|v_t\|_{L^{(\Bbbk-2)d}}^{\Bbbk-2}+\|v_t\|_{L^{(\Bbbk-2)4}}^{\Bbbk-2})
						\\  \nonumber &&\leq C \zeta_t   \|v_t\|_{L^m}^{m} \|v_t\|_{L^{m}}^{\Bbbk-2}
						+C \zeta_t  \| \nabla  v_t^{m/2}\| \|v_t\|_{L^m}^{m/2}  \|v_t\|_{L^{\mmm}}^{\Bbbk-2}
						\\  \nonumber   &&\leq \eps \| \nabla  v_t^{m/2}\|^2+
						C_\eps (\zeta_t +\zeta_t^2) \|v_t\|_{L^m}^{m}(  \|v_t\|_{L^{m}}^{\Bbbk-2} +  \|v_t\|_{L^{m}}^{2\Bbbk-4}   )
						\\  \nonumber   &&\leq \eps \| \nabla  v_t^{m/2}\|^2+
						C_\eps (\zeta_t +\zeta_t^2) (1+\|v_t\|_{L^m}^{2m}).                            \end{eqnarray*}
					In the above, we have used    $m\geq 4d \Bbbk.$  
					With the help of the above inquality and \eqref{p1126-2}, for any $\eps>0,$ we arrive at
					\begin{eqnarray}
						\label{p38-3}    && |J_3| \leq  \eps  \| \nabla  v_t^{m/2}\|^2
						+
						C_\eps (\zeta_t +\zeta_t^2) (1+\|v_t\|_{L^m}^{2m}).
					\end{eqnarray}
					
					Setting $\eps$ small enough, combining the estimates of $J_1,J_2,J_3$ with \eqref{p1126-4}, we arrive at
					\begin{eqnarray}
						\label{p1126-5}
						\begin{split}
							\frac{\dif  }{\dif t }\|v_t\|_{L^m}^{m}  & \leq -\mathscr{C}_m^{-1}\|v_t\|_{L^m}^{m}  +\mathscr{C}_m(\zeta_t+\zeta_t^2)(1
							+\|v_t\|_{L^m}^{2m} ),~\forall t\geq 0,
						\end{split}
					\end{eqnarray}
					where $\mathscr{C}_m\in (1,\infty)$ is a constant depending on $m$ and
					$d,\Bbbk, (b_{i})_{i\in \cZ_0}, \mathbb{U}$, $(c_{\mathbbm{i},\mathbbm{j}})_{1\leq \mathbbm{i}\leq d,0\leq \mathbbm{j} \leq \Bbbk}.$

					For any $u_0$ with
					$\|u_0\|_\nn \leq \cC$,
					obviously, there exists a constant $\cN\geq 1$ such that
					\begin{eqnarray*}
						\|u_0\|_{L^m}^{m} \leq \cN.
					\end{eqnarray*}
					For $\gamma\in (0,1)$, let
					\begin{equation*}
						T_1=T_1(\gamma,m,\cN)>0
					\end{equation*}
					be a constant such that
					$\exp\big\{-\mathscr{C}_m^{-1} T_1 /2 \big\} \cN<\frac{\gamma}{2}$, let  {  $\delta=\delta(\gamma,\cN, m,\mathscr{C}_m)\in (0,1)$ be constant
						such that}
					\begin{eqnarray}
						&& \label{p1126-6}   -\mathscr{C}_m^{-1} x +\mathscr{C}_m
						\delta (1+x^{2}   ) \leq -\frac{\mathscr{C}_m^{-1}x}{2},\quad  \forall x\in [\gamma/2, \cN].
					\end{eqnarray}
					For any $T_2>0$,  define
					\begin{eqnarray*}
						\Omega^{\gamma, \delta, T_1,T_2}:= \big\{ \omega:  \sup_{s\in [0,T_1+T_2]}  (\zeta_s+\zeta^2_s)   \leq \delta \wedge \gamma \big\}.
					\end{eqnarray*}
					There are the following two cases about $\|v_0\|_{L^m}^{m}.$
					
					\textbf{Case 1:}  $\|v_0\|_{L^m}^{m}=\|u_0\|_{L^m}^{m}\leq \gamma/2.$
					Combining the  fact (\ref{p1126-6}) with  (\ref{p1126-5}),
					for any $\omega\in \Omega^{\gamma, \delta, T_1,T_2}$
					and $t\in [0,T_1]$, it holds that
					\begin{eqnarray}
						\label{p1126-8}
						\|v_t\|_{L^m}^{m}\leq \gamma/2.
					\end{eqnarray}

					\textbf{Case 2:}  $\|v_0\|_{L^m}^{m}=\|u_0\|_{L^m}^{m}\in (\gamma/2, \cN).$
					Define
					\begin{eqnarray*}
						\tau=\inf\{t\geq 0,  \|v_t\|_{L^m}^{m}\leq \gamma/2\}.
					\end{eqnarray*}
					For any $t\leq \tau$ and $\omega\in \Omega^{\gamma, \delta, T_1,T_2}$,
					in view of  (\ref{p1126-6}) and (\ref{p1126-5}),   we have
					\begin{eqnarray*}
						&& \frac{\dif  }{\dif t }\|v_t\|_{L^m}^{m} \leq -\frac{\mathscr{C}_m^{-1}}{2}\|v_t\|_{L^m}^{m}.
					\end{eqnarray*}
					Thus,
					\begin{eqnarray*}
						\|v_t\|_{L^m}^{m}\leq \exp\big\{-\mathscr{C}_m^{-1} t /2 \big\} \cN.
					\end{eqnarray*}
					In view of  the above inequality and the  fact that  $\exp\big\{-\mathscr{C}_m^{-1} T_1 /2 \big\}\cN <\frac{\gamma}{2}$, one easily sees  that $\tau\leq T_1.$
					
					Combining the above two cases,
					for $\omega\in \Omega^{\gamma,\delta,T_1,T_2}$,   we have
					\begin{eqnarray*}
						\|v_{T_1} \|_{L^m}^{m}\leq \frac{\gamma}{2}.
					\end{eqnarray*}
					Putting everything together, one has
					\begin{eqnarray*}
						\|u_0\|\leq  \mathcal{C}   \text{ and  }
						\omega \in  \Omega^{\gamma,\delta,T_1,T_2}  \Rightarrow
						\|v_{T_1} \|_{L^m}^m \leq  \frac{\gamma}{2}.
					\end{eqnarray*}
					In view of the above fact, also with the help of    the following  fact:
					\begin{eqnarray*}
						\frac{\dif  }{\dif t }\|v_t\|_{L^m}^{m}  \leq 0 \text{ if }
						\|v_t\|_{L^m}^{m}\in [\frac{\gamma}{2}, \cN],  \omega \in \Omega^{\gamma, \delta, T_1,T_2}
						\text{ and } t\in [0,T_1+T_2],
					\end{eqnarray*}
					for any $\omega \in \Omega^{\gamma, \delta, T_1,T_2}$,
					one arrives at
					\begin{eqnarray}
						\label{p14-1}
						\begin{split}
							& \|v_{t} \|_{L^m}^m \leq  \frac{\gamma}{2}, \quad \forall t\in [T_1,T_1+T_2],
							\\
							& \|v_{t} \|_{L^m}^m\leq \cN,\quad \forall t\in [0,T_1+T_2].
						\end{split}
					\end{eqnarray}
					
					{
						By (\ref{p14-1}), (\ref{p14-2}) and the definition of $\zeta_t$, for any  $\omega \in \Omega^{\gamma, \delta, T_1,T_2}$,   we conclude that
						\begin{eqnarray}
							\nonumber && \|v_{T_1+T_2}\|_\nn^2  \leq e^{-(T_1+T_2)} \|u_0\|_\nn^2
							\\ \nonumber&& +C \int_0^{T_1} e^{-(T_1+T_2-s)}(\|v_s\|_{L^\mmm}^{\kappa}+\|\eta_s\|_{L^m}^{\kappa}
							+\|v_s\|_{L^m}^{m}
							+\|\eta_s\|_{L^m}^{m} +\|\eta_s  \|_{\nn+1}^2) \dif s
							\\ \nonumber && \quad\quad+ C \int_{T_1}^{T_1+T_2} e^{-(T_1+T_2-s)}(\|v_s\|_{L^m}^{\kappa}+\|\eta_s\|_{L^m}^{\kappa}
							+\|v_s\|_{L^m}^{m}
							+\|\eta_s\|_{L^m}^{m} +\|\eta_s  \|_{\nn+1}^2)  \dif s
							\\ \nonumber &&\leq
							e^{-(T_1+T_2)} \|u_0\|_\nn^2
							+C e^{-T_2} \int_0^{T_1} e^{-(T_1-s)}(\cN+\gamma)  \dif s
							\\ \label{p14-3} && \quad\quad+ C \int_{T_1}^{T_1+T_2} e^{-(T_1+T_2-s)}\gamma \dif s.
						\end{eqnarray}
						In the second inequality of the above, we have used   $\cN\geq 1$ and $\kappa<m.$
					} Therefore, for any $\gamma>0$ and  $u_0\in H$ with $ \|u_0\|\leq  \mathcal{C}$,
					by (\ref{p14-3}),
					we set $ T_2>0$ such that for any $\omega\in \Omega^{\gamma, \delta, T_1,T_2}$,
					it holds that
					\begin{eqnarray*}
						\|v_{T_1+T_2}\|_\nn^2\leq C\gamma.
					\end{eqnarray*}
					Fix this $T_2.$
					For any $\omega\in \Omega^{\gamma, \delta, T_1,T_2}$,
					the above implies that
					\begin{eqnarray*}
						&& \|u^{u_0}_{T_1+T_2}\|_\nn^2 \leq  C\|v^{u_0}_{T_1+T_2}\|_\nn^2 +C\|\eta_t\|_\nn^2
						\\ &&  \leq
						C\|v^{u_0}_{T_1+T_2}\|_\nn^2 +C\zeta_t \leq C\gamma.
					\end{eqnarray*}
					In the end, we conclude that
					\begin{eqnarray*}
						&& \mP\big(\|u^{u_0}_{T_1+T_2}\|_\nn \leq \sqrt{C\gamma}   \big)
						\geq \mP(\Omega^{\gamma, \delta, T_1,T_2})>0.
					\end{eqnarray*}
					The proof is complete.
				\end{proof}

				
				\textbf{Now we are in a position to prove Proposition \ref{16-6}. }
				\begin{proof}
					For any $\eps\in (0,\frac{\gamma}{2}),N\geq 1$,
					$u_0\in H^{\nn}$  with $\|u_0\|_\nn\leq \mathcal{C}$,
					and $u_0' \in H^{\nn+1}$  with $\|u_0'\|_\nn\leq \mathcal{C}$,
					one has
					\begin{eqnarray*}
						&& \mP( \|u_T^{u_0}\|_\nn <\gamma)
						\geq \mP( \|P_N u_T^{u_0}\|_\nn+ \|Q_N u_T^{u_0}\|_\nn <\gamma)
						\\ &&\geq \mP( \|P_N u_T^{u_0}\|_\nn<\gamma-\eps,  \|Q_N u_T^{u_0}\|_\nn <\eps)
						\\ &&\geq  \mP( \|P_N u_T^{u_0}\|_\nn<\gamma-\eps)-\mP(  \|Q_N u_T^{u_0}\|_\nn \geq \eps)
						\\ &&\geq  \mP( \|P_N u_T^{u_0'}\|_\nn<\gamma-2\eps, \|P_N u_T^{u_0'}-P_N u_T^{u_0}\|_\nn<\eps )-\mP(  \|Q_N u_T^{u_0}\|_\nn \geq \eps)
						\\ &&\geq \mP( \|P_N u_T^{u_0'}\|_\nn<\gamma-2\eps)-\mP( \|P_N u_T^{u_0'}-P_N u_T^{u_0}\|_\nn\geq \eps )-\mP(  \|Q_N u_T^{u_0}\|_\nn \geq \eps)
						\\ && \geq p_0(\cC,\gamma-\eps)-\mP( \|P_N u_T^{u_0'}-P_N u_T^{u_0}\|_\nn\geq \eps )-\mP(  \|Q_N u_T^{u_0}\|_\nn \geq \eps),
					\end{eqnarray*}
					where $p_0(\cC,\gamma-\eps)$ is a constant given by Proposition \ref{p24-5}.
					In the above inequality, we first let $u_0'\rightarrow u_0$  and then let $N\rightarrow \infty$,
					we deduce that
					\begin{eqnarray*}
						\mP( \|u_T^{u_0}\|_\nn <\gamma)\geq p_0(\cC,\gamma-\eps).
					\end{eqnarray*}
					The proof is complete.
				\end{proof}

				\begin{appendices}
				\section*{Appendix}
				\appendix

				\section{Existence of an invariant measure}\label{appen B}
				\begin{lemma}
					There exists a probability measure $\mu \in \mathcal{P}\left(H^\nn\right)$ which is invariant \textup{w.r.t.} the Markov semigroup $\left(P_t\right)_{t \geq 0}$.
				\end{lemma}
				\begin{proof}
					Using the Markov inequality and  Lemma \ref{priori H^m}, we have
					\bae\label{tight}
					\frac{1}{T} \int_0^T \mathbb{P}\Big(\|u(t)\|_{\nn}^2>\frac{1}{\varepsilon}\Big) \dif t \leq \frac{C \epsilon}{T}\Big(T+\left\|u_0\right\|_{\nn}^2+\left\|u_0\right\|_{L^{\mm}}^{\mm}
\Big),~{  T\geq 1}.
					\eae
					Let
					\begin{equation*}
						K_{\varepsilon}:=\left\{u \in L^1:\|u\|_{\nn}^2 \leq \frac{1}{\varepsilon}\right\}.
					\end{equation*}
					Notice that the embedding $ \iota:  H^\nn \rightarrow L^1$ is compact. Then $K_{\varepsilon}$ is compact in $L^1$. For $\varsigma\in\mathcal{P}(H^\nn)$, we use the notation $\iota^* \varsigma:= \varsigma\circ \iota^{-1}\in\mathcal{P}(L^1)$. Recall that the empirical measure is given by $R_T^* \varsigma(O)=\frac{1}{T} \int_0^T P_t^* \varsigma(O) \dif t$. For $T\geq1$,  we  rewrite \eqref{tight} as follows:
					\begin{equation*}
						\iota^* {R}_T^* \delta_{u_0}\left(L^1(\mathbb{T}^d) \backslash K_{\varepsilon}\right) \leq \eps C\left(1+\left\|u_0\right\|_{\nn}^2+\left\|u_0\right\|_{L^{\mm}}^{\mm}\right).
					\end{equation*}
					Thanks to Prokhorov's theorem, we deduce that $\{\iota^* {R}_T^* \delta_{u_0}\}_{T\geq1}$ is tight in $ \mathcal{P}\left(L^1\right)$. Then we have a weak convergence subsequence $\tilde\mu_n:=\iota^* {R}_{T_n}^* \delta_{u_0}\rightarrow \tilde\mu\in \mathcal{P}\left(L^1\right)$.
					
					Notice that $\|\cdot\|_{H^\nn}^2$ is lower semi-continuous in $L^1$. By Portemanteau's theorem, we have
					\baee
					&\int_{L^1(\mathbb{T}^d)}\|x\|_{\nn}^2 \dif\tilde\mu\leq\lim_{n\rightarrow\infty}\int_{L^1(\mathbb{T}^d)}\|x\|_{\nn}^2 \dif\mu_n\\
					=& \lim_{n\rightarrow\infty}\mathbb{E}\Big[\frac {1}{T_n} \int_0^{T_n} \|u_s\|_{\nn}^2\dif s\Big]\leq C<\infty.
					\eaee
					Thus we know that $\tilde \mu(H^\nn)=1$.
					
					Set $ \mu:=\tilde \mu|_{H^\nn}\in\mathcal{P}\left(H^\nn\right) $ be the restriction of $\tilde \mu$ on $H^\nn$.	Let $\varphi \in C_b\left(L^1\right)$.  When we restrict the domain of $\varphi$ to  $H^\nn$, we denote this function by   $\varphi_{\mid_{  H^\nn}}.$
					Since   $\varphi_{\mid_{  H^\nn}}\in C_b(H^\nn)$,  with some abuse of notation, we still write
					$P_t \varphi_{\mid_{  H^\nn}}$ as $P_t \varphi$.
					Moreover, with the help of  Lemma \ref{path L1 contr},  $P_t \varphi$ is continuous \textup{w.r.t.} the $ L^1$ norm. Then
					\baee
					&\int_{L^1(\mathbb{T}^d)} P_t \varphi \mathrm{~d} \widetilde{\mu} =  \lim_{n\rightarrow\infty}\int_{L^1(\mathbb{T}^d)} P_t \varphi \mathrm{~d} \widetilde{\mu}_n \\
					= & \lim _{n \rightarrow \infty} \int_{L^1(\mathbb{T}^d)} P_t \varphi \mathrm{~d}\iota^* {R}_{T_n}^* \delta_{u_0} =  \lim _{n \rightarrow \infty} \int_{H^\nn} P_t \varphi \mathrm{~d} R_{T_n}^* \delta_{u_0} \\
					= & \lim _{n \rightarrow \infty} \frac{1}{T_n} \int_0^{T_n} \int_{H^\nn} \varphi \mathrm{d} P_{s+t}^* \delta_{u_0} \mathrm{~d} s =  \lim _{n \rightarrow \infty} \frac{1}{T_n} \int_t^{T_n+t} \int_{H^\nn} \varphi \mathrm{d} P_s^* \delta_{u_0} \mathrm{~d} s \\
					= & \lim _{n \rightarrow \infty}\left(\frac{1}{T_n} \int_0^{T_n} \int_{H^\nn} \varphi \mathrm{d} P_s^* \delta_{u_0} \mathrm{~d} s\right. \\
					& \left.+\frac{1}{T_n} \int_{T_n}^{T_n+t} \int_{H^\nn} \varphi \mathrm{d} P_s^* \delta_{u_0} \mathrm{~d} s-\frac{1}{T_n} \int_0^t \int_{H^\nn} \varphi \mathrm{d} P_s^* \delta_{u_0} \mathrm{~d} s\right) \\
					= & \lim _{n \rightarrow \infty} \int_{H^\nn} \varphi \mathrm{d} R_{T_n}^* \delta_{u_0} =  \lim _{n \rightarrow \infty} \int_{L^1(\mathbb{T}^d)} \varphi \mathrm{d} \iota^*{R}_{T_n}^* \delta_{u_0}=\int_{L^1(\mathbb{T}^d)} \varphi \mathrm{d} \widetilde{\mu} .
					\eaee
					Note that
					\baee
					&\int_{H^\nn}  \varphi \mathrm{~d}P_t^*{\mu}=\int_{H^\nn} P_t \varphi \mathrm{~d}{\mu}\\
					=&\int_{L^1(\mathbb{T}^d)} P_t \varphi \mathrm{~d} \widetilde{\mu}=\int_{L^1(\mathbb{T}^d)} \varphi \mathrm{d} \widetilde{\mu}=\int_{H^\nn} \varphi \mathrm{d} {\mu}.
					\eaee
					Since $\varphi\in C_b(L^1)$ is arbitrarily, it leads to $\iota^*P_t^*{\mu}=\iota^*\mu$. By an easy modification,
					\cite[Lemma 6 (2)]{MR20} also holds when we replace  the  $\mathbb T$   by $\mathbb T^d$.  Thus,  one has
					$P_t^*{\mu}=\mu.$
				\end{proof}
				\section{ Proof of Corollary \ref{1217-1} }
				\label{B}
				
				First, we begin with a lemma.
				
				\begin{lemma}
					\label{p26-5}
					Assume that
					\begin{eqnarray*}
						A_i(u)=c_{i, \Bbbk}  u^{\Bbbk}+c_{i,\Bbbk-1}u^{\Bbbk-1}
						+c_{i,\Bbbk-2}u^{\Bbbk-2}+\cdots+c_{i,1}u^1+c_{i,0},\quad i=1,\cdots,d.
					\end{eqnarray*}
					where $c_{i,j}\in \mR,  \Bbbk \geq 2$    { and at least one of elements in $\{ c_{i,\Bbbk}:i=1,\cdots,d\}$ is not zero}. If
					\begin{eqnarray*}
						\cZ_0 = \big\{\varsigma_i, -\varsigma_i, 2\varsigma_i,-2\varsigma_i,
						i=1,\cdots,d\big\},
					\end{eqnarray*}
					where $\varsigma_i=(\varsigma_{ij})_{j=1}^d\in \mZ_*^d$ with $\varsigma_{ii}=1$ and $\varsigma_{ij}=0,j\neq i$.
					Then,   one has
					\begin{eqnarray*}
						\cZ_\infty\supseteq \{k\in \mZ^d_*:\langle c_\Bbbk,k\rangle=\sum_{i=1}^d c_{i,\Bbbk} k_i\neq 0 \}.
					\end{eqnarray*}
				\end{lemma}
				
				\begin{proof}
					To shorten the notation, we  always  write $c_{i,\Bbbk}$ as $a_i,i=1,\cdots,d.$
					This means that
					\begin{align*}
						A_i(u)=a_i  u^{\Bbbk}+\sum_{j=0}^{\Bbbk-1}c_{i,j}u^{j},\quad i=1,\cdots,d.
					\end{align*}
					For any $n\geq 1$, by the definition  of   $\cZ_n$, it holds that
					\begin{eqnarray*}
						\cZ_{n}= \big\{\kappa+\ell\in \mZ^d:\kappa=(\kappa_i)_{i=1}^d \in \cZ_{n-1},\ell=(\ell_i)_{i=1}^d \in   \mathbb{L} ,
						\sum_{i=1}^d a_i(\kappa_i+\ell_i)\neq 0 \big\}.
					\end{eqnarray*}
					where
					\begin{align*}
						\mathbb{L}=&\{\ell\in \mZ^d: \ell=\sum_{i=1}^{\Bbbk-1 }\ell^{(i)}, \ell^{(i)}\in \cZ_0,i=1,\cdots , \Bbbk-1 \}.
					\end{align*}

					In the  first,  by direct calculations,  we have   the following claim:
					
					\textbf{Claim:}
					We have
					{\begin{eqnarray*}
							&& \mathbb{L}\supseteq \{ \varsigma_i,-\varsigma_i, 2\varsigma_i,-2\varsigma_i: i=1,\cdots,d\}.
					\end{eqnarray*}}
					If $\Bbbk-1$  is an odd number, in view of
					\begin{equation*}
						\varsigma_i=\frac{\Bbbk}{2}\cdot \varsigma_i+\frac{(\Bbbk-2)}{2}\cdot (-\varsigma_i)\text{ and } 2\varsigma_i=\frac{(\Bbbk-2)}{2}\cdot \varsigma_i+\frac{(\Bbbk-2)}{2}\cdot (-\varsigma_i)+(2 \varsigma_i),
					\end{equation*}
					we obtain $ \varsigma_i, 2 \varsigma_i \in \mathbb{L}.$
					If $\Bbbk-1$  is an even  number, then $\Bbbk\geq 3.$
					In view of
					\begin{eqnarray*}
						\varsigma_i=\frac{\Bbbk-3}{2}\varsigma_i+\frac{\Bbbk-1}{2}(-\varsigma_i)+(2\varsigma_i)
						\text{ and }
						2\varsigma_i=\frac{\Bbbk+1}{2}\varsigma_i+\frac{\Bbbk-3}{2}(-\varsigma_i),
					\end{eqnarray*}
					we get   $\varsigma_i, 2\varsigma_i\in \mathbb{L}.$
					With similar arguments, whether $\Bbbk-1$  is an even  number or not,  we also get  $-\varsigma_i,-2\varsigma_i\in \mathbb{L}.$

					In the following, we will use iteration to prove that, for any $1\leq n\leq d$
					and \baee
					k=(k_1,\cdots,k_n,0,\cdots,0)\in \mZ_*^d\eaee with  $\sum_{i=1}^n a_i k_i\neq 0$,
					one has
					\begin{eqnarray}
						\label{p26-10}
						k=(k_1,\cdots,k_n,0,\cdots,0) =\sum_{i=1}^nk_i\varsigma_i\in \cZ_\infty.
					\end{eqnarray}
					Obviously, the above claim (\ref{p26-10}) holds for $n=1.$
					Assume that we have proved  the above claim  (\ref{p26-10}) for
					$n=\ell\in \mN$ and we proceed with the proof  for $n=\ell+1\leq d.$
					Assume that  $k=(k_1,\cdots,k_\ell,k_{\ell+1},0,\cdots,0)$
					with
					\begin{align}
						\label{p1219-1}
						\sum_{i=1}^{\ell+1} a_i k_i\neq 0.
					\end{align}
					If $k_{\ell+1}=0$, then  by iteration, it holds that   \begin{equation*}
						k=(k_1,\cdots,k_\ell,k_{\ell+1},0,\cdots,0)
						=(k_1,\cdots,k_\ell,0,0,\cdots,0)\in \cZ_\infty.
					\end{equation*}
					If
					$a_{\ell+1}=0$, first   by the   iteration and (\ref{p1219-1}), we have
					$(k_1,\cdots,k_\ell,0,0,\cdots,0)\in \cZ_\infty.$
					Then, by the fact  $\varsigma_{\ell+1},-\varsigma_{\ell+1} \in \cZ_0\cap \mathbb{L}$ and the definitions  of $\cZ_\infty$
					we conclude that  \baee(k_1,\cdots,k_\ell,k_{\ell+1},0,\cdots,0) \in \cZ_\infty.\eaee
					
					Therefore, we can  assume  that  $k_{\ell+1}a_{\ell+1}\neq  0.$
					Furthermore, we also assume that  $k_{\ell+1}>0.$
					For the case of  $k_{\ell+1}<0$, the proof is similar and we omit the details.
					There are the following   two cases about $(k_1,\cdots,k_\ell).$

					\textbf{Case 1: $\sum_{i=1}^{\ell} a_i k_i=0.$}
					In this case, we can furthermore assume that   at least one of $a_i,i=1,\cdots,\ell$ is not equal  $0$\footnote{If
						$a_i=0,\forall 1\leq i\leq \ell$,
						by the fact $\varsigma_{\ell+1},-\varsigma_{\ell+1} \in \cZ_0\cap \mathbb{L}$ and the definition of $\cZ_n$,
						one arrives at $k_{\ell+1}\varsigma_{\ell+1} \in \cZ_\infty.$
						Then, by the fact $a_{\ell+1}k_{\ell+1}\neq 0,a_1=0$ and $ \varsigma_{1},-\varsigma_{1} \in \cZ_0\cap \mathbb{L}$,
						it holds that $k_1\varsigma_1+k_{\ell+1}\varsigma_{\ell+1}\in \cZ_\infty.$
						With similar arguments, one also gets   $\sum_{i=1}^{2}k_i\varsigma_i+a_{\ell+1}\varsigma_{\ell+1}\in \cZ_\infty$
						and finally arrives at $\sum_{i=1}^{\ell+1}k_i\varsigma_i\in \cZ_\infty.$
					}. Without loss of generality, we   assume that $a_1\neq 0.$
					Therefore, it holds that    $a_1(k_1-1)+\sum_{i=2}^{\ell} a_i k_i\neq 0$
					and $a_1(k_1+1)+\sum_{i=2}^{\ell} a_i k_i\neq 0.$
					Obviously, $\frac{a_1}{a_{\ell+1}}\in \{1,2,\cdots,k_{\ell+1}\}$ and  $-\frac{a_1}{a_{\ell+1}}\in \{1,2,\cdots,k_{\ell+1}\}$
					can't  hold  simultaneously.

					If $\frac{a_1}{a_{\ell+1}}\notin \{1,2,\cdots,k_{\ell+1}\}$,
					then for any $j\in \{1,2,\cdots,k_{\ell+1}\}$,
					it holds that
					\begin{eqnarray}
						\label{pp10-2}
						a_1(k_1-1)+\sum_{i=2}^{\ell} a_i k_i+a_{\ell+1}j\neq 0.
					\end{eqnarray}
					In the above, we have used the fact $\sum_{i=1}^{\ell} a_i k_i=0.$
					Noticing that $a_1(k_1-1)+\sum_{i=2}^{\ell} a_i k_i\neq 0, $
					by iteration, (\ref{pp10-2}) and the definition of $\cZ_\infty$, it holds that
					\begin{eqnarray}
						\label{pp10-3}
						(k_1-1,k_2,\cdots,k_\ell,k_{\ell+1},0,\cdots,0) \in \cZ_\infty.
					\end{eqnarray}
					Furthermore, by the definitions of $\cZ_\infty$ and the fact $\varsigma_1\in \mathbb{L}\cap \cZ_0$,  we also have
					\begin{equation*} (k_1,k_2,\cdots,k_\ell,k_{\ell+1},0,\cdots,0) \in \cZ_\infty.\end{equation*}
					
					If $-\frac{a_1}{a_{\ell+1}}\notin \{1,2,\cdots,k_{\ell+1}\}$,
					it holds that
					\begin{eqnarray*}
						a_1(k_1+1)+\sum_{i=2}^{\ell} a_i k_i+a_{\ell+1}j\neq 0,\quad \forall j\in \{1,2,\cdots,k_{\ell+1}\}.
					\end{eqnarray*}
					In the above, we have used the fact $\sum_{i=1}^{\ell} a_i k_i=0.$
					Similar to (\ref{pp10-3}), it holds that  $ (k_1+1,k_2,\cdots,k_\ell,k_{\ell+1},0,\cdots,0) \in \cZ_\infty.$
					Furthermore, by the definitions of $\cZ_\infty$ and the fact $\varsigma_1,-\varsigma_1\in \mathbb{L}\cap \cZ_0$,  we also have
					$ (k_1,\cdots,k_\ell, k_{\ell+1},0,\cdots,0) \in \cZ_\infty.$

					\textbf{Case 2: $\sum_{i=1}^{\ell} a_i k_i\neq 0.$}
					We divide the following two subcases about $a_{\ell+1}.$
					
					Subcase 2.1: $a_{\ell+1}\notin \{-\sum_{i=1}^\ell a_ik_i/j, j=1,\cdots,k_{\ell+1}\}.$
					In this subcase, for any  $j\in \{1,\cdots,k_{\ell+1}\}$,
					we have
					\begin{eqnarray*}
						\sum_{i=1}^\ell a_ik_i+ a_{\ell+1} j \neq 0.
					\end{eqnarray*}
					Thus, by  iteration,  the definition of $\cZ_\infty$  and the fact $\varsigma_{\ell+1}\in \mathbb{L}\cap \cZ_0$, one easily sees that
					$ (k_1,k_2,\cdots,k_\ell,k_{\ell+1},0,\cdots,0) \in \cZ_\infty.$

					Subcase 2.2: $a_{\ell+1}\in \{-\sum_{i=1}^\ell a_ik_i/j, j=1,\cdots,k_{\ell+1}\}.$
					In this subcase, by  (\ref{p1219-1}), for some   $j_0\in \{1,2,\cdots,k_{\ell+1}-2,k_{\ell+1}-1\}$,
					we have
					\begin{eqnarray}
						\label{p29-2}
						\sum_{i=1}^\ell a_ik_i+j_0 a_{\ell+1}= 0 \text{ and }\sum_{i=1}^\ell a_ik_i+j a_{\ell+1}\neq 0,\forall j\neq j_0.
					\end{eqnarray}
					First, by the definition of $\cZ_\infty$ and the fact $\varsigma_{\ell+1}\in \mathbb{L}\cap \cZ_0$,  one easily sees that
					\begin{eqnarray}
						\label{p29-1}
						(k_1,k_2,\cdots,k_\ell,j_0-1,0,\cdots,0) \in \cZ_\infty.
					\end{eqnarray}
					Since $2\varsigma_{\ell+1}\in \mathbb{L}$, by the above fact and (\ref{p29-2}),  we also   get
					\begin{eqnarray*}
						(k_1,k_2,\cdots,k_\ell,j_0+1,0,\cdots,0) \in \cZ_\infty.
					\end{eqnarray*}
					Also by the  definitions of $\cZ_\infty$  and (\ref{p29-2}),   we arrive at
					\begin{eqnarray*}
						(k_1,k_2,\cdots,k_\ell,k_{\ell+1},0,\cdots,0) \in \cZ_\infty.
					\end{eqnarray*}

					We complete the proof this Lemma  by iteration.
				\end{proof}

				Now we are in a position  to prove   Corollary \ref{1217-1}
				based on Lemma \ref{p26-5}.

				\begin{proof}
					(\romannumeral1)
					By Lemma \ref{p26-5},
					one has
					\begin{eqnarray*}
						\cZ_\infty\supseteq \{k\in\mZ^d:\langle c_\Bbbk,k\rangle\neq 0 \}.
					\end{eqnarray*}
					On the other hand,  obviously, we have
					\begin{eqnarray*}
						A^\perp=\{k\in\mZ^d:\langle c_\Bbbk,k\rangle= 0 \}.
					\end{eqnarray*}
					Thus, $\cZ_\infty^c \subseteq   A^\perp$ and the Condition \ref{16-5} holds.

					(\romannumeral2)
					In this case,
					by Lemma \ref{p26-5} and (\ref{pp10-4}),
					it holds that
					\begin{eqnarray*}
						\cZ_\infty^c \subseteq\{k\in\mZ^d:\langle c_\Bbbk,k\rangle=0 \}
						=\{\mathbf{0}\}\subseteq A^\perp.
					\end{eqnarray*}
					Thus,  the Condition \ref{16-5} holds.

					(\romannumeral3)
					We only need to consider the case $\Bbbk\geq 2.$
					By Lemma \ref{p26-5},
					it holds that
					\begin{eqnarray*}
						\cZ_\infty \supseteq \big\{k_1\in \mZ : c_{1,\Bbbk} k_1\neq 0 \big\}
						=\big\{k_1:  k_1\neq 0\big\}.
					\end{eqnarray*}
					and
					\begin{eqnarray*}
						\cZ_\infty^c
						\subseteq \big\{k_1 :   k_1=0\big\}\subseteq A^\perp.
					\end{eqnarray*}
					Thus, the proof is complete.
					
				\end{proof}
	\end{appendices}
\bmhead{Acknowledgements}

	We  would like to thank Professor  Zhao Dong and Doctor Zhengqian Li  for  their  useful  discussions and suggestions.
\section*{Declarations}

\begin{itemize}
\item Funding

	This work  is supported
by National Key R and D Program of China
(No. 2020YFA0712700), National Natural Science Foundation of China (Nos. 12090010, 12090014, 12471138),  the science and technology innovation Program of Hunan Province (No. 2022RC1189)    and Key Laboratory of Random Complex Structures and Data Science, Academy of Mathematics and
Systems Science, Chinese Academy of Sciences (Grant No. 2008DP173182).

\end{itemize}



\bibliography{sn-bibliography}

\end{document}